\documentclass[12pt]{article}

\usepackage[algo2e,boxed]{algorithm2e}
\usepackage{amsmath}
\usepackage{amssymb}
\usepackage{amsfonts}
\usepackage{amsthm}

\usepackage{authblk}

\usepackage{mathrsfs}
\usepackage[new]{old-arrows}
\usepackage{bm}
\usepackage[sans]{dsfont}
\usepackage{enumerate}
\usepackage{graphicx}
\usepackage{pstricks}
\usepackage{subfigure}
\usepackage{hyperref}
\usepackage{float}

\usepackage[a4paper, top=1in, bottom=1in, left=1in, right=1in]{geometry}
\usepackage{xcolor}

\newcommand{\Dim}{{\scriptsize \textsf{D}}}

\newcommand{\Span}{\mathrm{span}}

\newcommand{\unitV}{\mathds{1}}

\newcommand{\bmi}{\mathbf{i}}
\newcommand{\bmj}{\mathbf{j}}
\newcommand{\bmn}{\mathbf{n}}

\newtheorem{thm}{Theorem}[section]
\newtheorem{lem}[thm]{Lemma}
\newtheorem{coro}[thm]{Corollary}
\newtheorem{defn}[thm]{Definition}
\newtheorem{exm}[thm]{Example}

\newtheorem{ntn}{Notation}

\newcommand{\revise}[1]{#1}

\title{
  An AI-aided algorithm for multivariate polynomial reconstruction
  on Cartesian grids
  and the PLG finite difference method
}

\author[1,2]{Qinghai Zhang\thanks{Corresponding author: qinghai@zju.edu.cn}}
\author[1]{Yuke Zhu}
\author[1]{Zhixuan Li}

\affil[1]{School of Mathematical Sciences,
  Zhejiang University,
  866 YuHangTang Road,
  Hangzhou, Zhejiang Province, 310058, China}
\affil[2]{Shanghai Institute for Advanced Study of 
  Zhejiang University, Shanghai AI Laboratory, Shanghai, 200000, China}

\date{}

\begin{document}

\maketitle

\begin{abstract}
  Polynomial reconstruction on Cartesian grids 
 \revise{is fundamental
 in many scientific and engineering applications}, 
 yet it is still an open problem
 how to construct for a finite subset $K$ of $\mathbb{Z}^{\Dim}$
 a lattice ${\cal T}\subset K$
 so that multivariate polynomial interpolation on this lattice
 is unisolvent.
In this work, we solve this open problem
 of poised lattice generation (PLG)
 via an interdisciplinary research
 of approximation theory, abstract algebra,
 and artificial intelligence. 
Specifically, 
 we focus on the triangular lattices in approximation theory, 
 study group actions of permutations upon triangular lattices,  
 prove an isomorphism between the group of permutations
 and that of triangular lattices,
 and dynamically organize the state space of permutations
 so that a depth-first search of poised lattices
 has optimal efficiency. 
Based on this algorithm,
 we further develop the PLG finite difference method
 that retains the simplicity of Cartesian grids
 yet overcomes the disadvantage of legacy finite difference methods
 in handling irregular geometries.
Results of various numerical tests
 demonstrate the effectiveness of our algorithm
 and the simplicity, flexibility, efficiency,
 and fourth-order accuracy of
 the PLG finite difference method.


\end{abstract}

\section{Introduction}
\label{sec:introduction}

In numerically solving partial differential equations (PDEs), 
 the simplest and historically oldest method is probably
 the finite difference (FD) method, 
 which mainly consists of replacing spatial derivatives in the PDE
 with FD formulas
 and solving the resulting system of 
 ordinary differential equations (ODEs)
 by a time integrator.
Supported by an extensive body of theory,
 FD methods are robust, efficient, and accurate
 for a variety of PDEs \cite{strikwerda89:_finit_differ_schem_partial_differ_equat,leveque07:_finit_differ_method_ordin_partial_differ_equat,li18:_numer_solut_differ_equat_introd}.

Legacy FD methods are often criticized
 as ill-suited for complex geometries. 
Indeed, 
 FD methods rely heavily
 on the geometric regularity of the underlying Cartesian grids, 
 and multidimensional FD formulas
 are usually based on tensor products
 of one-dimensional formulas.
These regularities of FD stencils
 severely limit the application of FD methods 
 to irregular domains.

\subsection{Motivations from 
  solving PDEs on irregular domains with Cartesian grids}
\label{sec:motivations}

There are two approaches 
 that could alleviate the unfitness of legacy FD methods
 for irregular domains while retaining
 the simplicity of Cartesian grids. 
 
In the first approach,
 classical FD stencils are retained for \emph{point} values
 at grid points in the interior of the domain 
 while special treatments are adopted
 at grid points near the irregular boundary. 
Successful examples of this approach include
 the immersed boundary method
 \cite{peskin77:_numer,mittal05:_immer_bound_method}, 
 the immersed interface method
 \cite{leveque94:_immer,li06:_immer_inter_method}, 
 the ghost cell approach
 \cite{mayo84:_fast_solut_of_poiss_and,tseng03:_ghost_cell_immer_bound_method,liu03:_ghost,zhang10:_handl_solid_fluid_inter_for_viscous_flows,xu23:_ghost},
 and many others such as 
 the MIB method \cite{zhou06:_high_order_match_inter_and}
 and the KFBI method \cite{xie20}. 

The second approach concerns a finite volume (FV) formulation, 
 where the unknowns are \emph{averaged} values
 over the rectangular cells of the Cartesian grid.
Similar to the first approach,
 spatial discretization away from the irregular boundary
 is based on symmetric stencils 
 while cells near the boundary are treated differently.
In both approaches,
 this strategy of separately treating
 `regular' grids and `irregular' grids 
 can be made very efficient 
 by exploiting the fact that
 the irregular grids requiring more involved treatments
 form a set of codimension one.

A popular example of the second approach
 is the cut-cell method, also known as the embedded boundary (EB) method,
 in which cells close to the boundary
 are cut by the irregular interface
 and averaged values are defined
 on the open region, i.e., the intersection
 of the cell and the problem domain;
 see, e.g., 
 \cite{johansen98:_cartes_grid_embed_bound_method,tucker00:_cartes,ingram03:_devel_cartes,devendran17:_cartes_poiss,overton-katz23:_stokes}
 and references therein.
Previous EB methods are second-order accurate 
 and it is only recently that fourth-order EB methods
 emerged \cite{devendran17:_cartes_poiss,overton-katz23:_stokes}. 
In these methods, FV formulas are derived
 for regular cells based on Taylor expansions
 while multivariate polynomials are fitted
 to interpolate cell-averaged values
 near the irregular boundary in a least-square sense.
Then spatial operators are discretized
 via integrating the derivatives of the fitted polynomials
 over irregular cells.
 
This work is motivated
 by the following observations and questions.
 \begin{enumerate}[(Q.1)]\itemsep0em
 \item For irregular domains,
   most FD-based methods are only first- or second-order accurate.
   Can we develop, for arbitrarily complex
   topology and geometry,
   a new FD method with fourth- or higher-order accuracy?
 \item Treatments of the irregular boundary in most FD-based methods
   revolve around modifying one-dimensional FD formulas.
   As such, it is not clear how to generalize
   them to PDEs with cross derivatives
   such as $\frac{\partial^2 u}{\partial x\partial y}$.
   Can we design a new FD method
   whose formulation is completely decoupled from
   the specific form of the PDE? 
 \item \revise{In current EB methods,
     a multivariate polynomial is fitted for each cell
     near the irregular boundary
     from averaged values over a stencil of nearby cells,
     the cardinality of which is more than necessary. 
     However, it has never been rigorously proven that
     this polynomial reconstruction
     indeed admits a unique solution.}
   In fact, even a precise description of which cells
   get selected into the stencil is rare. 
   Of course, one can keep adding nearby cells
   until the linear system becomes uniquely solvable
   in the sense of least squares.
   But blindly expanding the stencil
   may lead to a large number of redundant cells,
   deteriorating computational efficiency.
   As will be shown in Figure \ref{fig:TLG_conditioning}, 
   it may also increase the condition number
   of the linear system by a large factor.
   \revise{Therefore, for the multivariate polynomial fitting
   at a cell near the irregular boundary, 
   can we give an explicit and systematic construction
   of the stencil so that
   its number of cells is absolutely the \emph{minimum}? %
   Can we prove the unique solvability? 
   Can we exert fine control over the condition number
   of least squares?}
 \item The algorithmic complexity of EB methods
   and many FD-based methods 
   increases rapidly as the dimensionality 
   goes from two to a higher one.
   Can we suppress the algorithmic complexity of the new FD method 
   so that the special treatment remains the same
   in two and higher dimensions?
 \end{enumerate}

In this paper, we give positive answers to all above questions, 
 via an intensive study on data reconstruction 
 by multivariate polynomial interpolation.

\begin{defn}[Lagrange interpolation problem (LIP) \cite{carnicer06:_inter}]
  \label{def:LagrangeProblem}
  Let $\Pi^{\Dim}$ denote the linear space of all $\Dim$-variate
  polynomials with real coefficients.
  Given a subspace $V$ of $\Pi^{\Dim}$,
  a finite number of points
  \mbox{$\mathbf{x}_1, \mathbf{x}_2, \ldots, \mathbf{x}_N \in
   \mathbb{R}^{\Dim}$}
  and the same number of data
  $f_1, f_2, \ldots, f_N \in \mathbb{R}$,
  the \emph{Lagrange interpolation problem}
  seeks a polynomial $f\in V$ such that
  \begin{equation}
    \label{eq:LagrangeProblem}
    \forall j=1,2,\ldots,N,\ 
    f (\mathbf{x}_j) = f_j,
  \end{equation}
  where $V$ is called the \emph{interpolation space}
  and $\mathbf{x}_j$'s the \emph{interpolation sites}. 
\end{defn}

A LIP with $\dim V = N$
 is said to be \emph{unisolvent}
 if for \emph{any} given data $(f_j)_{j=1}^N$
 there exists $f\in V$
 satisfying (\ref{eq:LagrangeProblem});
in this case we also say that the sites $(\mathbf{x}_j)_{j=1}^N$
 are \emph{poised} 
 in $V$ or
 they are {poised} with respect to a basis of $V$.

All our answers to (Q.1--4) are based on
 a novel and efficient solution to the open problem as follows.

\subsection{The open problem of poised lattice generation (PLG) in $\mathbb{Z}^{\Dim}$}
\label{sec:PLG}
 
\begin{defn}[PLG in $\mathbb{Z}^{\Dim}$]
  \label{def:PLG}
  Given a finite set of feasible nodes $K\subset\mathbb{Z}^{\Dim}$, 
  a starting point $\mathbf{q}\in K$, and a degree $n\in \mathbb{Z}^+$, 
  the \emph{problem of poised lattice generation}
  is to choose a lattice ${\mathcal T}\subset K$
  such that $\mathbf{q}\in {\mathcal T}$ and ${\mathcal T}$ is poised in $\Pi^{\Dim}_n$
  with $\#{\mathcal T} = \dim \Pi^{\Dim}_n = {\Dim+n \choose n}$,
  where $\Pi^{\Dim}_n$ is 
  the space of $\Dim$-variate polynomials
  of total degree at most $n$.
\end{defn}


In this work, we limit values of $n$
 in Definition \ref{def:PLG} to $n\le 6$;
 see Section \ref{sec:nature-pois-latt} for reasons of this limitation.
Two examples of PLG are shown in Figure \ref{fig:selectingPoisedLattices}.
The starting point $\mathbf{q}$ corresponds to the grid point
 at which the spatial operators are discretized, 
 and the feasible set $K$ can be customized
 according to the physics of the PDE.
As illustrated in Figure \ref{fig:select_box}, 
 for the Laplacian operator one might want to choose $K$
 so that $\mathbf{q}$ is centered in $K$
 to reflect the isotropic nature of diffusion.
In contrast,
 for the advection operator it is more appropriate to choose $K$
 to favor the upwind direction so that
 the domain of dependence of the advection is covered. 
Note that the problem in Definition \ref{def:PLG}
 might not have a solution; 
 in this case one can either shift/expand the feasible set $K$
 or decrease the degree $n$ until the problem is solvable. 
 
\begin{figure}
  \centering
  \subfigure[A poised lattice in $\Pi^2_5$ at the hollow dot $\mathbf{q}$.
  ]{
    \includegraphics[width=0.37\textwidth]{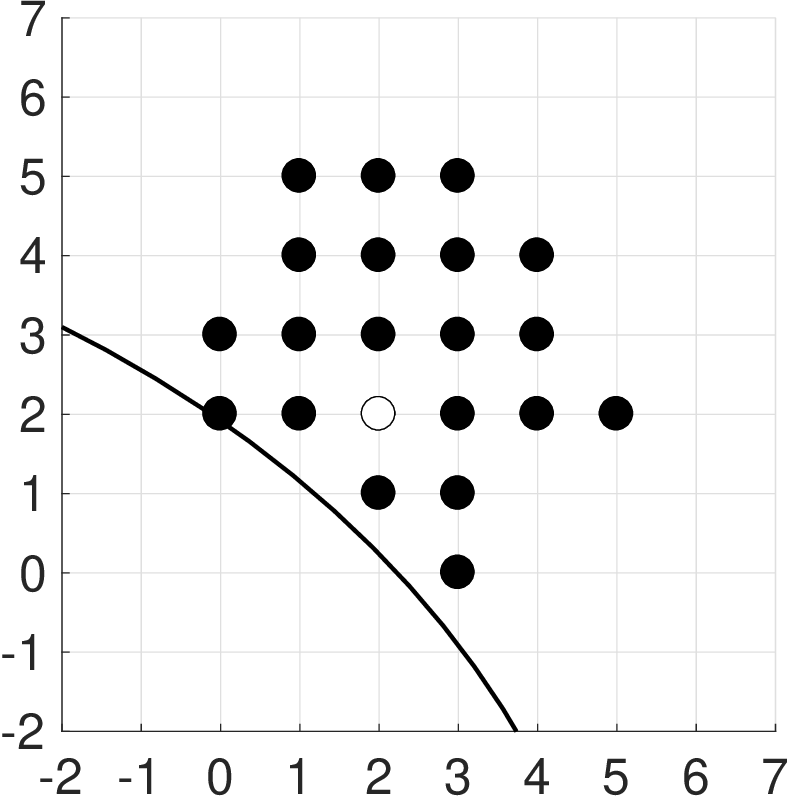}
  }
  \hfill
  \subfigure[A poised lattice in $\Pi^3_3$ at the yellow ball $\mathbf{q}$.
  ]{
    \includegraphics[width=0.58\textwidth]{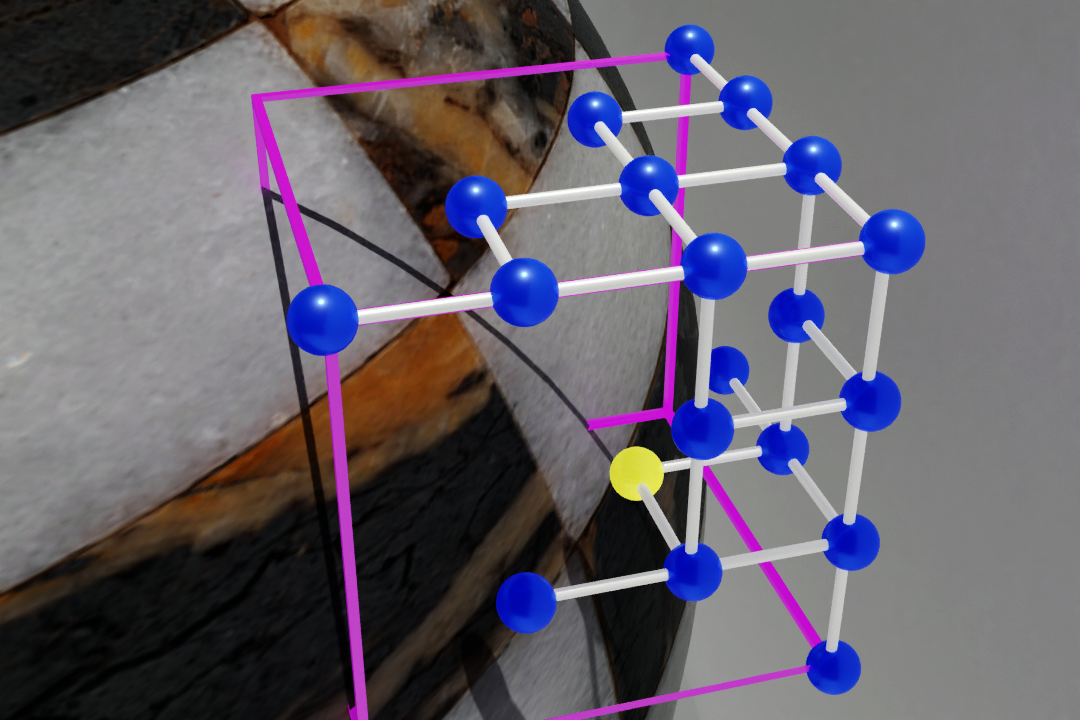}}
  
  \caption{Two solutions of the PLG problem, 
    where the feasible set $K$ is obvious.
  }
  \label{fig:selectingPoisedLattices}
\end{figure}

Suppose that a LIP
 is unisolvent with its interpolation space $V$ spanned
 by basis functions $\phi_1, \phi_2, \ldots, \phi_N$.
Then the \emph{sample matrix}
 \begin{equation}
   \label{eq:sampleMatrix}
   M =
   \begin{bmatrix}
     \phi_1(\mathbf{x}_1) & \phi_2(\mathbf{x}_1) & \cdots &
     \phi_N(\mathbf{x}_1)
     \\
     \phi_1(\mathbf{x}_2) & \phi_2(\mathbf{x}_2) & \cdots &
     \phi_N(\mathbf{x}_2)
     \\
     \vdots & \vdots & \ddots & \vdots 
     \\
     \phi_1(\mathbf{x}_N) & \phi_2(\mathbf{x}_N) & \cdots &
     \phi_N(\mathbf{x}_N)
   \end{bmatrix}
 \end{equation}
 must be non-singular.
Indeed, $\det M=0$ would imply
that the dimension of the range of $M$ be less than $N$
and that there exist some data  $(f_j)_{j=1}^N$
for which the LIP has no solution.
Therefore,
the nonsingularity of the sample matrix,
the existence of a unique solution of the LIP,
and the poisedness of the interpolation sites
are equivalent conditions.
 
In the familiar case of $\Dim=1$,
 the LIP is unisolvent if and only if
 its sites are pairwise distinct.
In contrast, for $\Dim>1$, 
 it is much more difficult to decide
 whether a set of sites is poised in $\Pi^{\Dim}_n$.
For example,
suppose six data are given
at the sites
\begin{displaymath}
 (5,0),\ (-5,0),\ (0,5),\ (0,-5),\ (4,3),\ (-3,4).
\end{displaymath}
Then for the space $V=\Pi^2_2=\Span(1,x,y,x^2,y^2,xy)$, 
 the sample matrix
 \begin{displaymath}
   M=
   \begin{bmatrix}
     1 & 5 & 0 & 25 & 0 & 0
     \\
     1 & -5 & 0 & 25 & 0 & 0
     \\
     1 & 0 & 5 & 0 & 25 & 0
     \\
     1 & 0 & -5 & 0 & 25 & 0
     \\
     1 & 4 & 3 & 16 & 9 & 12
     \\
     1 & -3 & 4 & 9 & 16 & -12
   \end{bmatrix}
 \end{displaymath}
 is singular since its first, fourth, and fifth columns are
 linearly dependent;
 indeed, all six sites are on the circle $x^2+y^2-25=0$.
As a crucial difference between univariate and multivariate polynomials,
 the latter usually vanishes on an \emph{infinite} number of points.
In general, a set of sites is poised
 if and only if the interpolation space does not contain a polynomial
 that vanishes on all the sites.
This underlines a core difficulty of multivariate interpolation:
 \emph{the poisedness of interpolation sites depends on their geometric configuration}.

In approximation theory, 
 much effort has been directed to constructing poised lattices. 
In the classical paper by Chung and Yao \cite{chung77:_lagran},
a set $X$ of $N:={\Dim+n \choose n}$ sites in $\mathbb{R}^{\Dim}$ is said
to satisfy the \emph{condition of geometric characterization} in
$\Pi^{\Dim}_n$ 
if for each site $\mathbf{x}_i$
there exist $n$ distinct hyperplanes
$G_{i,1}$, $G_{i,2}$, $\ldots$, $G_{i,n}$
such that
\begin{equation}
  \label{eq:conditionGC}
  \forall i,j=1,2,\ldots,N, \qquad
  i\ne j \ \Leftrightarrow\
  \mathbf{x}_j\in \cup_{\ell=1}^{n} G_{i,\ell},
\end{equation}
i.e., each $\mathbf{x}_i$ does not lie on any of these hyperplanes
and all the other nodes in $X$ lie on at least one of these hyperplanes.
By constructing Lagrange fundamental polynomials
that are 1 at one site and 0 at all other sites,
 they showed that the condition of geometric characterization
 guarantees a unique interpolating polynomial of degree at most $n$.
Chung and Yao \cite{chung77:_lagran} also proposed
 two special sets of poised sites,
 called natural lattices and principal lattices
 (in this work a lattice refers to
 a set of interpolation sites whose cardinality equals
 the dimension of the space of interpolating polynomials),
 both of which satisfy the condition of geometric characterization.
Following
 this idea of choosing intersections of appropriate hyperplanes,
 researchers have proposed other types of poised lattices
 such as $(\Dim+1)$-pencil lattices \cite{lee91:_const_lagran,jaklic10:_lattic},
 (fully) generalized principal lattices \cite{carnicer06:_inter,carnicer06:_geomet,boor09:_multiv},
 Aitken-Neville sets
 \cite{sauer02:_aitken_nevil,carnicer09:_aitken_nevil,boor09:_multiv},
 and lower sets \cite{werner80:_remar_newton,sauer03:_lagran,dyn14:_multiv}.
See \cite{gasca00:_polyn} for a review.

Unfortunately,
the freedom of choosing \emph{any} points in $\mathbb{R}^{\Dim}$
is always assumed in the aforementioned methods; 
but this assumption in $\mathbb{R}^{\Dim}$ does not hold for
the PLG problem in Definition \ref{def:PLG}. 

To our best knowledge,
 there are no systematic solutions
 to Definition \ref{def:PLG},
 and thus PLG is still an open problem.

\subsection{The nature of the PLG problem}
\label{sec:nature-pois-latt}

It is now a good time to discuss several key issues
 on the formulation of PLG in $\mathbb{Z}^{\Dim}$
 to reveal its nature
 and to dispel potential doubts on Definition \ref{def:PLG}.

First, we explain
 why values of $n$ are limited 
 to $n\le 6$ in Definition \ref{def:PLG}. 
 
By the Runge phenomenon, 
 polynomial interpolation on uniform grids may diverge 
 as the polynomial degree $n\rightarrow \infty$. 
Indeed, 
 the condition number of the Vandermonde matrix
 (\ref{eq:sampleMatrix}) on uniform grids
 grows exponentially \cite[p. 24]{gautschi12:_numer_analy}: 
 \begin{equation}
   \label{eq:cond8Vandermonde}
   \text{cond}_{\infty} M \sim \frac{1}{\pi}
   \exp\left[\frac{n}{4}(\pi+2\ln2)-\frac{\pi}{4}\right]. 
 \end{equation}

\revise{Consider univariate polynomial interpolation
 as a linear projection $L_p$ that sends a continuous function
 $f\in {\cal C}([-1,+1])$ to
 a polynomial $p\in \Pi_n^1$ 
 so that the interpolation condition 
 $p(x_i)=f(x_i)$ holds
 on a sequence of interpolation sites $(x_i)_{i=0}^n$.
 Its \emph{Lebesgue constant} is defined 
 as $\Lambda:=\sup_{f\in {\cal C}([-1,1])}
 \frac{\|p\|_{\infty}}{\|f\|_{\infty}}$. 
Then it is not difficult to show
 \cite[Thm 15.1]{trefethen17:_approx_theor_and_approx_pract}
 \begin{equation}
   \label{eq:nearBestApprox}
   \|p-f\|_{\infty} \le (1+\Lambda)\|p^*-f\|_{\infty}, 
 \end{equation}
 where $p^*$ is the best approximation of $f$ in $\Pi_n^1$
 based on the max-norm.

It has long been known  \cite{turetskii40,schonhage61:_fehler_inter}
 \cite[Thm 15.2]{trefethen17:_approx_theor_and_approx_pract}
that the Lebesgue constant on uniform grids
 can be estimated by 
 \begin{equation}
   \label{eq:LebesgueConstant}
   \Lambda(n)\sim \frac{2^{n+1}}{e n \log n}.
 \end{equation}
Hence it is not a good idea to seek the solution 
 of PLG in $\mathbb{Z}^{\Dim}$ for a large $n$.
}

But how large is large?
For the Vandermonde matrix,
 we have from (\ref{eq:cond8Vandermonde}) that
$\text{cond}_{\infty} M \approx 160, 8700, 1.0\times 10^9$
for $n=6, 10, 20$, respectively. 
As for the Lebesgue constant,
 we have \revise{from (\ref{eq:LebesgueConstant})
 $\Lambda \approx 16, 132$, and $9.8\times 10^{12}$}
 for $n$=7, 11, and 50, respectively. 
Of course these results are for one dimension,  
 but they serve as a strong heuristic
 that multivariate polynomial interpolation on uniforms grids is fine
 if we impose a small upper bound on the total degree, 
 say, $n\le 6$.
This is confirmed by results of numerical experiments
 in Section \ref{sec:tests}. 

Second, 
 we emphasize that the PLG problem, 
 even with limited small values of $n$,
 is not trivial at all.
In particular, brute-force algorithms do not work.
Consider solving PLG for $n=4$ and $\Dim=3$
 by exhaustive enumeration in a cube. 
The number of possibilities
 of choosing $N={n+\Dim \choose n}$ points,
 c.f. Lemma \ref{lem:cardTriangularLattice},
 from the cube of $(n+1)^{\Dim}$ points 
 is already $1.2\times 10^{31}$!
This number grows to $4.2\times 10^{81}$ for $n=6$ and $\Dim=3$!
It is very challenging to \emph{efficiently} find a poised lattice
 out of such an enormous number of candidates.
After all, the PLG problem has to be solved
 for each irregular cell
 and the total cost of PLG should be a small fraction
 of that of solving the discretized equations.
 
Finally,
 a PLG algorithm 
 must be able to handle all possible
 geometric configurations of an irregular boundary;
 the feasible set $K$ in the signature of PLG
 serves this purpose. 
Since phyiscally meaningful regions with arbitrarily complex
 topology and geometry can be accurately and efficiently represented
 by Yin sets \cite{zhang20:_boolean}, 
 we insist that the algorithm
 operates in a dimension-agnostic manner
 so that there is no need to switch gears
 for different dimensions.
 
Given the above discussions on the nature of PLG,
 we believe that tools in traditional approximation theory
 are insufficient to satisfactorily solve the PLG problem.
A coupling to search algorithms in artificial intelligence (AI)
 turns out to be helpful.

\subsection{A fusion of approximation theory,
  group theory, and AI search algorithms}
\label{sec:search-algor-AI}

A \emph{search algorithm} is an AI algorithm for finding the best solution
 of a \emph{search problem}
 that is specified by the \emph{initial state},
 the \emph{goal state}, and a \emph{state space} or \emph{solution space}. 
A search algorithm solves the search problem
 by first organizing the state space
 into a special graph (typically a spanning tree)
 and then traversing the graph from the initial state
 by repeatedly determining the next best move
 according to some heuristic or deductive strategy.
A \emph{solution}
  is a path from the initial state to the goal state that
  minimizes a certain cost function.

\revise{
Typical applications of search algorithms
 include pathfinding, optimization, and game playing.
Figure \ref{fig:shortestPath} illustrates
 how a shortest-path problem is solved
 by a search algorithm, where
 the spanning tree is usually \emph{exhaustive}
 in that every acyclic path from the start to the end
 is contained in the spanning tree.
Visiting one city at a time
 appears to be the \emph{only} appropriate way
 to construct the state space,
 leading to a \emph{static} structure
 that works very well for minimizing
 the total length of the traversed path. 
}

\revise{
  Game playing is another grand triumph of search algorithms.
  The research on playing games by computer programs traces back
  to a paper in 1950, in which Claude Shannon \cite{shannon50:_progr}
  proposed a framework that combines the structure of a spanning tree 
  with an optimization problem called ``minimaxing.''
  But it was not until 1997
  that a human world champion of chess was 
  defeated by an AI computer ``Deep Blue''
  from IBM\texttrademark. 
  In 2016 and 2017,
  two human world champions of Go
  were defeated by ``AlphaGo,'' an AI computer
  developed by Google\texttrademark\ DeepMind. 
  During these seven decades,
  search algorithms had always been a fundamental tool
  in developing AI game players. 

  The state space of a game contains roughly $b^d$ states
  where $d$ is the game depth
  and $b$ the game breath (i.e., number of legal moves per position).
  Their values are $(b,d)\approx (35,80)$ for chess
  and $(b,d)\approx (250,150)$ for Go.
  Since only one move is allowed at a time, 
  the state space of game playing
  has the same static structure as that of a shortest-path problem. 
  However, there is a key difference:
  \emph{exhaustive search is infeasible for game playing}.
  Consequently,
  much effort has been focused on the design of
  adversarial search algorithms
  \cite{knuth75:_alpha_beta_pruning}
  \cite[chap. 5]{russell21:_artif_intel} 
  that efficiently reduce the complexity
  of the actual state space
  by pruning away subtrees of unfavorable outcome. 
  For this purpose,
  AlphaGo is a breakthrough \cite{silver16:_AlphaGo}
  in that it combines an adversarial search algorithm
  and deep neural networks 
  to successfully handle the previously unmanageable
  complexity of $(b,d)\approx (250,150)$. 
}

 
\begin{figure}
\centering
\subfigure[A map of a local region in east China.
An integer represents the railway distance between two cities.]{
\includegraphics[width=.42\linewidth]{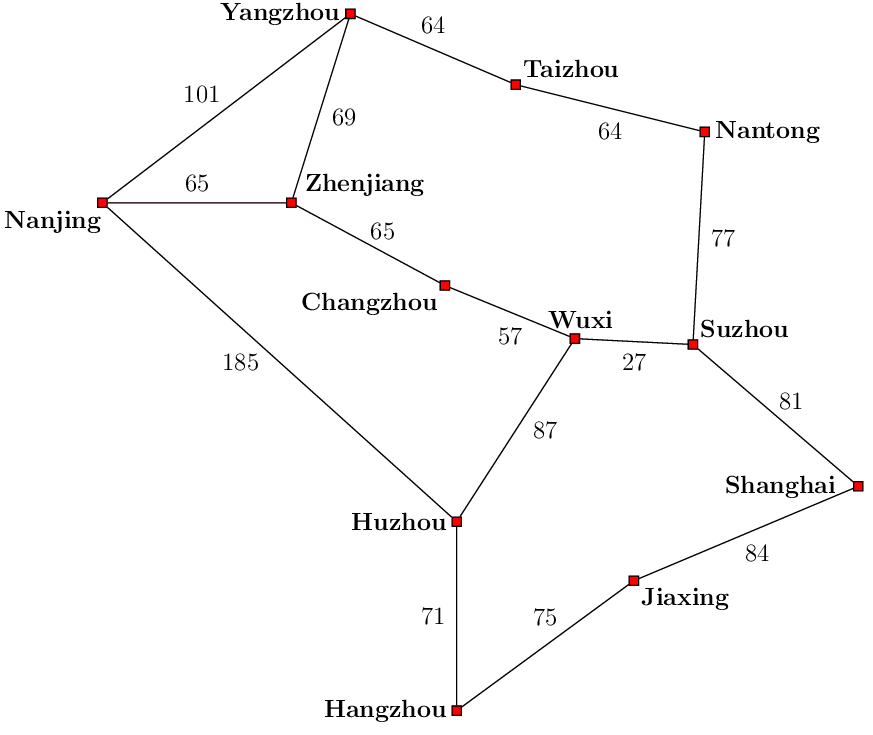}
}
\hfill
\subfigure[The state space is a spanning tree of the graph in
subplot (a). 
We ignore any leaf node that forms a Hamiltonion circle with its parents.]{
\includegraphics[width=.50\linewidth]{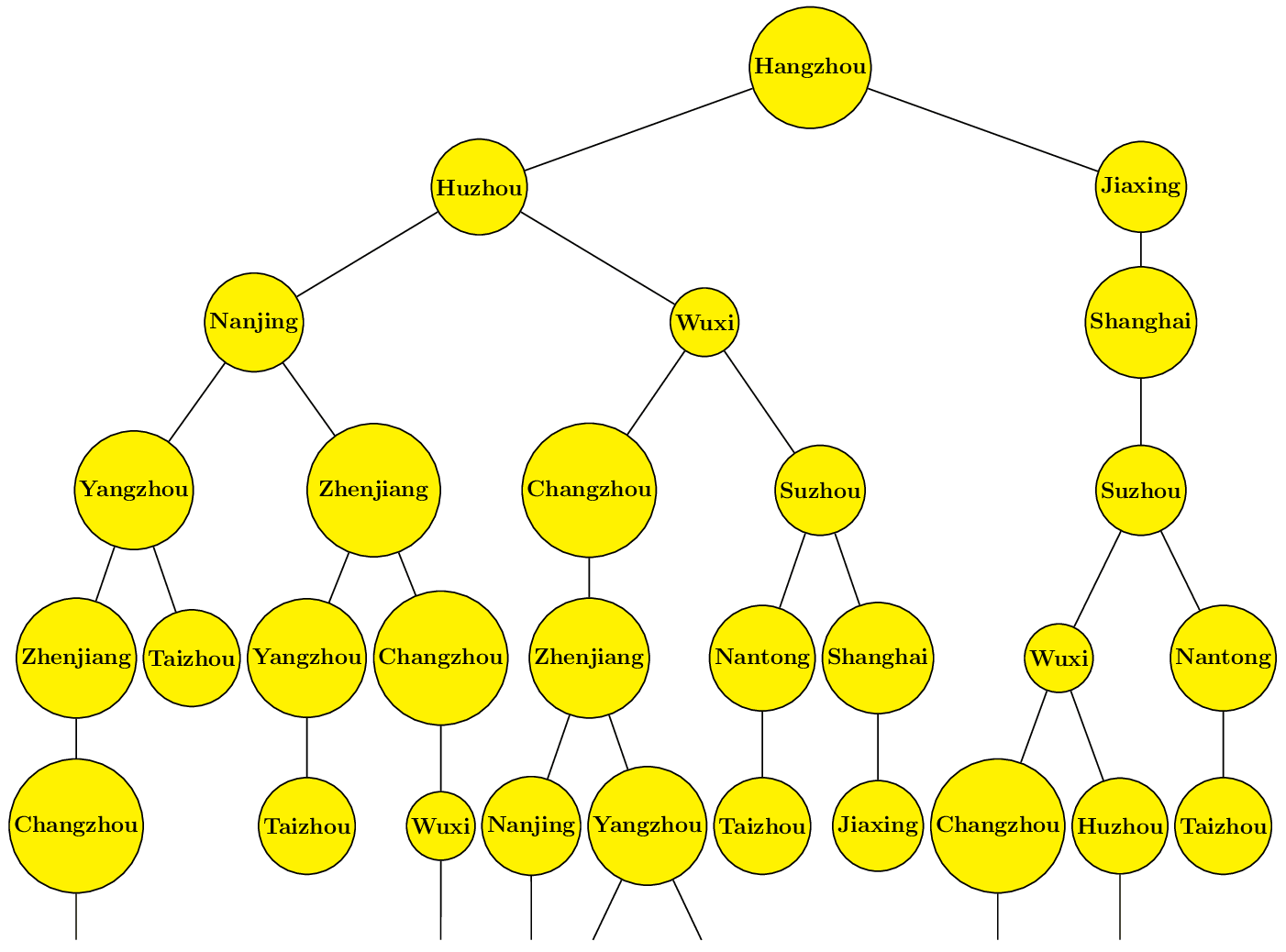}
}
\caption{An example search problem of finding
  the shortest path from Hangzhou to Taizhou.
  The initial state and the goal state are Hangzhou and Taizhou,
  respectively.
  This search problem can be solved
  by search algorithms
  such as the depth-first search, 
  the best-first search, and the $A^*$ search.
  See \cite[Chap.~3]{russell21:_artif_intel} for more details. 
}
\label{fig:shortestPath}
\end{figure}

\revise{
The PLG problem can also be formulated as a search problem
 with its initial state as the singleton lattice
 $\{\mathbf{q}\}$ that only contains the starting point.
It is similar to AI game playing in that any exhaustive search,
 by the discussion in Section \ref{sec:nature-pois-latt},
 is infeasible
 and thus the key difficulty also lies in 
 how to efficiently reduce
 the width and depth of the subtrees that actually get spanned.
 
There are, however, two crucial differences
 between PLG and AI game playing. 
First,
 it is difficult to convert the poisedness of lattices
 to the minimization of any cost functions:
 a lattice is either poised or not.
More importantly, 
 there is no intuitive way to construct the state space
 for the PLG problem. 
Adding one point at a time into the lattice
 could be one way to construct the state space,
 but it is by no means the only way.
In fact, in multiple dimensions,
it is more natural to add at a time
$m_k:={k+\Dim+1 \choose k+1}-{k+\Dim \choose k}$
 points so that the polynomial degree of a non-initial state 
 is greater by one than that of its parent state.
Even for this strategy,
 one still has many possible choices of the $m_k$ points
 to maintain the invariant of poisedness; 
 furthermore,
 these choices depend on the particularity of the input
 feasible set $K$ which is unknown at the time of algorithm design.
In summary, the structure of the state space of the PLG problem
 is highly \emph{dynamic}. 

Our algorithm for PLG
 is a fusion of elementary concepts 
 from approximation theory, group theory, and AI search algorithms.
First, we observe
 that only one poised lattice is needed per cell
 and there is no need to find all poised lattices in the feasible
 set $K$.
Hence we focus on a particular class
 of poised lattices, 
 namely the triangular lattices in Definition \ref{def:triangularLatticeDimD}, 
 to specialize PLG
 to the problem of triangular lattice generation (TLG);
 see Definition \ref{def:poisedLatticeGenProblem}. 
Second,
 for fixed $n$ and $\Dim$ 
 we study the structure of the group ${\cal X}$ of
 triangular lattices 
 and show that ${\cal X}$ is isomorphic
 to the group of $\Dim$-permutations that act on ${\cal X}$, 
 c.f. Definition \ref{def:DPermutation}. 
Based on this isomorphism,
 we identify the state space of poised lattices
 with that of $\Dim$-permutations, 
 organize the state space
 via orbits of points under the $\Dim$-permutations, 
 and prove in Theorem \ref{thm:DpermutAction}
 that a simple depth-first search with backtracking
 is optimal in generating a poised lattice; 
 see Sections \ref{sec:analysis} and \ref{sec:algorithms}
 for more details on TLG.
}

\subsection{Contributions of this work}
\label{sec:contr-this-work}

The TLG algorithm solves the PLG problem
 and yields straightforward answers of questions posed in Section \ref{sec:motivations}.
\begin{enumerate}[({A}.1)]\itemsep0em
\item Utilizing the TLG algorithm,
  we propose a new FD method, called the PLG-FD method,
  that retains Cartesian grids
  for comparable simplicity of legacy FD methods
  and solves PDEs on arbitrarily complex irregular domains
  with fourth-order accuracy.
  See Section \ref{sec:PLG-FD} for more details.
\item After generating the triangular lattice,
  we fit a multivariate polynomial
  and obtain a discrete equation for each grid point
  according to the PDE and its boundary conditions. 
  Thanks to its complete independence
  on the specific form of the PDE,
  this process also applies
  to PDEs with cross derivatives
  and complex boundary conditions. 
\item The TLG algorithm gives users of EB methods precise controls
  over least squares.
  Starting with the generated poised lattice,
  one can keep adding nearby points until the condition number
  reaches a minimum.
  \revise{Results of numerical experiments show
    that adding 2 to 5 extra points
    substantially reduces the condition number;
    see Figure \ref{fig:TLG_conditioning}. 
  }
\item The algorithmic steps of the PLG-FD method
  are conceptually the same for different dimensions
  and different orders of accuracy,
  thus a single implementation
  covers a wide range of applications.
\end{enumerate}

Apart from the above contributions for FD/FV methods,
 the TLG algorithm goes well with interpolation methods
 \cite{sauer95:_lagran,neidinger19:_multiv_newton,errachid20:_rmvpia}
 that compute the interpolating polynomials from a given poised lattice.

 \revise{Vector spaces and linear maps
   go hand in hand in mathematics;
   so do data structures and algorithms in computer science.
This duality principle is also manifested in TLG: 
 we \emph{design} the dynamic structure of the state space
 so that the search algorithm over the space may exploit
 this structure to achieve an optimal pruning of subtrees.
This distinguishes TLG
 from current AI search algorithms,
 which act over spaces whose structures are determined \emph{a priori}.
The presentation of dynamical structures of state spaces
 via abstract algebra,
 coupled with AI algorithms that exploit these structures, 
 may lead to the solution of other open problems
 in scientific computing as well as AI. 
In particular, this strategy
 could be very useful for an informed search problem 
 where the information is not conducive
 to current AI search algorithms. 
}

The rest of this paper is organized as follows.
In Section \ref{sec:preliminaries},
we introduce notation and preliminary concepts. 
In Section \ref{sec:analysis}, 
we examine the concept of triangular lattices,
 formalize the TLG problem, 
 and analyze it from the viewpoint of group actions of permutations upon triangular lattices.
We formalize in Section \ref{sec:algorithms}
 the TLG algorithm as our solution to the PLG problem
 and propose in Section \ref{sec:PLG-FD} the PLG-FD method.
In Section \ref{sec:tests},
 we test the TLG algorithm as well as the PLG-FD method
 by numerical experiments,
 demonstrating fourth-order convergence
 of the PLG-FD method for a variety of PDEs on irregular domains.
We conclude this work with some research prospects
 in Section \ref{sec:conclusion}.



\section{Preliminaries}
\label{sec:preliminaries}

In this section,
 we introduce notation
 and collect relevant definitions and results
 from several distinct disciplines 
 to make this paper self-contained.

\begin{ntn}
  \label{ntn:firstNnaturalNumbers}
  The first $n+1$ nonnegative integers  
  and the first $n$ positive integers are denoted respectively by
  \begin{equation}
    \label{eq:firstNnaturalNumbers}
    \mathbb{Z}_n := \{0,1,\ldots,n\},
    \quad \mathbb{Z}_n^+ := \{1,\ldots,n\}.
  \end{equation}
  In particular, $\mathbb{Z}_m=\emptyset$
  for any negative integer $m$
  and $\mathbb{Z}_m^+=\emptyset$
  for any nonpositive integer $m$. 
\end{ntn}

\subsection{Triangular lattices in the plane}
\label{sec:triangularGrids2D}


\begin{defn}
  \label{def:triangularLattice2D}
  A \emph{triangular lattice of degree $n$ in two dimensions}
   is a set of isolated points in $\mathbb{R}^2$, 
    \begin{equation}
    \label{eq:triangularLattice2D}
    {\mathcal T}^{2}_n := \{(x_i,y_j) : i,j\ge 0, i+j\le n\},
  \end{equation}
  where $x_i$'s and $y_j$'s are
  $n+1$ distinct $x$-
  and $y$-coordinates, respectively.
\end{defn}

\begin{exm}
  \label{exm:latticeD2n2}
  The constraints in (\ref{eq:triangularLattice2D})
   are not on the coordinates,
   but on their \emph{indices}.
  Hence a triangular lattice might have a shape
   that does not look like a triangle at all.
   \begin{center}
     \includegraphics[width=0.2\textwidth]{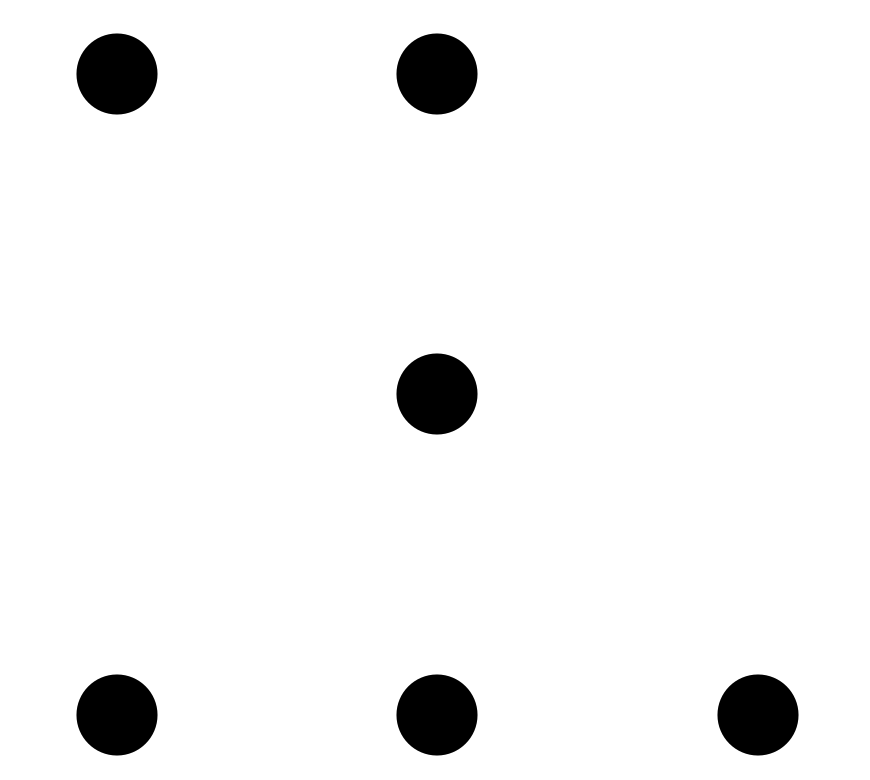}
   \end{center}
  For example, the above triangular lattice
   \begin{displaymath}
     {\mathcal T}_{2}^2 = \{(0,0), (0,2), (1,0), (1,1), (1,2), (2,0)\}
   \end{displaymath}
   has distinct coordinates
   $x_0=1$, $x_1=0$, $x_2=2$
   and $y_0=0$, $y_1=2$, $y_2=1$.
\end{exm}

\begin{thm}
  \label{thm:triangularSampleMatrix}
  A triangular lattice ${\mathcal T}^2_n$
  is poised with respect to
  bivariate \revise{monomials} of degree no more than $n$
   \begin{equation}
     \label{eq:triangularBasisFuncs}
     \Phi^2_n := \{1, x, y, x^2, xy, y^2, \ldots,
     x^n, x^{n-1}y, \ldots, xy^{n-1}, y^n\}; 
   \end{equation}
   the corresponding sample matrix $M_{2}$
   resulting from (\ref{eq:sampleMatrix}) satisfies
  \begin{equation}
    \label{eq:triangularSampleMatrix}
    \det M_{2} = C \psi_n(x) \psi_n(y), 
  \end{equation}
  where $C$ is a nonzero constant and
  $\psi_n(x)$ is a polynomial
  in terms of the $n+1$ distinct coordinates $x_i$'s, 
  \begin{equation}
    \label{eq:triangularSampleMatrixDetFactor}
    \psi_n(x) 
    := \prod_{i=1}^n \prod_{{\ell}=0}^{i-1} (x_i-x_{\ell})^{n+1-i}.
  \end{equation}
\end{thm}
\begin{proof}
  \revise{For any fixed $i,{\ell}$ with $i>{\ell}$,
  we have, from (\ref{eq:sampleMatrix}),
  \begin{align}
    \nonumber
    \det M_{2} &= \det
    \begin{bmatrix}
      \cdots & \phi_1(x_{\ell},y_j) & \cdots
      & \phi_1(x_{i},y_j) & \cdots
      \\
      \vdots & \vdots & \vdots & \vdots & \vdots
      \\
      \cdots & \phi_N(x_{\ell},y_j) & \cdots
      & \phi_N(x_{i},y_j) & \cdots
    \end{bmatrix}
    \\ \nonumber
    &= \det
      \begin{bmatrix}
        \cdots & \phi_1(x_{\ell},y_j)-\phi_1(x_{i},y_j) & \cdots
        & \phi_1(x_{i},y_j) & \cdots
        \\
        \vdots & \vdots & \vdots & \vdots & \vdots
        \\
        \cdots & \phi_N(x_{\ell},y_j)-\phi_N(x_{i},y_j) & \cdots
        & \phi_N(x_{i},y_j) & \cdots
      \end{bmatrix}
    \\ \label{eq:determinantM2}
    &= \sum_{k=1}^N[\phi_k(x_{\ell},y_j)-\phi_k(x_{i},y_j)]
      (-1)^{k+c_{\ell}}\det C_{k,c_{\ell}},
  \end{align}
  where $c_{\ell}$ is the column index
  of $(x_{\ell},y_j)$ in $M_2$, 
  the last step follows from the Laplace formula of determinants,
  and $C_{k,c_{\ell}}$ is the matrix that results from
  deleting the $k$th row and the $c_{\ell}$th column of $M_2$.
  For any $\phi_k$ in $\Phi^2_n$,
  we have $(x_i-x_{\ell}) | [\phi_k(x_{\ell},y_j)-\phi_k(x_{i},y_j)]$,
  i.e. $(x_i-x_{\ell})$ divides
  $[\phi_k(x_{\ell},y_j)-\phi_k(x_{i},y_j)]$.
  Apply (\ref{eq:determinantM2}) to $C_{k,c_{\ell}}$
  and we know
  $(x_i-x_{\ell}) | \det C_{k,c_{\ell}}$ for a different $y_j$.
  Since (\ref{eq:triangularLattice2D}) dictates $j\le n-i$, 
  a recursive expansion of $\det C_{k,c_{\ell}}$ implies
  $(x_i-x_{\ell})^{n+1-i} | \det M_2$.
}

  Now we vary ${\ell}$ while keeping $i$ fixed.
  Since there are $i$ indices less than $i$,
   the term $\prod_{{\ell}=0}^{i-1} (x_i-x_{\ell})^{n+1-i}$
   contributes to a total degree of 
   $i(n+1-i)$ in terms of the $n+1$ coordinates $x_0, \ldots, x_n$.
  Hence the total degree of $\psi_n(x)$ is \revise{at least}
   \begin{equation}
     \label{eq:degreeInduction1}
     \sum_{i=1}^n i(n+1-i)
     = (n+1)\sum_{i=1}^ni -\sum_{i=1}^ni^2 = \frac{n(n+1)(n+2)}{6},
   \end{equation}
   where the second equality is proven by an easy induction.
  Similarly,
   $\det M_{2}$ must contain a factor of $\psi_n(y)$,
   of which the total degree is \revise{at least}
   (\ref{eq:degreeInduction1}).

  The Leibniz formula of the determinant of a square matrix $A\in \mathbb{R}^{n\times n}$
   is $\det A = \sum_{\sigma\in S_n} \text{sgn}(\sigma) \prod_{k=1}^n
   a_{\sigma(k),k}$
   where the sum is over the symmetric group $S_n$ of all
   permutations, c.f. Definition \ref{def:symmetricGroup}, 
   and $a_{\sigma(k),k}$ is the element of $A$
   at the $\sigma(k)$th row and the $k$th column.
  Thus the determinant of the sample matrix $M_{2}$
   in (\ref{eq:sampleMatrix})
   is a polynomial in terms of the variables $x_0,x_1,\ldots,x_n$ and
   $y_0,y_1,\ldots,y_n$,
   with each monomial being a product of all basis functions in
   (\ref{eq:triangularBasisFuncs})
   evaluated at some point $(x_i,y_j)$.
  Hence the total degree of $\det M_{2}$
   in the variables $x_0,x_1,\ldots,x_n$ and
   $y_0,y_1,\ldots,y_n$ is \revise{equal to}
  \begin{equation}
     \label{eq:degreeInduction2}
     \sum_{i=1}^n i(i+1) = \frac{n(n+1)(n+2)}{3},
   \end{equation}
   where $i$ refers to the degree of a monomial
   and $i+1$ the number of monomials of degree $i$,
   c.f. (\ref{eq:triangularBasisFuncs}). 
   \revise{
  By the above arguments,
   the total degrees of 
   $\psi_n(x)\psi_n(y)$ and $\det M_{2}$ are
   at least and equal to (\ref{eq:degreeInduction2}), respectively.
  The proof is then completed by
   $\psi_n(x)\psi(y) | \det M_{2}$ and $\det M_{2}\ne 0$.}
\end{proof}

The $k$th \emph{divided difference} of a function
 $f: \mathbb{R}\rightarrow \mathbb{R}$
 over $x_0, \ldots, x_k\in \mathbb{R}$
 is
\begin{equation}
  \label{eq:divDiffRecursion}
  [x_0, x_1,\ldots, x_k]f = 
  \frac{[x_1, x_2, \ldots, x_k]f - [x_0, x_1,\ldots, x_{k-1}]f}{x_k-x_0}
\end{equation}
with the recursion termination condition
as $[x_0]f = f(x_0)$.
We use the notation $[x_0, x_1,\ldots, x_k]f$
to emphasize the functional nature
of divided differences.
\revise{For $f: \mathbb{R}^2\rightarrow \mathbb{R}$,
 its divided difference $[x_0, x_1,\ldots, x_k][y_0, y_1,\ldots, y_j]f$
 is defined recursively from (\ref{eq:divDiffRecursion})
 as $[x_0, x_1,\ldots, x_k]g$
 where $g:\mathbb{R}\rightarrow\mathbb{R}$
 depends only on $x$ and is given by $g(x)=[y_0, y_1,\ldots, y_j]f(x,y)$}.
 
The following is an error estimate
 for interpolation on triangular lattices.

\begin{thm}
  \label{thm:triangularInterpolantAndRemainder}
  For any scalar function $f$
   whose domain includes ${\mathcal T}^2_n$,
   we have, 
   \begin{equation}
     \label{eq:triangularInterpolantAndRemainder}
     \forall m\in \mathbb{Z}_n,\  
     f(x,y) = p_m(x,y) + r_m(x,y),
   \end{equation}
   where $p_m(x,y)$ is a polynomial interpolator of $f(x,y)$
   and $r_m(x,y)$ is the remainder on ${\mathcal T}^2_m$, 
   \begin{subequations}
     \label{eq:triangularInterpolantTerms}
   \begin{align}
     p_m(x,y) =& 
     \begin{cases}
       [x_0][y_0]f = f(x_0,y_0), & m=0;
       \\
       p_{m-1}(x,y)+q_m(x,y), & m>0,
     \end{cases}
     \\
     q_m(x,y) =& \sum_{k=0}^m\pi_k(x)\pi_{m-k}(y)
     [x_0,\ldots,x_k][y_0,\ldots,y_{m-k}]f,
     \\ 
     r_m(x,y) =& \sum_{k=0}^m\pi_{k+1}(x)\pi_{m-k}(y) 
     [x,x_0,\ldots,x_k][y_0,\ldots,y_{m-k}]f
     \\ \nonumber
     &+ \pi_{m+1}(y)[x][y,y_0,\ldots,y_{m}]f,
   \end{align}
 \end{subequations}
 where the symmetric polynomial $\pi_n(x)$ is
  \begin{equation}
    \label{eq:NewtonPoly}
    \pi_n(x)=
    \begin{cases}
      1, & n=0;
      \\
      \prod_{i=0}^{n-1}(x-x_i), & n>0.
    \end{cases}
  \end{equation}
\end{thm}
\begin{proof}
  See \cite[p. 180]{phillips03:_inter_approx_polyn}.
\end{proof}
\revise{
\subsection{Groups, homomorphisms,
   permutations, orbits, and isotropy subgroups}
\label{sec:symmetric-group}

\begin{defn}
  \label{def:group}
A \emph{group} is a set $G$
 with a binary operation
 $*: G\times G\rightarrow G$
 such that the following axioms hold, 
\begin{enumerate}[(GRP-1)]\itemsep0em
  \itemsep0em
\item associativity: 
  $\forall a,b,c\in G,\ \ (a* b)* c = a* (b* c)$;
\item identity element: 
  $\exists e\in G,\text{ s.t. }$
  $\forall x\in G,\ 
  e* x=x* e =x$;
\item inverse element: 
  $\forall a\in G, \exists a^{-1}\in G \text{ s.t. }
  a^{-1}* a = a* a^{-1}= e$. 
\end{enumerate}
\end{defn}

\begin{defn}
  \label{def:isomorphicGroups}
  Two groups $(G,*)$ and $(H,\cdot)$
  are \emph{homomorphic} if
  there exists a function $\varphi: G\rightarrow H$
  such that
  \begin{equation}
    \label{eq:homomorphism}
    \forall g_1, g_2\in G, \quad
    \varphi(g_1*g_2) = \varphi(g_1)\cdot\varphi(g_2); 
  \end{equation}
  in which case $\varphi$ is called a \emph{homomorphism} from $G$ to $H$.
  Furthermore,
  a homomorphism is called
  a \emph{monomorphism}, \emph{epimorphism},
  and \emph{isomorphism}
  if it is respectively injective, surjective, and bijective; 
  then the two groups are said to be
  \emph{monomorphic}, \emph{epimorphic},
  and \emph{isomorphic}, respectively.
\end{defn}

A homomorphism from $G$ to $H$
 preserves algebraic structures of $G$.
Isomorphic groups have exactly the \emph{same} algebraic structure
 and thus we can characterize this structure
 in any of these groups.

A \emph{permutation of a set} $A$
  is a bijective function $\sigma: A \rightarrow A$.
Let $A$ be a non-empty set,
and let $S_A$ be the collection of all permutations of $A$.
Then $(S_A, \circ, ^{-1}, e)$
is a group where the identity $e$ is the identity function.

\begin{defn}
  \label{def:symmetricGroup}
  The \emph{symmetric group on $n$ letters},
  written $S_{\mathbb{Z}_n^+}$ or simply $S_n$,
  is the group of all permutations of 
  $\mathbb{Z}_n^+$ in Notation \ref{ntn:firstNnaturalNumbers}. 
\end{defn}

By Cayley's theorem, 
 every group $G$ is isomorphic to a subgroup of $S_G$.

\begin{defn}
  \label{def:orbitOfPermutation}
  An \emph{orbit of a permutation $\sigma\in S_n$}
  is an equivalence class of $\mathbb{Z}_n^+$ determined by the equivalence
  relation
  \begin{equation}
    \label{eq:orbitOfPermutationEquiv}
    \forall a,b\in \mathbb{Z}_n^+,\quad
    a\sim b \ \Leftrightarrow\
    \exists m\in \mathbb{Z} \text{ s.t. } a = \sigma^m b. 
  \end{equation}
\end{defn}

A \emph{cycle} is a permutation $\sigma\in S_n$
that has at most one orbit that contains multiple elements.
In particular, the identity permutation is a cycle.
Two cycles in the same group $S_n$
 are \emph{disjoint} if any integer in $\mathbb{Z}_n^+$
 is moved by at most one of these two cycles.
Disjoint cycles are commutative.

\begin{thm}
  \label{thm:permutationIsDisjointCycles}
  Any permutation of a finite set is a product of disjoint cycles.
\end{thm}
\begin{proof}
  For each $B_i$ of the $m$ orbits of a permutation $\sigma$,
  define a cycle $\mu_i$ by
  $\mu_i(x):= \sigma(x)$ if $x\in B_i$
  and $\mu_i(x):= x$ otherwise.
  Clearly $\sigma=\mu_1\mu_2\cdots \mu_m$. 
  By Definition \ref{def:orbitOfPermutation},
  the orbits $B_i$'s of $\sigma$ are disjoint,
  hence so are the cycles.
\end{proof}

A set $H$ is a \emph{subgroup} of $G$ if $H\subset G$
and $(H, *, ^{-1}, e)$ is a group.
The \emph{order of a group} $G$, written $|G|$, 
is the number of elements in $G$.
By the Lagrange theorem, 
 the order $|H|$ of a subgroup $H$ of a finite group $G$
 is a divisor of $|G|$. 
 
\begin{defn}
  \label{def:coset}
  A \emph{(left) coset} of a subgroup $H$ in a group $G$
  is a set
  $aH := \{a h : h\in H\}$ 
  where $a$ is an element in $G$.
\end{defn}

In particular, a subgroup $H$ of a group $G$
 is a coset in $G$ since $eH=H$. 
Conversely, a coset of $H$
 needs not to be a subgroup in $G$.

\begin{lem}
  \label{lem:leftCosetsPartitionsGroup}
  For any subgroup $H$ of a group $G$, 
  its left cosets partition $G$. 
\end{lem}
\begin{proof}
  It is readily verified from Definitions
  \ref{def:group} and \ref{def:coset}
  that cosets of $H$ in $G$
  are equivalence classes of the congruence relation
  \begin{displaymath}
    \forall a,b\in G,\quad
    (a\sim b\ \Leftrightarrow\ 
    \exists h\in H \text{ s.t. } b=ah).
  \end{displaymath}
\end{proof}

In light of Lemma \ref{lem:leftCosetsPartitionsGroup}, 
\emph{the index of a subgroup $H$ of a group $G$},
 written $(G:H)$, 
 is the number of left cosets of $H$ in $G$.
Lemma \ref{lem:leftCosetsPartitionsGroup} also yields

\begin{coro}
  \label{coro:equalCosets}
  The left cosets $aH$ and $bH$ are equal
  if and only if $b\in aH$
  if and only if $a^{-1}b\in H$.
\end{coro}

 \begin{defn}
  \label{def:groupAction}
  An \emph{action of a group $G$ on a set $X$}
  is a map $G\times X \rightarrow X$,
  written $(g,x)\mapsto g(x)$ for $g\in G$ and $x\in X$, 
  that satisfies
  (i) $\forall x\in X$, $e(x) = x$
  and (ii) $\forall x\in X$, $\forall g_1,g_2\in G$,
  $(g_1*g_2)(x)=g_1(g_2(x))$.
  $X$ is called a \emph{$G$-set} if $G$ has an action on $X$.
\end{defn}

\begin{defn}
  \label{def:orbitForGsets}
  The \emph{orbit of an element $x$ in a $G$-set $X$
    under the group $G$} 
  is the set $Gx := \{g(x) : g\in G\}$. 
\end{defn}

\begin{lem}
  \label{lem:isotropySubgroup}
  For any $G$-set $X$ and $x\in X$,
  the set $G_x:= \{g\in G: gx = x\}$
  is a subgroup of $G$
  called the \emph{isotropy subgroup of $G$ at $x$}.
\end{lem}
\begin{proof}
  This follows from Definition \ref{def:group}
  and the definition of subgroups.
\end{proof}
}



   
 



\subsection{Backtracking}
\label{sec:backtracking}

Consider a task that consists of a sequence of $k$ independent steps.
The \emph{fundamental principle of counting} states that
the total number of distinct ways to complete the task
is $\prod_{i=1}^{k} n_i$, 
where $n_i$ is the number of different choices for the $i$th step.
From another perspective,
we can view the process of completing the task
as traversing a directed tree.
The root node of the tree represents the starting point
with no steps done yet.
The root node has $n_1$ children, each of which
represents one way to finish the first step.
Inductively,
the $n_i$ ways to finish the $i$th step 
are represented by $n_i$ nodes,
all of which are children of the single node representing
the $(i-1)$th step.
The $\prod_{i=1}^{k} n_i$ leaf nodes form the solution space
of all different ways to complete the task.
This tree is called the \emph{spanning tree}
of the solution space.

\begin{defn}
\label{def:backtracking}
\revise{\emph{Backtracking} is a class of algorithms
 for solving a search problem $P$
 that admit an incremental assemblage of the solution
 and hence a spanning-tree organization of the state space.}
A particular implementation of backtracking is
 \texttt{BackTrack}($P$, \texttt{root}$(P),\{\})$ 
 where the generic algorithm is
 
\IncMargin{1em}
\RestyleAlgo{algoruled}\LinesNumbered
\begin{procedure}[H]
  \DontPrintSemicolon
  \SetKwInOut{SideEffect}{Side effects}
  \SetKwInOut{Precond}{Preconditions}
  \SetKw{Return}{return}

\caption{BackTrack($P$, $\mathbf{r}$, $T$)}
  \KwIn{$\mathbf{r}$ is a starting node in the solution space
   and $T$ is a set of solutions}
  \SideEffect{$T$ contains all valid solutions}
\BlankLine

\uIf{\texttt{accept}$(P, \mathbf{r})$}{
  $T = T \cup \{\mathbf{r}\}$\;
  \lIf{\texttt{stopAfterAccept}$(P, \mathbf{r})$}{\Return}
}
\ElseIf{\texttt{reject}$(P, \mathbf{r})$}{
  \Return
}
$\mathbf{s} \leftarrow$ \texttt{first}$(P,\mathbf{r})$\; 
\While{$\mathbf{s}$ is not null}{
  $\mathrm{BackTrack}(P, \mathbf{s}, T)$\;
  $\mathbf{s} \leftarrow$ \texttt{next}$(P, \mathbf{r}, \mathbf{s})$
}
\end{procedure}
\DecMargin{1em}

and the user-defined subroutines are
\begin{itemize}\itemsep0em
\item \texttt{root}($P$):
  return the root node of the spanning tree of the solution space,
\item \texttt{accept}$(P,\mathbf{r})$:
  return true if and only if $\mathbf{r}$ is already a solution;
  return false otherwise,
\item \texttt{stopAfterAccept}$(P,\mathbf{r})$:
  return true if and only if
  the valid solution $\mathbf{r}$ can never be extended to another valid
  solution; 
  return false otherwise,
\item \texttt{reject}$(P,\mathbf{r})$:
  return true if and only if the partial candidate (or node) $\mathbf{r}$
  can never be completed to a valid solution;
  return false otherwise,
\item \texttt{first}$(P,\mathbf{r})$:
  return the first extension of $\mathbf{r}$ in the spanning tree
  if $\mathbf{r}$ is extendable;
  return null otherwise,
\item \texttt{next}$(P,\mathbf{r},\mathbf{s})$:
  return the next extension of $\mathbf{r}$ after $\mathbf{s}$
  if $\mathbf{r}$ is still extendable;
  return null otherwise,
\end{itemize}

\end{defn}


Backtracking incrementally builds solutions
 by abandoning any candidate
 that cannot be completed to a valid solution.
Hence it can be much faster
 than brute-force enumeration
 since the procedure \texttt{reject}
 may eliminate a large number of candidates
 with a single test.
In particular, if \texttt{reject} always returns false,
 the procedure in Definition \ref{def:backtracking}
 reduces to a brute-force enumeration.

In Definition \ref{def:backtracking},
 although results of the two subroutines \texttt{accept}($P,\mathbf{r}$)
 and \texttt{reject}($P,\mathbf{r}$) cannot be both \texttt{true},
 they may be both \texttt{false}:
 an intermediate node $\mathbf{r}$ for which both subroutines return \texttt{false}
 may later be extended to a valid solution.
On the other hand,
 if \texttt{accept}($P,\mathbf{r}$) is true for some $\mathbf{r}$,
 only two cases are possible:
 either $\mathbf{r}$ may be extended to another valid solution,
 or, $\mathbf{r}$ can never be so.
In the latter case
 \texttt{stopAfterAccept} should return true.

All non-null nodes generated by \texttt{first}($P,\mathbf{r}$)
 and \texttt{next}($P,\mathbf{r},\mathbf{s}$)
 have the same rank 
 with respect to \texttt{root}($P$),
 i.e. the number of extensions from \texttt{root}($P$).



\section{Analysis and algebra}
\label{sec:analysis}

\subsection{Triangular lattices in $\Dim$ dimensions}
\label{sec:triang-latt-high}

Hereafter, we reserve
the symbol $i$ for indexing dimensions:
$i=1,2,\ldots,\Dim$.

\begin{defn}
  \label{def:triangularLatticeDimD}
  A subset ${\mathcal T}_n^{\Dim}$ of $\mathbb{R}^{\Dim}$
   is called
   a \emph{triangular lattice of degree $n$ in $\Dim$ dimensions}
   if for each of the $\Dim$ dimensions
   there exist $n+1$ distinct coordinates 
   and a numbering of these coordinates,
   \begin{equation}
     \label{eq:trigLatCoords}
     \begin{bmatrix}
       p_{1,0} & p_{1,1} & \ldots & p_{1,n}
       \\
       p_{2,0} & p_{2,1} & \ldots & p_{2,n}
       \\
       \vdots & \vdots & \ddots & \vdots
       \\
       p_{{\Dim},0} & p_{{\Dim},1} & \ldots & p_{{\Dim},n}
     \end{bmatrix}
     \in \mathbb{R}^{\Dim\times (n+1)},
   \end{equation}
   such that ${\mathcal T}_n^{\Dim}$ can be expressed as
   \begin{equation}
     \label{eq:triangularLatticeDDim}
     {\mathcal T}_n^{\Dim} := \left\{\left
         (p_{1,k_1}, p_{2,k_2}, \ldots, p_{D,k_D}\right)
       \in \mathbb{R}^{\Dim}:\ 
       k_i\in \mathbb{Z}_n;\ \sum_{i=1}^\Dim k_i\le n\right\},
   \end{equation}
   where $p_{i,j}$ denotes the $j$th coordinate
   of the $i$th variable $p_i$.
   \revise{When there is no danger of confusion,
     we sometimes use the shorthand notation ${\cal T}$
     for ${\mathcal T}_n^{\Dim}$}. 
\end{defn}

\begin{exm}
  For $\Dim=2$, Definition \ref{def:triangularLatticeDimD}
  reduces to Definition \ref{def:triangularLattice2D}
  since (\ref{eq:triangularLatticeDDim}) simplifies to
    \begin{displaymath}
      {\mathcal T}^2_n = \{(p_{1,k_1}, p_{2,k_2}):\
      k_1,k_2\ge 0;\ k_1+k_2\le n\},
  \end{displaymath}
  which is the same as (\ref{eq:triangularLattice2D}).
\end{exm}

\begin{lem}
  \label{lem:cardTriangularLattice}
  The cardinality of a triangular lattice is
  \begin{equation}
    \label{eq:numTriangularLattice}
    \#{\mathcal T}^{\Dim}_n=\binom{n+\Dim}{\Dim}
    =\sum_{i=0}^n \#{\mathcal T}^{\Dim}_{=i}
    = \sum_{i=0}^{n}\binom{i+\Dim-1}{\Dim-1}, 
  \end{equation}
  where
  \begin{equation}
    \label{eq:newTriangularLattice}
    {\mathcal T}_{=m}^{\Dim} :=
    \left\{(p_{1,k_1}, p_{2,k_2}, \ldots, p_{\Dim,k_\Dim}):
      k_i\ge 0; \sum_{i=1}^{\Dim} k_i=m\right\}
  \end{equation}
  and its cardinality is
  \begin{equation}
    \label{eq:numNewTriangularLattice}
    \#{\mathcal T}_{=m}^{\Dim}
    = \#{\mathcal T}^{\Dim}_m-\#{\mathcal T}^{\Dim}_{m-1}
    =\binom{m+\Dim-1}{\Dim-1}.
  \end{equation}
\end{lem}
\begin{proof}
  As the only nontrivial constraint on ${\mathcal T}_n^{\Dim}$
  in (\ref{eq:triangularLatticeDDim}),
   the sum of the $\Dim$ non-negative integers
   $i_1, \ldots, i_{\Dim}$ cannot exceed $n$.
  Hence $\#{\mathcal T}^{\Dim}_n$ equals the number of possibilities
   of placing $n$ indistinguishable balls into $\Dim+1$ distinguishable urns,
   with the first $\Dim$ urns corresponding to the $\Dim$ dimensions,
   respectively,
   and the last urn accounting for the deficit of $\sum_k i_k$ from $n$.
  This ball-urn problem is further equivalent to
   choosing $\Dim$ balls from $n+\Dim$ balls in a row 
   because the chosen $\Dim$ balls divide the rest $n$ balls
   into $\Dim+1$ consecutive arrays of balls,
   each of which corresponds to an urn.
  This proves the first equality in (\ref{eq:numTriangularLattice}).

  As for the second equality in (\ref{eq:numTriangularLattice}),
   (\ref{eq:triangularLatticeDDim}) and (\ref{eq:newTriangularLattice})
   imply
    ${\mathcal T}_n^{\Dim} = \cup_{m=0}^n {\mathcal T}_{=m}^{\Dim}$
  and that two subsets ${\mathcal T}_{=m}^{\Dim}$
  and ${\mathcal T}_{=\ell}^{\Dim}$ are disjoint
  if and only if $m\ne \ell$.

  An easy calculation yields (\ref{eq:numNewTriangularLattice}),
   which implies the third equality in (\ref{eq:numTriangularLattice}).
\end{proof}

\begin{defn}
  \label{def:DvariatePolysOfDegreeNoMoreThanN}
  The set of \emph{$\Dim$-variate monomials of degree no more
  than $n$} is 
  \begin{equation}
    \label{eq:triangularBasisFuncsDDim}
    \Phi^{\Dim}_n := \left\{p_1^{e_1}p_2^{e_2}\ldots p_\Dim^{e_\Dim}:
      e_i\geq 0; \sum_{i=1}^\Dim e_i \le n \right\} 
  \end{equation}
  and the set of \emph{$\Dim$-variate monomials of degree $m$} is
  \begin{equation}
    \label{eq:triangularBasisFuncsDDimSlice}
    \Phi^{\Dim}_{=m} := \left\{p_1^{e_1}p_2^{e_2}\ldots p_\Dim^{e_\Dim}:
      e_i\geq 0; \sum_{i=1}^\Dim e_i = m \right\}. 
  \end{equation}
\end{defn}

\begin{lem}
  \label{lem:cardMultivariatePolys}
  The results on triangular lattices 
  in Lemma \ref{lem:cardTriangularLattice}
  also hold for sets of multivariate polynomials.
  More precisely,
  (\ref{eq:numTriangularLattice})
  and (\ref{eq:numNewTriangularLattice}) still hold
  when the symbol ${\mathcal T}$ is replaced with the symbol $\Phi$.
\end{lem}
\begin{proof}
  This follows from a natural bijection between
  $\Phi^{\Dim}_n$ and ${\mathcal T}_n^{\Dim}$,
  the restriction of which is also a bijection between
  $\Phi^{\Dim}_{=m}$ and ${\mathcal T}^{\Dim}_{=m}$.
\end{proof}

\begin{lem}
  \label{lem:detDegreeEquality}
  For any positive integers $\Dim$ and $n$, we have
  \begin{equation}
    \label{eq:detDegreeEquality}
    (\Dim+1) \sum_{j=1}^n j \binom{n-j+\Dim}{\Dim}
    = \sum_{j=1}^n j \binom{j+\Dim}{\Dim}.
  \end{equation}
\end{lem}
\begin{proof}
  For $n=1$, both sides reduce to $\Dim+1$.
  Suppose (\ref{eq:detDegreeEquality}) holds.
  Then the inductive step also holds because
  \begin{align*}
    (\Dim+1) \sum_{j=1}^{n+1} j \binom{n+1-j+\Dim}{\Dim}
    =& (\Dim+1) \sum_{i=0}^{n} (i+1) \binom{n-i+\Dim}{\Dim} \\
    =& (\Dim+1) \sum_{i=0}^{n} i \binom{n-i+\Dim}{\Dim}
       + (\Dim+1)\sum_{j=0}^{n} \binom{j+\Dim}{\Dim}
    \\
    =& \sum_{j=1}^n j \binom{j+\Dim}{\Dim}
       + (\Dim+1) \binom{n+\Dim+1}{\Dim+1} \\
    =& \sum_{j=1}^{n+1} j \binom{j+\Dim}{\Dim},
  \end{align*}
  where the first two steps follow from variable substitutions,
  the third step from the induction hypothesis
  and the well-known identity
  $\sum_{j=0}^{n} \binom{\Dim+j}{\Dim} = \binom{\Dim+n+1}{\Dim+1}$,
  and the last step from
  \begin{displaymath}
    (\Dim+1) \binom{n+\Dim+1}{\Dim+1}
    = (\Dim+1) \frac{(n+\Dim+1)!}{(\Dim+1)!n!}
    = (n+1) \binom{n+\Dim+1}{\Dim}. 
  \end{displaymath}
\end{proof}

\begin{thm}
  \label{thm:triangularSampleMatrixDDim}
 A triangular lattice ${\mathcal T}_n^{\Dim}$ is poised
  with respect to $\Phi^{\Dim}_n$ in
  (\ref{eq:triangularBasisFuncsDDim})
  and the corresponding sample matrix $M_{\Dim}$ satisfies
  \begin{equation}
    \label{eq:triangularSampleMatrixDDim}
    \det M_{\Dim} = C\prod_{k=1}^{\Dim} \psi_n(p_k)
  \end{equation}
  where $C$ is a nonzero constant and
  $\psi_n(p_k)$ is a polynomial
  in terms of the $n+1$ distinct coordinates of the variable $p_k$: 
  \begin{equation}
    \label{eq:triangularSampleMatrixDetFactorDDim}
    \psi_n(p_k) 
    := \prod_{i_k=1}^n \prod_{\ell=0}^{i_k-1}
    (p_{k,i_k}-p_{k,\ell})^{\alpha({i_k})}
    \text{ and }
    \alpha({i_k}) := \binom{n-i_k+\Dim-1}{\Dim-1}.
  \end{equation}
\end{thm}
\begin{proof}
  We follow the strategy in the proof
  of Theorem \ref{thm:triangularSampleMatrix},
  with more complicated book-keeping on the combinatorics. 

  Consider the $k$th variable $p_k$.
  \revise{For any fixed $i_k$ and $\ell$ with $i_k >\ell$,
  we would like to show
  $(p_{k,i_k}-p_{k,\ell})^{\alpha({i_k})} | \det M_{\Dim}$. }
  When the coordinate index of the $k$th variable $p_k$
   is fixed at $i_k$ in ${\mathcal T}_n^{\Dim}$,
   the cardinality of ${\mathcal T}_n^{\Dim}$ reduces
   to $\# {\mathcal T}^{\Dim-1}_{n-i_k}$,
   which,
   by Lemma \ref{lem:cardTriangularLattice}, 
   must equal $\alpha(i_k)$.
  In other words, 
  \begin{displaymath}
    \# \left\{(p_{1,i_1}, \ldots, p_{k,i_{k}}, \ldots, 
    p_{\Dim, i_{\Dim}}) \in {\mathcal T}_n^{\Dim}:
    \ i_{k} \text{ is fixed}
    \right\}
  \end{displaymath}
  equals the cardinality of a triangular lattice
  of degree $n-i_k$ in $\Dim-1$ dimensions
  because an index of $i_k$ has been consumed from the total index $n$
  and one of the $\Dim$ dimensions has already been fixed.
  \revise{By an analogy to (\ref{eq:determinantM2}),
  the number of the factor $(p_{k,i_k}-p_{k,\ell})$
  in $\det M_{\Dim}$ 
   is at least $\alpha(i_k)$.}
   
  Now we vary $\ell$ while keeping $i_k$ fixed.
  Since there are $i_k$ indices less than $i_k$,
   the term $\prod_{\ell=0}^{i_k-1} (p_{k,i_k}-p_{k,\ell})^{\alpha(i_k)}$
   contributes to a total degree
   $i_k\alpha(i_k)$ in terms of the $n+1$ coordinates
   $p_{k,0}, \ldots, p_{k,n}$.
  It follows that $\psi_n(p_k)$ must be a factor of $\det M_{\Dim}$
   and the total degree of $\psi_n(p_k)$ is
   \begin{displaymath}
     \sum_{i_k=1}^n i_k \alpha(i_k)
     = \sum_{j=1}^n j \binom{n-j+\Dim-1}{\Dim-1}.
   \end{displaymath}
   
  Similarly,
   $\det M_{\Dim}$ must contain a factor of $\psi_n(p_j)$
   for each variable $p_j$, $j=1,\ldots, \Dim$.
  Hence the total degree of $\det M_{\Dim}$ is at least
   \begin{displaymath}
     \xi := \Dim \sum_{j=1}^n j \binom{n-j+\Dim-1}{\Dim-1}.
   \end{displaymath}

  The determinant of $M_{\Dim}$
   is also a polynomial in terms of the coordinates
   $p_{1,0},p_{1,1},\ldots,p_{1,n}$, $\ldots$,
   $p_{\Dim,0},p_{\Dim,1},\ldots,p_{\Dim,n}$, 
   with each monomial being a product of all basis functions in
   (\ref{eq:triangularBasisFuncsDDim})
   evaluated at some point
   $(p_{1,i_1},p_{2,i_2},\ldots,p_{\Dim,i_{\Dim}})$.
   By Lemma \ref{lem:cardMultivariatePolys}
    and (\ref{eq:numNewTriangularLattice}), 
    the total degree of $\det M_{\Dim}$ equals
    $\eta := \sum_{i=1}^n i \binom{i+\Dim-1}{\Dim-1}$, 
   where $i$ refers to the degree of monomials
   in the subset $\Phi^{\Dim}_{=i}$
   and $\binom{i+\Dim-1}{\Dim-1}$ 
   the cardinality of $\Phi^{\Dim}_{=i}$.
  Lemma \ref{lem:detDegreeEquality} implies $\xi=\eta$.
  \revise{Then (\ref{eq:triangularSampleMatrixDDim}) follows
  from $\det M_{\Dim}\ne 0$
  and $\psi_n(p_k) | \det M_{\Dim}$
  for each $k\in \mathbb{Z}_{\Dim}^+$.}
\end{proof}

\subsection{The problem of triangular lattice generation (TLG)}
\label{sec:TLG}

By Definition \ref{def:triangularLatticeDimD},
 projecting the isolated points in a triangular lattice
 to any axis yields $n+1$ distinct coordinates in $\mathbb{R}$.
Restricting $\mathbb{R}$ to $\mathbb{Z}$
 incurs no loss of generality 
 since each triangular lattice in $\mathbb{R}^{\Dim}$
 is isomorphic to some triangular lattice in $\mathbb{Z}^{\Dim}$.
 
\begin{defn}[The TLG problem]
  \label{def:poisedLatticeGenProblem}
  Denote the \emph{$\Dim$-dimensional cube of size $n+1$} as 
  \begin{equation}
    \label{eq:DcubeN}
    \mathbb{Z}_n^{\Dim} := (\mathbb{Z}_n)^{\Dim} = \{0,1,\ldots,n\}^{\Dim}
  \end{equation}
  and define the \emph{set of all triangular lattices
   of degree $n$} in $\mathbb{Z}_n^{\Dim}$ as 
   \begin{equation}
     \label{eq:setOfTrigLats}
     {\mathcal X} := \left\{{\mathcal T}^{\Dim}_n:
      {\mathcal T}^{\Dim}_n\subset \mathbb{Z}_n^{\Dim}\right\}.
   \end{equation}
  For a set of feasible nodes $K\subseteq \mathbb{Z}_n^{\Dim}$
  and a starting point $\mathbf{q}\in K$,
  the \emph{triangular-lattice generation (TLG) problem}
  seeks ${\mathcal T}\subseteq K$
  such that $\mathbf{q}\in {\mathcal T}$ and
  ${\mathcal T}\in {\mathcal X}$. 
\end{defn}

The above formulation 
 simplifies the PLG problem in Definition \ref{def:PLG}
 in that (i) the desirable poised lattice is restricted to the type of
 triangular lattices, 
 and (ii) only the \emph{numbering} of the coordinates
 in $\mathbb{Z}_n^{\Dim}$ needs to be determined.

\subsection{A group action of $\Dim$-permutations
  on triangular lattices}
\label{sec:group-action-trig-latt}

\begin{defn}
  \label{def:principalLatticeGeneral}
  The \emph{principal lattice} of degree $n$ in $\mathbb{Z}^{\Dim}$
   is 
  \begin{equation}
    \label{eq:principalLatticeGeneral}
    {\mathcal P}_n^{\Dim} := \left\{
      (j_1,\ldots, j_{\Dim})\in \mathbb{Z}^{\Dim}:
      \forall k=1,2,\cdots, \Dim,\ j_k\ge 0;
      \sum_{k=1}^{\Dim}j_k\le n\right\}.
  \end{equation}
  \revise{When there is no danger of confusion,
    we sometimes write ${\cal P}$ for ${\mathcal P}_n^{\Dim}$}. 
\end{defn}

As an example,
 the principal lattice of degree 2 in two dimensions is
 \begin{equation}
   \label{eq:principalLatD2n2}
   {\mathcal P}_{2}^2 = \{(0,0),(0,1),(0,2),(1,0),(1,1),(2,0)\}.
 \end{equation}
 
\begin{defn}
  \label{def:slice}
  The \emph{$m$-slice of a subset} $T\subseteq \mathbb{Z}_n^{\Dim}$
  across the $i$th dimension is a subset of $T$ defined as
  \begin{equation}
    \label{eq:slice}
    L_{i,m}(T) = \{\mathbf{p}\in T:
    p_i = m\}.
  \end{equation}
\end{defn}

As in Definition \ref{def:triangularLatticeDimD},
the $n+1$ coordinates $0, 1, \ldots, n$ in $\mathbb{Z}_n$
are given an ordering for each dimension $i$
and $p_{i,m}$ represents the $m$th coordinates
along the $i$th dimension.
In particular, $p_{i,m}$ may or may not equal $m$.

\begin{lem}
  \label{lem:sliceOfTriangularLattice}
  Any $p_{i,m}$-slice of a triangular lattice of degree $n$
  in $\Dim$ dimensions
  is a triangular lattice of degree $n-m$ 
  in $\Dim-1$ dimensions.
\end{lem}
\begin{proof}
  By Definition \ref{def:triangularLatticeDimD},
   the triangular lattice is
    \begin{displaymath}
    {\mathcal T}_n^{\Dim} = \left\{\left(p_{1,j_1}, p_{2,j_2}, \ldots,
        p_{D,j_D}\right):\ 
     j_k\ge 0;\ \sum_{k=1}^\Dim j_k\le n\right\},
  \end{displaymath}
  where each $p_i$
  has exactly $n+1$ coordinates. 
  By Definition \ref{def:slice}, we have
   \begin{displaymath}
     L_{i,p_{i,m}}({\mathcal T}) =
     \left\{ 
     \begin{array}{l}
       \left(p_{1,j_1},\, \ldots,\, p_{i-1,j_{i-1}},\, p_{i,m},\,
       p_{i+1,j_{i+1}},\, \ldots,\, 
       p_{D,j_D}\right):\ 
       \\
       \quad j_k\ge 0;\ \sum_{k\ne i, k=1}^{\Dim} j_k\le n-m
     \end{array}
     \right\},
  \end{displaymath}
  which, by Definition \ref{def:triangularLatticeDimD},
  is a triangular lattice of degree $n-m$ in $\Dim-1$ dimensions
  after renumbering the $n-m+1$ coordinates
  for each $k\ne i$.
\end{proof}


\begin{defn}
  \label{def:DPermutation}
  A \emph{$\Dim$-permutation of degree $n$}, written 
  \begin{displaymath}
    A=(a_1,a_2,\ldots,a_{\Dim})^T, 
  \end{displaymath}
   is a map $A: \mathbb{Z}_n^{\Dim} \rightarrow \mathbb{Z}_n^{\Dim}$ defined as
  \begin{equation}
    \label{eq:DPermutation}
    A \mathbf{p} =
    \left(a_1(p_1), a_2(p_2), \cdots, a_{\Dim}(p_{\Dim})\right)^T, 
  \end{equation}
  where each
  \mbox{$a_i: \mathbb{Z}_n \rightarrow \mathbb{Z}_n$} is a permutation.
\end{defn}

\begin{ntn}
  \label{ntn:D-permutations}
  $G_n^{\Dim}$ denotes 
  the \emph{set of all $\Dim$-permutations of degree $n$}.
\end{ntn}

\begin{defn}
  \label{def:multDpermut}
  The \emph{multiplication of two $\Dim$-permutations}
  is a binary operation $\cdot: G_n^{\Dim}\times G_n^{\Dim}\rightarrow G_n^{\Dim}$ given by
  \begin{equation}
    \label{eq:multDpermut}
    A\cdot B = (a_1\circ b_1, a_2\circ b_2, \ldots,
    a_{\Dim}\circ b_{\Dim})^T, 
  \end{equation}
  where ``$\circ$'' denotes function composition.
\end{defn}

\begin{defn}
  \label{def:DpermutInverse}
  The \emph{inverse of a $\Dim$-permutation} $A$
  is a unitary operation $^{-1}: G_n^{\Dim}\rightarrow G_n^{\Dim}$
  such that $A^{-1}$ satisfies
  \begin{equation}
    \label{eq:DpermutInverse}
    A^{-1}\cdot A = A\cdot A^{-1} = E = (e_1, e_2, \ldots, e_{\Dim})^T,
  \end{equation}
  where $E$ denotes the distinguished $\Dim$-permutation
  with each constituting permutation $e_i$
  as the identity map on $\mathbb{Z}_n$.
\end{defn}

Definition \ref{def:DpermutInverse} is well defined
because of (\ref{eq:multDpermut})
and the condition that each constituting permutation
is a bijection.

\begin{lem}
  \label{lem:permutationGroup}
  The following algebra, c.f. Notation \ref{ntn:D-permutations},
  is a group,
  \begin{equation}
    \label{eq:permutationGroup}
    (G_n^{\Dim}, \cdot, ^{-1}, E).
  \end{equation}
\end{lem}
\begin{proof}
  It is straightforward to verify the conditions
  of a group in Definition \ref{def:group} from Definitions 
  \ref{def:DPermutation}, \ref{def:multDpermut},
  and \ref{def:DpermutInverse}.
\end{proof}


\begin{lem}
  \label{lem:permutationOnTrigLats}
  $\Dim$-permutations map triangular lattices
  to triangular lattices.
  More precisely, for any $A\in G_n^{\Dim}$,
  the map $\sigma_A$ given by
\revise{
  \begin{equation}
    \label{eq:permutationOnTrigLats}
    \forall {\mathcal T}\in {\mathcal X}, \
    \sigma_A({\mathcal T}) = A {\mathcal T}
    := \bigl\{ A \mathbf{p}: \mathbf{p}\in {\mathcal T} \bigr\}
  \end{equation}
}
  is a permutation of ${\mathcal X}$.
\end{lem}
\begin{proof}
  Since ${\mathcal T}$ is a triangular lattice,
   we know from Definition \ref{def:triangularLatticeDimD}
   that there exist $n+1$ distinct coordinates for each dimension
   such that
   \emph{indices} of the $\Dim$ constituting coordinates
   of each point $\textbf{p}\in {\mathcal T}$
   add up to no more than $n$.
  The action of $A$ upon ${\mathcal T}$ in (\ref{eq:DPermutation})
   can be undone
   by applying $A^{-1}$;
   this means that for $A{\mathcal T}$ there exists a renumbering
   (specified by $A^{-1}$)
   of the coordinates along each axis
   such that $A{\mathcal T}$ is a triangular lattice. 
\end{proof}

\begin{lem}
  \label{lem:groupActionOnTrigLats}
  The set of triangular lattices ${\mathcal X}$ is a $G_n^{\Dim}$-set with
  its group action \mbox{$G_n^{\Dim}\times {\mathcal X} \rightarrow {\mathcal X}$}
  given by $\sigma_A({\mathcal T})$ in (\ref{eq:permutationOnTrigLats}).
\end{lem}
\begin{proof}
  By Lemma \ref{lem:permutationOnTrigLats}, 
  $\bullet(A,{\mathcal T})=\sigma_A({\mathcal T})$
  indeed has the signature $G_n^{\Dim}\times {\mathcal X} \rightarrow {\mathcal X}$.
  By (\ref{eq:DpermutInverse}), 
   we have
  \begin{displaymath}
    \forall {\mathcal T}\in {\mathcal X},\ \ 
    E{\mathcal T} = {\mathcal T}.
  \end{displaymath}
  In addition,
  for any $A,B\in G_n^{\Dim}$ and any ${\mathcal T}\in {\mathcal X}$,
  we have
  \begin{align*}
    (A\cdot B){\mathcal T}
    &= \bigl\{ (A\cdot B)\mathbf{p}: 
      \mathbf{p}\in {\mathcal T} \bigr\}
    = \bigl\{
      \left((a_1\circ b_1)(p_1), \ldots, (a_{\Dim}\circ b_{\Dim})(p_{\Dim})
      \right)^T: \mathbf{p}\in {\mathcal T} \bigr\}
    \\
    &= \bigl\{
      \left(a_1( b_1(p_1)), \ldots, a_{\Dim} (b_{\Dim}(p_{\Dim}))
      \right)^T: \mathbf{p}\in {\mathcal T} \bigr\}
    = A(B{\mathcal T}),
  \end{align*}
  where the first step follows from (\ref{eq:permutationOnTrigLats}), 
  the second from (\ref{eq:DPermutation}) and (\ref{eq:multDpermut}),
  and the last from (\ref{eq:permutationOnTrigLats}).
  Then the proof is completed by Definition \ref{def:groupAction}. 
\end{proof}

\begin{defn}
  \label{def:restorationOp}
  The \emph{restoration of a triangular lattice}
  ${\mathcal T}_n^{\Dim}$ is a $\Dim$-permutation
  $R_{{\mathcal T}}=(r_1,r_2, \ldots, r_{\Dim})^T$
  such that 
  \begin{equation}
    \label{eq:restorationOp}
    \forall i=1, 2, \ldots, \Dim, \     \forall m \in \mathbb{Z}_{n},\ 
    r_i(m) = \# \left\{
    j\in \mathbb{Z}_n: \#L_{i,j}({\mathcal T}) > \#L_{i,m}({\mathcal T})
    \right \}.
  \end{equation}
\end{defn}

By Lemma \ref{lem:sliceOfTriangularLattice}
 and Lemma \ref{lem:cardTriangularLattice}, 
 the slices of ${\mathcal T}$ along each axis
 have pairwise distinct cardinalities.
Hence (\ref{eq:restorationOp}) is well defined;
 see Example \ref{exm:restorationD2n2}. 

\begin{lem}
  \label{lem:restorationOpIsBijective}
  The image of a triangular lattice ${\mathcal T}_n^{\Dim}$
  under its restoration 
  is 
  the principal lattice ${\mathcal P}_n^{\Dim}$, i.e. 
  \begin{equation}
    \label{eq:restorationOpIsBijective}
    R_{{\mathcal T}_n^{\Dim}} {\mathcal T}_n^{\Dim} = {\mathcal P}_n^{\Dim}. 
  \end{equation}
\end{lem}
\begin{proof}
  By Definition \ref{def:restorationOp},
   the cardinalities of rearranged slices
   along the $i$th dimension are not changed by
   any constituting permutations except $r_i$.
  By Lemma \ref{lem:permutationOnTrigLats}, 
   $R_{{\mathcal T}} {\mathcal T}$ is also a triangular lattice.
  Furthermore, 
   cardinalities of the $m$-slices of ${R_{{\mathcal T}} \mathcal T}$
   decrease strictly monotonically as $m$ increases.
  There is only one such triangular lattice,
  namely ${\mathcal P}_n^{\Dim}$ in (\ref{eq:principalLatticeGeneral}).
\end{proof}

As a $\Dim$-permutation, a restoration is a bijection 
 and hence it has an inverse.

\begin{defn}
  \label{def:formationOp}
  The \emph{formation of a triangular lattice}
  ${\mathcal T}$
  is the inverse of its restoration, i.e., 
  \begin{equation}
    \label{eq:formationOp}
    A_{{\mathcal T}} = R_{{\mathcal T}}^{-1}.
  \end{equation}
\end{defn}

\begin{lem}
  \label{lem:formationOpIsBijective}
  The formation of a triangular lattice ${\mathcal T}_n^{\Dim}$
  maps the principal lattice ${\mathcal P}_n^{\Dim}$
  to ${\mathcal T}_n^{\Dim}$, 
  \begin{equation}
    \label{eq:formationOpIsBijective}
    {\mathcal T}_n^{\Dim} = A_{{\mathcal T}_n^{\Dim}} {\mathcal P}_n^{\Dim}. 
  \end{equation}
\end{lem}
\begin{proof}
  This follows directly from Lemma \ref{lem:restorationOpIsBijective}
  and Definitions \ref{def:DpermutInverse} and \ref{def:formationOp}.
\end{proof}



\begin{exm}
  \label{exm:restorationD2n2}
  For the triangular lattice ${\mathcal T}_2^2$
  in Example \ref{exm:latticeD2n2},
  its formation $A_{{\mathcal T}_2^2}$ equals its restoration $R_{{\mathcal T}_2^2}$,
   \begin{equation}
     \label{eq:normalizationT22}
     \left\{
       \begin{array}{l}
         r_1: 0\mapsto 1; 1\mapsto 0; 2\mapsto 2,
         \\
         r_2: 0\mapsto 0; 1\mapsto 2; 2\mapsto 1.
       \end{array}
     \right.
   \end{equation}
  The processes of restoration and formation are shown below.
   \begin{center}
     \includegraphics[width=0.5\textwidth]{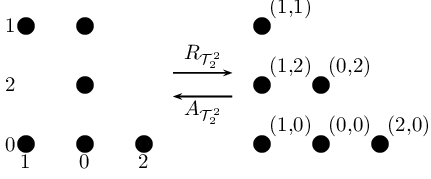}
   \end{center}
   On the left,
   the integers below the lattice are $r_1([0, 1, 2])$,
   i.e., the numbers of vertical slices whose cardinalities
   are larger than the current slice; 
   the integers to the left of the lattice
   are $r_2([0, 1, 2])$,
   i.e., the numbers of horizontal slices whose cardinalities
   are larger than the current slice.
   On the right,
   the multiindex close to a dot is the preimage
   of the dot under $R_{{\mathcal T}_2^2}$,
   whose behavior is summarized as follows.
  \begin{center}
    \begin{tabular}{c|c|c|c|c|c|c}
      \hline
      $R_{{\cal T}_2^2} \mathbf{p}\in {\mathcal P}_{2}^2$
      & (0,0) & (0,1) & (0,2) & (1,0) & (1,1) & (2,0)
      \\
      \hline
      $\mathbf{p}\in {\mathcal T}_{2}^2$
      & (1,0) & (1,2) & (1,1) & (0,0) & (0,2) & (2,0)
      \\
      \hline
    \end{tabular}
  \end{center}
\end{exm}

\begin{lem}
  \label{lem:bijectionfromG2X}
  Let ${\mathcal P}$ be the principal lattice of ${\mathcal X}$.
  The map $\varphi: G_n^{\Dim} \rightarrow {\mathcal X}$
  given by $\varphi(A_{{\mathcal T}}) := A_{{\mathcal T}} {\mathcal P}$
  is bijective. 
\end{lem}
\begin{proof}
  It can be proved from Definition \ref{def:DPermutation}
  and Theorem \ref{thm:permutationIsDisjointCycles} that 
   $A_1\ne A_2$ implies
   $A_1{\mathcal P}\ne A_2{\mathcal P}$, 
   hence $\varphi$ is injective.
  By Lemma \ref{lem:formationOpIsBijective}, 
   the formation $A_{\mathcal T}$ of any ${\mathcal T}\in {\mathcal X}$
   satisfies $\varphi(A_{{\mathcal T}})
   = A_{{\mathcal T}} {\mathcal P} = {\mathcal T}$.
  Hence $\varphi$ is surjective.
\end{proof}

\begin{thm}
  \label{thm:binaryOpOnX}
  The group $(G_n^{\Dim}, \cdot)$ is isomorphic to
  $( {\mathcal X}, *)$ with the binary operation $*$ on ${\mathcal X}$
  given by
  \begin{equation}
    \label{eq:binaryOpOnX}
    {\mathcal T}*{\mathcal S} := \varphi(\varphi^{-1}({\mathcal
      T})\cdot \varphi^{-1}({\mathcal S}))
  \end{equation}
  where $\varphi$ is the bijection defined in Lemma
  \ref{lem:bijectionfromG2X}. 
\end{thm}
\begin{proof}
  This follows from 
  Definitions \ref{def:isomorphicGroups}
  and \ref{def:multDpermut}, 
  Lemma \ref{lem:bijectionfromG2X}, and (\ref{eq:binaryOpOnX}). 
\end{proof}

\begin{coro}
  \label{coro:permutationMonomorphism}
  For the symmetric group $S_{\mathcal X}$ on ${\mathcal X}$, 
  the map $\phi: G_n^{\Dim}\rightarrow S_{\mathcal X}$ given by
  $\sigma_A$ in (\ref{eq:permutationOnTrigLats}) as 
  \begin{equation}
    \label{eq:permutationMonomorphism}
    \phi(A) = \sigma_A
  \end{equation}
  is a monomorphism.
\end{coro}
\begin{proof}
  By definition, we have 
  $\phi= \bar{i}\circ c \circ \varphi$, i.e.,
  $\phi$ is the chain of maps
  \begin{equation}
    \label{eq:chainOfMaps}
    \phi: G_n^D \overset{\varphi}{\longrightarrow} {\mathcal X} \overset{c}{\longrightarrow} 
    {\mathcal C}_{\mathcal X} \overset{\bar{i}}{\longhookrightarrow}
    {\mathcal S}_{\mathcal X}, 
  \end{equation}
  where ${\mathcal C}_{\mathcal X}$ is the subgroup
  of ${\mathcal S}_{\mathcal X}$
  isomorphic to ${\mathcal X}$, c.f. Cayley's theorem,  
  $c$ is the corresponding isomorphism,
  and $\bar{i}$ is the inclusion.
  By Definition \ref{def:isomorphicGroups}, 
  $\phi= \bar{i}\circ c \circ \varphi$ is a monomorphism
  because $\varphi$ and $c$ are isomorphisms
  and $\bar{i}$ is a monomorphism.
  Since
  \revise{$ \left|G_n^D\right| = [(n+1)!]^{\Dim} <
    |{\mathcal S}_{\mathcal X}| =
    \left([(n+1)!]^{\Dim}\right)!$},
  $\phi$ is not an epimorphism. 
\end{proof}

\begin{coro}
  \label{coro:onePairDetermines}
  Each preimage-image pair of triangular lattices
  uniquely determines a $\Dim$-permutation $A\in G_n^{\Dim}$ 
  and the corresponding $\sigma_A\in {\mathcal S}_{\mathcal X}$ as in
  (\ref{eq:permutationMonomorphism})
  is also uniquely determined.
\end{coro}
\begin{proof}
\revise{
Consider a permutation $\sigma_A: {\mathcal X}\rightarrow {\mathcal X}$ that maps
${\mathcal T}_1$  to ${\mathcal T}_2$.
Denote by $R_{{\mathcal T}_1}$ and $A_{{\mathcal T}_2}$
the restoration of ${\mathcal T}_1$
and the formation of ${\mathcal T}_2$, respectively, 
Lemma \ref{lem:permutationGroup} implies that
$A=A_{{\mathcal T}_2}\cdot R_{{\mathcal T}_1}$
is in $G_n^{\Dim}$ and 
\begin{displaymath}
  \phi(A){\mathcal T}_1 = \sigma_A {\mathcal T}_1
  = (A_{{\mathcal T}_2}\cdot R_{{\mathcal T}_1}){\mathcal T}_1
  = A_{{\mathcal T}_2}(R_{{\mathcal T}_1}{\mathcal T}_1)
  = A_{{\mathcal T}_2}{\mathcal P} = {\mathcal T}_2,
\end{displaymath}
where the first step follows from (\ref{eq:permutationMonomorphism}),
the second from (\ref{eq:permutationOnTrigLats}),
the third from Lemma \ref{lem:groupActionOnTrigLats},
the fourth from Lemma \ref{lem:restorationOpIsBijective}, 
and the fifth from Lemma \ref{lem:formationOpIsBijective}.
By Theorem \ref{thm:binaryOpOnX}, ${\cal T}_1$ and ${\cal T}_2$ uniquely determine
$R_{{\mathcal T}_1}$ and $A_{{\mathcal T}_2}$,
which uniquely determine $A$, 
which, by Corollary \ref{coro:permutationMonomorphism},
uniquely determines $\sigma_A$.
}
\end{proof}


\begin{coro}
  \label{coro:suffConditionsForTLG}
  The TLG problem $(K,\mathbf{q})$
   in Definition \ref{def:poisedLatticeGenProblem}
   is solved by a triangular lattice ${\mathcal T}$
   if and only if its formation $A_{{\mathcal T}}$
   satisfies
  \begin{subequations}
    \label{eq:conditionOfPrincipalLatForTLG}
    \begin{align}
      \exists \mathbf{p}\in {\mathcal P}_n^{\Dim},
      \text{ s.t. } A_{\mathcal T} \mathbf{p} = \mathbf{q}; 
      \\
      \forall \mathbf{p}\in {\mathcal P}_n^{\Dim},\ 
      A_{\mathcal T}\mathbf{p} \in K.
    \end{align}
   Furthermore, all solutions to the TLG problem 
   can be obtained by enumerating the formations.
  \end{subequations}
\end{coro}
\begin{proof}
  The first statement follows
   from Lemma \ref{lem:formationOpIsBijective}
  while the second from Theorem \ref{thm:binaryOpOnX}.
\end{proof}

The significance of Corollary \ref{coro:suffConditionsForTLG}
 is that, instead of enumerating
 all triangular lattices in ${\mathcal X}$,
 we can obtain
 \emph{all} solutions to the TLG problem
 by enumerating the $\Dim$-permutations in $G_n^{\Dim}$,
 whose simple structure is very amenable
 to the enumeration.

\begin{ntn}
  \label{ntn:DPermutationMatrix}
 In matrix notation, a $\Dim$-permutation is
  \begin{equation}
    \label{eq:DPermutationMatrix}
    A=
    \begin{bmatrix}
      a_1(0) & a_1(1) & \ldots & a_1(n)
      \\
      a_2(0) & a_2(1) & \ldots & a_2(n)
      \\
      \vdots & \vdots & \ddots & \vdots
      \\
      a_{\Dim}(0) & a_{\Dim}(1) & \ldots & a_{\Dim}(n)
    \end{bmatrix},
  \end{equation}
  where the row index $i$ of the matrix
  starts from 1 while the column index $j$ starts from 0.
  In other words, the $(i,j)$th-element of $A$ is
  \begin{equation}
    \label{eq:DPermutationMatrixSimple}
    A(i,j) = a_i(j).
  \end{equation}
\end{ntn}

\subsection{Partitioning the principal lattice and $\Dim$-permutations}
\label{sec:part-princ-latt}

When enumerating $\Dim$-permutations
 and checking them against (\ref{eq:conditionOfPrincipalLatForTLG}),
 there are two extremes.
For a given $\Dim$-permutation, 
 we could check the ${n+\Dim \choose \Dim}=\frac{(n+\Dim)!}{n!\Dim!}$
 conditions one at a time:
 we deny the $\Dim$-permutation upon the first time it fails any check;
 otherwise we retain it as a solution.
On the other hand,
 we could check all conditions in (\ref{eq:conditionOfPrincipalLatForTLG})
 \emph{simultaneously}.
For example,
 we could represent a subset $P$ of $\mathbb{Z}_n^{\Dim}$
 by a binary number,
 each bit of which corresponds to a point
 $\mathbf{p}\in\mathbb{Z}_n^{\Dim}$,
 with $1$ meaning $\mathbf{p}\in P$ and $0$ meaning $\mathbf{p}\not\in P$;
 then checking all conditions in (\ref{eq:conditionOfPrincipalLatForTLG})
 reduces to a single calculation of the logical conjunction of two binary numbers
 that represent $K$ and the triangular lattice $A_{\mathcal T}{\mathcal P}$.
 
In the former method
 the number of steps $\frac{(n+\Dim)!}{n!\Dim!}$
 is unnecessarily large while 
 in the latter much computation is repeated
 for similar $\Dim$-permutations.
To strike a balance,
 we group the points of a principal lattice
 into $\Dim(n+1)$ equivalence classes, 
 and at each stage we check (\ref{eq:conditionOfPrincipalLatForTLG})
 for points in a single equivalence class.
This helps to preclude at early stages
 those $\Dim$-permutations that
 do not satisfy (\ref{eq:conditionOfPrincipalLatForTLG}b).

 \begin{defn}
   \label{def:testSet}
   A \emph{test set} $W_{(\ell,m)}$ is an equivalence class of points
   in the principal lattice, 
   \begin{equation}
     \label{eq:testSet}
      \bigcup_{(\ell,m)\in \mathbb{Z}_{\Dim}^+ \times \mathbb{Z}_n}
      W_{(\ell,m)} = {\mathcal P}_n^{\Dim}.
   \end{equation}
 \end{defn}


\begin{defn}
  \label{def:partialDpermut}
  A \emph{partial function} from $Y$ to $Z$
  is a function $Y'\rightarrow Z$ on some $Y'\subseteq Y$.
  The \emph{$(\ell,m)$-partial $\Dim$-permutation}, 
  denoted by $A^{(\ell,m)}$,
  is a partial function on the test set $W_{(\ell,m)}$
  that satisfies 
  \begin{equation}
    \label{eq:partialDpermut}
    \forall \mathbf{p}\in W_{(\ell,m)},
    \ A^{(\ell,m)}\mathbf{p} = A\mathbf{p}.
  \end{equation}  
\end{defn}

 Although $A^{(\ell,m)}\ne A$,
  their actions on $W_{(\ell,m)}$ are exactly the same.
 Instead of enumerating all $A\in G_n^{\Dim}$,
  we enumerate all $A^{(\ell,m)}$'s by adding one number at a time
  to the $\Dim$-by-$(n+1)$ grid in (\ref{eq:DPermutationMatrix}), 
  so that they are naturally organized into a spanning tree.
 The parent-child relation in this spanning tree
  and the incremental construction of $\Dim$-permutations
  can be made precise by a linear ordering of pairs
  such as the one below. 

\begin{defn}
  \label{def:linearOrderingOfGrid}  
  The \emph{linear ordering of integer pairs}
   on a grid $\mathbb{Z}_{\Dim}^+\times \mathbb{Z}_n$
   is the column-wise ordering of the grid,
   i.e., the total ordering obtained by
   stacking all the columns of the grid
   into one column.
  More precisely, 
   this ordering is a bijection $s$
   that maps a pair $(i,j)\in \mathbb{Z}_{\Dim}^+\times \mathbb{Z}_n$
   to a scalar index $k\in \mathbb{Z}_{\Dim(n+1)}^+$,
  \begin{align}
    \label{eq:mid2id}
    s(i,j) &= i + j\Dim; 
    \\
    \label{eq:id2mid}
    s^{-1}(k) &= \left(1 + (k-1) \text{ mod } \Dim,\
           \left\lfloor \frac{k-1}{\Dim}\right\rfloor\right),
  \end{align}
  where $\lfloor\cdot\rfloor: \mathbb{Q}\rightarrow \mathbb{N}$
  is the floor function.
\end{defn}

Then the linear ordering of the $\Dim$-permutations $A^{(\ell,m)}$'s is 
 \begin{equation}
   \label{eq:partialPermutScalar}
   \forall t=s(\ell,m),\
   A^{(t)}:= A^{(\ell,m)},
 \end{equation}
 where $A^{(t)}$ is interpreted as a matrix in
 (\ref{eq:DPermutationMatrix}). 
 Clearly we have $A^{(\Dim(n+1))}=A$
 and $A^{(0)}(i,j)=-1$.


\begin{lem}
  \label{lem:triangularLatPartialPartition}
  A triangular lattice ${\mathcal T}_n^{\Dim}$
  has the partition
  \begin{equation}
    \label{eq:triangularLatPartialPartition}
    {\mathcal T}_n^{\Dim} =
    \bigcup_{(\ell,m)\in \mathbb{Z}_{\Dim}^+ \times \mathbb{Z}_n}
    A_{\mathcal T}^{(\ell,m)} W_{(\ell,m)},
  \end{equation}
   where 
   the terms $A_{\mathcal T}^{(\ell,m)} W_{(\ell,m)}$
   are pairwise disjoint.
\end{lem}
\begin{proof}
  By Lemma \ref{lem:formationOpIsBijective}
  and Definitions \ref{def:testSet} and \ref{def:partialDpermut}, 
  we have
  \begin{displaymath}
    {\mathcal T}_n^{\Dim} 
    = A_{\mathcal T}\bigcup_{(\ell,m)}W_{(\ell,m)}
    = \bigcup_{(\ell,m)}A_{\mathcal T}W_{(\ell,m)}
    = \bigcup_{(\ell,m)}A_{\mathcal T}^{(\ell,m)}W_{(\ell,m)},
  \end{displaymath}
  where the pairwise disjointness 
  follows from the formation being a bijection.
\end{proof}

\begin{thm}
  \label{thm:incrementalConditionsForTLG}
  The TLG problem $(K, \mathbf{q})$
   in Definition \ref{def:poisedLatticeGenProblem}
   is solved by a triangular lattice ${\mathcal T}=A_{\mathcal T}{\mathcal P}$
   if and only if its formation $A_{{\mathcal T}}$ satisfies
  \begin{subequations}
    \label{eq:incrementalConditionsForTLG}
    \begin{align}
      \forall (\ell,m)\in \mathbb{Z}_{\Dim}^+\times \mathbb{Z}_n,\ 
      A_{\mathcal T}^{(\ell,m)} W_{(\ell,m)} \subset K;
      \\
      \exists \mathbf{p}\in {\mathcal P}, 
      \text{ s.t. } A_{\mathcal T}\mathbf{p} = \mathbf{q},
    \end{align}
  \end{subequations}
  where the test sets $W_{(\ell,m)}$'s
  form a partition of ${\mathcal P}$.
  Furthermore, an enumeration based on the 
  partial $\Dim$-permutations finds all solutions
  to the TLG problem.
\end{thm}
\begin{proof}
  This follows directly from Lemma
  \ref{lem:triangularLatPartialPartition}
  and Corollary \ref{coro:suffConditionsForTLG}.
\end{proof}

Theorem \ref{thm:incrementalConditionsForTLG}
 is the main theoretical foundation
 for our TLG algorithm.



\section{Algorithm}
\label{sec:algorithms}

In light of Corollary \ref{coro:suffConditionsForTLG}
 and Notation \ref{ntn:DPermutationMatrix}, 
 the key issue that affects the efficiency of a TLG algorithm
 is how to choose the equivalence classes
 in Definition \ref{def:testSet}. 
\revise{
  In Sections \ref{sec:test-sets-first} and \ref{sec:test-sets-second},
  we 
  define,
  for each $(\ell,m)\in \mathbb{Z}_{\Dim}^+\times \mathbb{Z}_{n}$,
  a subset $W_{(\ell,m)}$ of ${\cal P}_n^{\Dim}$ as the $(\ell,m)$th test set
  and then deduce directly from Definition \ref{def:partialDpermut}
  the corresponding partial $\Dim$-permutation $A^{(\ell,m)}$. 
Exploiting the total ordering $s$ in Definition \ref{def:linearOrderingOfGrid}
 on $\mathbb{Z}_{\Dim}^+\times \mathbb{Z}_{n}$,
 we give each pair $(\ell,m)$ the scalar index $s(\ell,m)$
 and determine all spanning trees
 in the solution space of partial $\Dim$-permutations. 
Based on the backtracking algorithm
 in Definition \ref{def:backtracking}, 
 we propose in Section \ref{sec:backtr-algor-plg} 
 the TLG algorithm that traverses the solution space
 to locate a plausible $\Dim$-permutation.
}
The importance of algebraic structures
 in designing the solution space 
 is emphasized in Section \ref{sec:tests}, 
 where the TLG algorithm based on test sets in Section \ref{sec:test-sets-second}
 is shown to be much more efficient
 than that based on Section \ref{sec:test-sets-first}. 
 
\subsection{Test sets of type I}
\label{sec:test-sets-first}


\revise{
\begin{defn}
  \label{def:testSet1}
  The $(\ell,m)$th \emph{test set of type I}
   is a subset of $\mathbb{Z}_n^{\Dim}$ given by
   $W_{(\ell,m)}:=\tilde{W}_{(\ell,m)}$
   if $(\ell,m)=(1,0)$
   and otherwise
   $W_{(\ell,m)}:=\tilde{W}_{(\ell,m)}\setminus \tilde{W}_{(i,j)}$
   where $s(i,j)=s(\ell,m)-1$, 
   $s$ is defined in (\ref{eq:mid2id}), and
   \begin{equation}
     \label{eq:testSetTypeIprimitive}
     \tilde{W}_{(\ell,m)} := \left\{
       \mathbf{p} \in {\cal P}_n^{\Dim}:\quad
         \begin{array}{l}
           i\le\ell \implies p_i\in \mathbb{Z}_m;\\
           i> \ell \implies p_i\in \mathbb{Z}_{m-1}
         \end{array}
     \right\}.
   \end{equation}
\end{defn}
 
By Notation \ref{ntn:firstNnaturalNumbers},
 $m=0$ yields $\mathbb{Z}_{m-1}=\emptyset$, 
 hence 
 $\tilde{W}_{(\ell,m)}=\emptyset$
 if and only if $\ell<\Dim$ and $m=0$. 
The conditions in (\ref{eq:testSetTypeIprimitive}) ensure
 that the sequence $(\tilde{W}_{(\ell,m)})_{s(\ell,m)=0}^{\Dim (n+1)}$
 be a filtration, i.e.,
\begin{displaymath}
  \emptyset = \tilde{W}_{(0,0)} \subseteq
  \tilde{W}_{(1,0)} \subseteq \cdots \subseteq
  \tilde{W}_{(\Dim,0)} \subseteq \tilde{W}_{(1,1)}
  \subseteq \cdots \subseteq
  \tilde{W}_{(\Dim,1)} \subseteq \cdots \subseteq
  \tilde{W}_{(\Dim,n)} = {\cal P},
\end{displaymath}
which implies that the test sets $\{W_{(\ell,m)}: (\ell,m)\in
\mathbb{Z}_{\Dim}^+\times \mathbb{Z}_n\}$
in Definition \ref{def:testSet1} 
indeed form a partition of the principal lattice,
c.f. Definition \ref{def:testSet}.
}

\begin{exm}
  \label{exm:testSetPartitionD2n2}
  For ${\mathcal P}_2^2$ in (\ref{eq:principalLatD2n2}),
  the test sets of type I
  are
  \begin{displaymath}
    \begin{array}{lll}
      W_{(1,0)}=\emptyset,
      & W_{(1,1)}=\{(1,0)\},
      & W_{(1,2)}=\{(2,0)\}, 
      \\
      W_{(2,0)}=\{(0,0)\},
      & W_{(2,1)}=\{(0,1),(1,1)\},
      & W_{(2,2)}=\{(0,2)\}.
    \end{array}
  \end{displaymath}
  The triangular lattice in Example \ref{exm:latticeD2n2} is
  \begin{equation}
    \label{eq:triangularLatD2n2Elements}
     {\mathcal T} = \{(0,0), (0,2), (1,0), (1,1), (1,2), (2,0)\}
  \end{equation}
  and its formation is
  \begin{equation}
    \label{eq:matrixOfLatticeD2n2}
    A_{{\mathcal T}} =
    \begin{bmatrix}
      1 & 0 & 2
      \\
      0 & 2 & 1
    \end{bmatrix}.
  \end{equation}
  In light of Lemma \ref{lem:triangularLatPartialPartition}, 
   ${\mathcal T}$ is partitioned into
  \begin{displaymath}
    \begin{array}{lll}
      A_{{\mathcal T}}^{(1,0)}W_{(1,0)}=\emptyset,\ 
      &A_{{\mathcal T}}^{(1,1)}W_{(1,1)}=\{(0,0)\},\
      &A_{{\mathcal T}}^{(1,2)}W_{(1,2)}=\{(2,0)\},
      \\
      A_{{\mathcal T}}^{(2,0)}W_{(2,0)}=\{(1,0)\},\ 
      &A_{{\mathcal T}}^{(2,1)}W_{(2,1)}=\{(1,2),(0,2)\},\ 
      &A_{{\mathcal T}}^{(2,2)}W_{(2,2)}=\{(1,1)\}.
    \end{array}
  \end{displaymath}
\end{exm}


\subsection{Test sets of type II}
\label{sec:test-sets-second}

\revise{
  The efficiency of test sets of type I
  can be improved
  by requiring that
  test sets to be checked early contain
  a larger number of points in the principal lattice; 
  then a failure of an early test set eliminates
  a larger subtree.
To this end, 
 we capture the structure
 of isotropy subgroups of $G_n^{\Dim}$ 
 in Theorem \ref{thm:DpermutAction}, 
 which guarantees the optimal subtree elimination 
 and gives rise to test sets of type II
 in Definition \ref{def:testSet2}. 

\begin{lem}
  \label{lem:groupActionOnTestSet}
  Let $H$ be a subgroup of a group $G$
  and let $Z$ be a subset of a $G$-set $Y$,
  c.f. Definition \ref{def:groupAction}. 
  Define a set 
  \begin{equation}
    \label{eq:main_def}
    W_{H,Z} := \left\{p \in Y : H p \subset Z \right\},
  \end{equation}
  where $Hp$ is the orbit of $p$ under $H$, 
  c.f. Definition \ref{def:orbitForGsets}. 
  Then we have
  \begin{equation}
    \label{eq:alg_foundation}
    \forall a \in G, \quad
    a W_{H,Z}  = \bigcap_{b \in aH} bZ,
  \end{equation}
  where $aH$ is the left coset of $H$, 
  c.f. Definition \ref{def:coset}. 
\label{prop:algebraic_foundation}
\end{lem}
\begin{proof}
  Since the signature of an action of $G$ on $Y$
  is $G\times Y \rightarrow Y$,
  any $a \in G$ can be considered as a permutation on $Y$.
  Multiply (\ref{eq:alg_foundation}) with $a^{-1}$,
  apply condition (ii) in Definition \ref{def:groupAction}, 
  and we deduce from $a^{-1}\cap_{b\in a H} bZ = \cap_{b\in H} bZ$
  that (\ref{eq:alg_foundation}) holds if and only if
  $W_{H,Z} = \bigcap_{b \in H} bZ$. 
  To show $W_{H,Z}$ is contained in $bZ$ for every $b \in H$, 
  suppose $p \in W_{H,Z}$. 
  Because the orbit of $p$ is contained in $Z$, 
  we have $b^{-1}p \in Z$ for any $b\in H$; 
  apply $b$ to both sides and we get $p \in bZ$. 
  Conversely, 
  suppose $p \in bZ$ for every $b \in H$; 
  apply $b^{-1}$ to both sides
  and we have $b^{-1}p \in Z$ for every $b \in H$. 
  Because $b^{-1}$ ranges over $H$, 
  the orbit of $p$ under $H$ is contained in $Z$. 
\end{proof}

Write a sequence $\Lambda$ of subsets of $\mathbb{Z}_n$ as
\begin{equation}
  \label{eq:D-SubsetsOfZn}
  \Lambda := (\Lambda_i)_{i=1}^{\Dim}
  = (\Lambda_1, \Lambda_2, \ldots, \Lambda_{\Dim})
\end{equation}
and denote by $H_{\Lambda}$ the subset of $G_n^{\Dim}$
that holds fixed the coordinates in $\Lambda$, i.e., 
 \begin{equation}
   \label{eq:H-Lambda}
  H_{\Lambda} := \left\{A \in G_n^{\Dim}:
    \forall i \in \mathbb{Z}_{\Dim}^+,\ \forall j\in \Lambda_i,\ 
    A(i,j) = j \right\}.
\end{equation}

By Lemma \ref{lem:isotropySubgroup},
 $H_{\Lambda}$ is the intersection of
 isotropy subgroups of $G_n^{\Dim}$
 at coordinates in $\Lambda$.
From Corollary \ref{coro:equalCosets}
 we readily deduce that
  two $\Dim$-permutations $A, B \in G_n^{\Dim}$
  are in the same coset of $H_{\Lambda}$
  if and only if 
  their actions agree on each $\Lambda_i$, i.e., 
  $\forall i\in \mathbb{Z}_{\Dim}^+,\  \forall j \in \Lambda_i,\ 
  A(i,j) = B(i,j)$.

What is 
 the orbit of a point $\mathbf{p} \in \mathbb{Z}_n^{\Dim}$
 under $H_{\Lambda}$?
For a fixed dimension $d\in \mathbb{Z}_{\Dim}^+$, 
 if $p_d \in \Lambda_d$, 
 all permutations in $H_{\Lambda}$ will hold $p_d$ fixed;
 otherwise $p_d \notin \Lambda_d$, 
 then $p_d$ will be mapped to one of $\mathbb{Z}_n \setminus \Lambda_d$. 
By Definition \ref{def:DPermutation},
 a $\Dim$-permutation consists of
 $\Dim$ independent 1-permutations,
 each of which, 
 by Theorem \ref{thm:permutationIsDisjointCycles},
 is simply a product of disjoint cycles.
Hence the orbit of $\mathbf{p}$ is
\begin{equation}
  \label{eq:orbitH}
  H_{\Lambda} \mathbf{p} = \prod_{i=1}^{\Dim} W^{\times}(p_i, \Lambda_i)
  \text{ where }
  W^{\times}(p_i, \Lambda_i) :=
  \begin{cases}
    \{p_i\} &\text{if } p_i \in \Lambda_i; \\
    \mathbb{Z}_n \setminus \Lambda_i &\text{otherwise}. 
  \end{cases}
\end{equation}

Lemma \ref{lem:groupActionOnTestSet} and 
 the above arguments prove 

\begin{thm}
  \label{thm:DpermutAction}
  For a sequence $\Lambda$ of subsets of $\mathbb{Z}_n$, 
  define
  \begin{equation}
    \label{eq:W-Lambda}
    W_{\Lambda,{\mathcal P}} := \left\{\mathbf{p} \in \mathbb{Z}_n^{\Dim}: 
     H_{\Lambda}\mathbf{p}  \subset {\mathcal P} \right\},
  \end{equation}
  where $H_{\Lambda}\mathbf{p}$ is the orbit in (\ref{eq:orbitH})
  and ${\mathcal P}$ the corresponding principal lattice
  in (\ref{eq:principalLatticeGeneral}).
  Then
  \begin{equation}
    \forall A\in G_{n}^{\Dim},\quad
    AW_{\Lambda,{\mathcal P}} =
    \bigcap_{B \in A H_{\Lambda}} B{\mathcal P},
  \end{equation}
  where the coset $A H_{\Lambda}$ is
  the set of all $\Dim$-permutations
  whose actions agree with that of $A$
  on $\prod_{i=1}^{\Dim}\Lambda_i$.
\end{thm}

Consider the two extreme cases of
$\Lambda_i \equiv \emptyset$
and $\Lambda_i \equiv \mathbb{Z}_n$ for each $i\in \mathbb{Z}_{\Dim}^+$.
In the former case,
 $\Lambda=\emptyset^{\Dim}$,
 $H_{\Lambda}=G_n^{\Dim}$,
 and no point in $\mathbb{Z}_n^{\Dim}$
 has an orbit under $H_{\Lambda}$
 that is contained in the principal lattice.
Hence (\ref{eq:W-Lambda}) yields
 $W_{\Lambda,{\mathcal P}} = \emptyset$.
In the latter case,
 $\Lambda\in \mathbb{Z}_n^{\Dim}$,
 $H_{\Lambda}$ is a singleton set containing the identity permutation,
 and each point $\mathbf{p}\in\mathbb{Z}_n^{\Dim}$
 has $\{\mathbf{p}\}$ as its orbit under $H_{\Lambda}$.
Hence (\ref{eq:W-Lambda}) yields
  $W_{\Lambda,{\mathcal P}} = {\mathcal P}$.
Between the two extremes, we have 
 \begin{displaymath}
   (\forall i\in \mathbb{Z}_{\Dim}^+,\ 
   \Lambda'_i \subset \Lambda_i )
   \implies 
   W_{\Lambda',{\mathcal P}} \subset W_{\Lambda,{\mathcal P}}
\end{displaymath}
because Lemma \ref{lem:isotropySubgroup}
implies that $H_{\Lambda}$ be a subgroup of $H_{\Lambda'}$.
As the number of fixed coordinates
 in $\Lambda$ increases from $0$ to $\Dim(n+1)$,
 $W_{\Lambda,{\mathcal P}}$ 
 grows from $\emptyset$ to ${\mathcal P}$.

The above discussion and
 the linear ordering in Definition \ref{def:linearOrderingOfGrid}
 suggest the construction of a filtration
 $(\Lambda^k)_{k=0}^{\Dim(n+1)}$ that satisfies 
\begin{displaymath}
  (\emptyset, \cdots, \emptyset)
  =\Lambda^0 \subset \Lambda^1 \subset \cdots \subset \Lambda^{\Dim(n+1)}
  = (\mathbb{Z}_n, \cdots, \mathbb{Z}_n), 
\end{displaymath}
 where each $\Lambda^k$ is a sequence in (\ref{eq:D-SubsetsOfZn})
 and $\Lambda^{k+1}$ has one and only one more fixed coordinate
 than $\Lambda^k$.
 The previous paragraph implies 
 $W_{\Lambda^{\Dim(n+1)},{\cal P}}={\cal P}$
 and the requirement of pairwise disjointness
 in Definition \ref{def:testSet} leads to

\begin{defn}
  \label{def:testSet2}
  The $(\ell,m)$th \emph{test set of type II} is a subset of $\mathbb{Z}_n^{\Dim}$,
  \begin{equation}
    \label{eq:testSet2}
    W_{(\ell,m)} := W(\Lambda^k)\setminus  W(\Lambda^{k-1}), 
  \end{equation}
  where $W(\Lambda^k):=W_{\Lambda^k, {\cal P}}$
  is defined in (\ref{eq:W-Lambda}),
  $k=s(\ell,m) = 1, 2, \ldots, \Dim(n+1)$ 
  as in Definition \ref{def:linearOrderingOfGrid},
  $\Lambda^0 := (\emptyset)_{1}^{\Dim}$,
  $\Lambda^k :=(\Lambda^k_i)_{i=1}^{\Dim}$, and
  \begin{equation}
    \label{eq:choiceOfLambdas}
    \Lambda_i^k :=
    \begin{cases}
      \mathbb{Z}_{m} & \text{if } i\le \ell; 
      \\
      \mathbb{Z}_{m-1} & \text{if } i> \ell. 
    \end{cases}
  \end{equation}
\end{defn}

In particular, $i>\ell$ and $m=0$
 implies $\Lambda_i^k=\mathbb{Z}_{m-1}=\emptyset$,  
 c.f. Notation \ref{ntn:firstNnaturalNumbers}.
The reader is invited to verify that (\ref{eq:choiceOfLambdas}) 
 indeed fulfills the requirement that
 one and only one coordinate is added
 from $\Lambda^k$ to $\Lambda^{k+1}$. 

\begin{exm}
  \label{exm:testSet2PartitionD2n2}
  For ${\mathcal P}_2^2$ in (\ref{eq:principalLatD2n2}), 
  the test sets of type II are
  \begin{displaymath}
    \begin{array}{lll}
      W_{(1,0)}=\{(0,0),(0,1),(0,2)\},
      &W_{(1,1)}=\emptyset, 
      &W_{(1,2)}=\emptyset, 
      \\
      W_{(2,0)}=\{(1,0),(2,0)\},
      &W_{(2,1)}=\{(1,1)\},
      &W_{(2,2)}=\emptyset.
    \end{array}
  \end{displaymath}
  We detail the process of calculating $W_{(\ell,m)}$
  only for $(\ell,m)=(1,0)$ since the others are similar.
  By Definition \ref{def:testSet2}, 
  we have $\Lambda^0 = (\emptyset, \emptyset)$
  and $\Lambda^1 = (\{0\},\emptyset)$.
  Then, for any $\mathbf{p}=(0,j)$ with $j=0,1,2$, 
  (\ref{eq:orbitH}) gives
  \begin{displaymath}
      W^{\times}(0, \Lambda_1^1)
      = \{0\}, \ 
      W^{\times}(j, \Lambda_2^1) = \mathbb{Z}_2;
      \ \ 
      H_{\Lambda}\mathbf{p} = \{(0,0),(0,1),(0,2)\}
                              \subset {\cal P}_2^2, 
  \end{displaymath}
  hence by (\ref{eq:W-Lambda}) we have
  $(0,j)\in W_{\Lambda^1,{\cal P}}$
  for each $j=0,1,2$.
  
  In comparison, for any $\mathbf{p}=(i,j)$ with $i=1,2$ and $j=0,1,2$, 
  (\ref{eq:orbitH}) gives
  \begin{displaymath}
    W^{\times}(i, \Lambda_1^1)
    = \{1,2\}, \ 
    W^{\times}(j, \Lambda_2^1) = \mathbb{Z}_2;\ \ 
    H_{\Lambda}\mathbf{p}
    = \{1,2\}\times\mathbb{Z}_2 \not \subset {\cal P}_2^2, 
  \end{displaymath}
  hence by (\ref{eq:W-Lambda}) we have
  $(i,j)\not\in W_{\Lambda^1,{\cal P}}$
  for each $(i,j)\in \{1,2\}\times \mathbb{Z}_2$.
  Finally, (\ref{eq:testSet2})
  yields $W_{(1,0)}= W(\Lambda^1)\setminus W(\Lambda^0)
  = \{0\}\times \mathbb{Z}_2$.
\end{exm}

It is emphasized that $n$ and $\Dim$
 completely determine test sets of both types I and II. 
Therefore, both test sets for a wide range of $(\Dim,n)$
 are computed once and for all TLG problems;
 in other words,
 test sets are not calculated {on the fly}
 \emph{at the runtime}
 but rather determined at the \emph{compile time}.

In light of Corollary \ref{coro:suffConditionsForTLG},
 we enumerate not the triangular lattices
 but the formations of triangular lattices;
 see Figure \ref{fig:backtrack_with_test_sets}. 
The root node of the solution space of $G_n^{\Dim}$
 corresponds to the state
 where none of the $\Dim(n+1)$ variables is determined.
At each step of spanning subtrees in the solution space,
 we append to each leaf node
 a set of acceptable $\Dim$-permutations, 
 each of which has one more element determined
 in its matrix representation in Notation
 \ref{ntn:DPermutationMatrix},
 with the index of this element
 given by the total ordering in Definition
 \ref{def:linearOrderingOfGrid}.
A particular case on how to determine the acceptability
 of the next element of the matrix of a $\Dim$-permutation
 is $W_{(\ell,m)}=\emptyset$,
 for which we have no points to check the feasibility
 and thus simply enumerate every possible
 value of $a(\ell,m)$
 in the set $\mathbb{Z}_n\setminus \{a(\ell,j): j\in \mathbb{Z}_{m-1}\}$;
 see, e.g., the first matrix in the third line
 in Figure \ref{fig:backtrack_with_test_sets} (b).
The above pattern is repeated recursively
 to span subtrees until
 we find a $\Dim$-permutation 
 such that all elements in its matrix representation
 are determined as acceptable.
Such a $\Dim$-permutation is the formation 
 of some triangular lattice 
 that solves the TLG problem in Definition \ref{def:poisedLatticeGenProblem}.

For the particular case of Examples \ref{exm:testSetPartitionD2n2}
 and \ref{exm:testSet2PartitionD2n2}, 
we compare the spanned subtrees
 for the two types of test sets
 in Figure \ref{fig:backtrack_with_test_sets}
 and observe the key advantage of test sets of type II
 over those of type I:
 test sets to be checked early contain a larger subtree
 and thus a failure of an early test set eliminates
 a larger number of $\Dim$-permutations in the solution space.
Consequently,
 the spanned subtrees for test sets of type II
 are much less complicated that those of type I.

This advantage of test sets of type II essentially comes from 
 the \emph{optimal} organization of the solution space
 via exploiting the orbital structure
 captured in Theorem \ref{thm:DpermutAction}: 
 the coset $A H_{\Lambda}$ contains
 \emph{all} $\Dim$-permutations
 whose actions agree with that of $A$
 on $\prod_{i=1}^{\Dim}\Lambda_i$.
Therefore, 
 \emph{if the current partial permutation $A^{(\ell,m)}$
 does not satisfy the condition
 (\ref{eq:conditionOfPrincipalLatForTLG}),
 then we eliminate \textbf{all} permutations
 whose restrictions to $W_{(\ell,m)}$ equal $A^{(\ell,m)}$.}

\small{
\begin{figure}
  \centering
  \subfigure[A TLG problem that is 
  solved by the triangular lattice in Example \ref{exm:latticeD2n2}.
  ]{
    \includegraphics[width=0.4\linewidth]{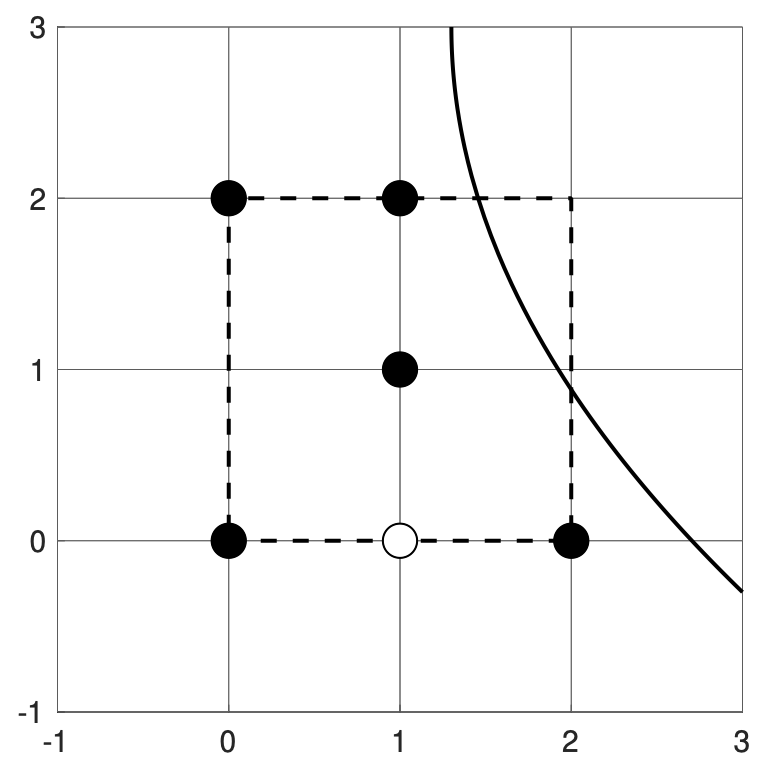}
  }
  \hfill
  \subfigure[Most subtrees in the solution space
  spanned by test sets of type II.
  ]{
    \includegraphics[width=0.5\textwidth]{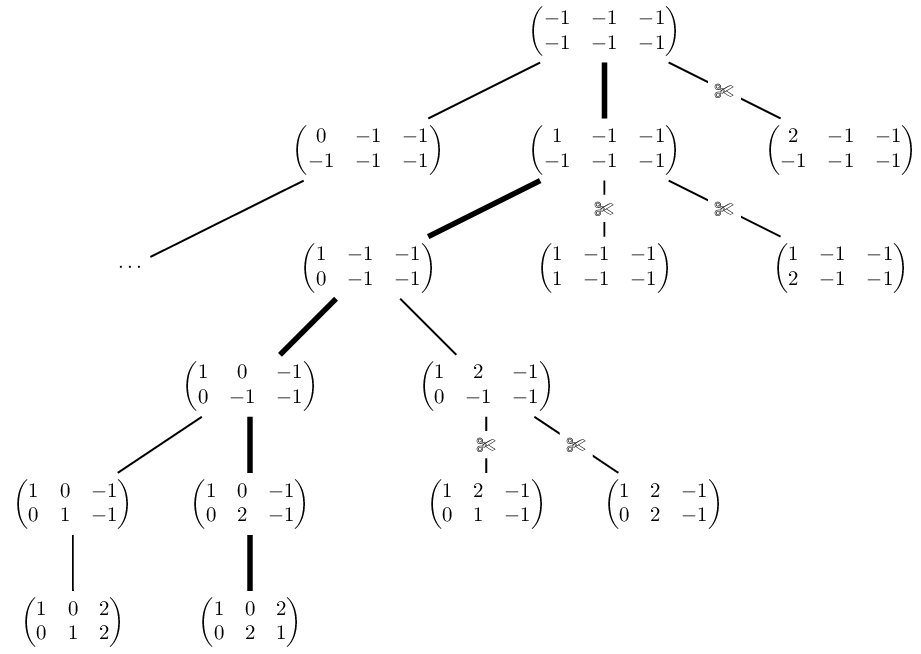}
  }
  \hfill
  \subfigure[the subtrees spanned by test sets of type I
  and corresponding to those in (b).
  ]{
    \includegraphics[width=0.8\textwidth]{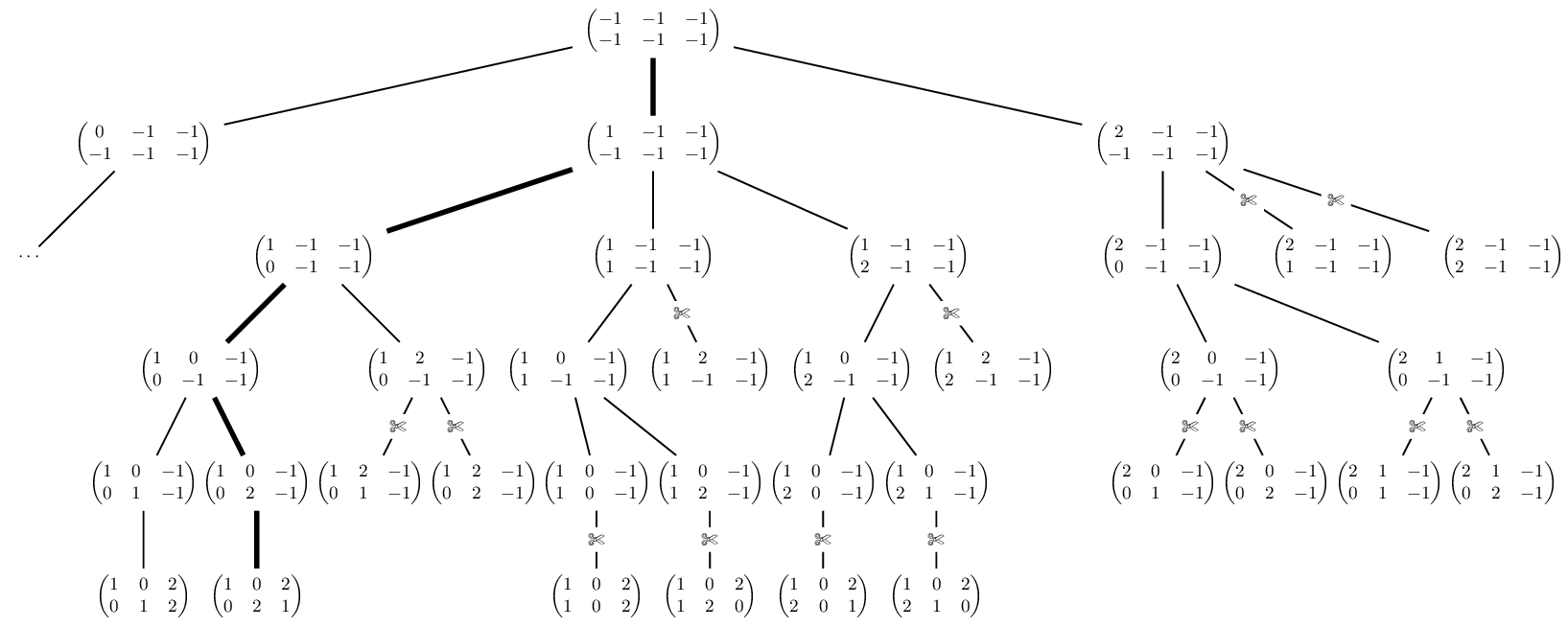}
  }
  \caption{An illustration of the efficiency advantage of test sets of
    type II over test sets of type I
    on spanning subtrees in the solution space.
    The TLG problem $(K,\mathbf{q})$ is shown in subplot (a), 
    with $\Dim=2$, $n=2$,
    $\mathbf{q}$ being the hollow dot at $(1,0)$,
    and the set of feasible nodes $K$ consisting of the grid nodes 
    within the dot-dashed box and to the left of the solid curve. 
    In both (b) and (c),
    a solid edge ended with a partial $\Dim$-permutation 
    represents a subtree spanning by an acceptable leaf node,
    a solid edge ended with ``$\cdots$'' 
    a subtree spanning with the subtree omitted,  
    a thick solid edge part of the path from the root node
    to the solution in subplot (a), 
    and an edge labeled with a pair of scissors a subtree pruning.
    The \emph{only criterion} for either accepting or rejecting $A^{(\ell,m)}$ is
    whether $A^{(\ell,m)}W_{(\ell,m)}\subset K$ holds or not,
    respectively.
    For example, 
    the last child of the root node in subplot (b)
    is rejected because $A^{(1,0)}W_{(1,0)}$
    with $A^{(1,0)}(1,0)=2$
    is not a subset of $K$.
    As another example,
    $W_{(1,1)}=\emptyset$ dictates that
    all immediate children of the first partial $\Dim$-permutation
    in the third row be enumerated,
    despite the fact that
    the second partial $\Dim$-permutation in the fourth row
    maps some points in the principal lattice out of $K$.
    Any partial $\Dim$-permutation
     in the penultimate line 
     already fully determines a $\Dim$-permutation: 
     in each row of its matrix
     the first $n$ entries settle the last one
     since each 1-permutation is a bijection.
    Hence we combine the last $\Dim$ steps of subtree spanning
     into one.
    Compared to that in subplot (c), 
    the subtree spanned in subplot (b) according to test sets of type II
    is much smaller and thus
    the resulting backtracking algorithm in Definition
    \ref{def:testOrderingOneNorm}
    is more efficient.
    Suppose the depth-first search scans the solution space
    from the right to the left,
    then for test sets of type II
    it would take 5 prunings and 6 spannings
    to obtain the solution in subplot (a)
    while the numbers of prunings and spannings
    for test sets of type I
    are 14 and 18, respectively.
  }
  \label{fig:backtrack_with_test_sets}
\end{figure}
} 
}

\subsection{Backtracking for TLG}
\label{sec:backtr-algor-plg}

To solve the TLG problem,
 we specialize the backtracking algorithm
 in Definition \ref{def:backtracking}
 by giving specific definitions to the itemized subroutines
 in \texttt{typewriter} font.
%
We begin by defining an ordering that is useful
 for the subroutines \texttt{first} and \texttt{next}.
 
 \begin{defn}
 \label{def:testOrderingOneNorm}
  For a TLG problem $(K, \mathbf{q})$,
  the (compactness-first) \emph{test ordering} ``$<$'' of a subset
  $J_i\subseteq \mathbb{Z}_n$ along the $i$th dimension
  is a total ordering on $J_i$
  determined first by the distance to $\mathbf{q}$
  along the $i$th dimension
  and then by breaking any tie with the cardinality of the slices.
  More precisely, for any distinct $j,k\in J_i$,
  we have
  \begin{equation} 
    \label{eq:testOrderingOneNorm}
    j<k \ \Leftrightarrow \ \bigl( \vert j-q_i \vert > \vert k - q_i \vert \bigr)
    \ \ \vee\ \
      \bigl(\vert j-q_{i} \vert = \vert k-q_i \vert \ \wedge\ 
    \#L_{i,j}(K) < \#L_{i,k}(K) \bigr),
  \end{equation}
  where $q_i$ is the $i$th coordinate of $\mathbf{q}$.
  \revise{We pick one randomly from $j$ and $k$
  if the right-hand side (RHS) logical statement
  in (\ref{eq:testOrderingOneNorm})
  is false.}
  In particular, the element in $J_i$
  that is greater and less than all other elements in $J_i$
  is denoted by $\max J_i$ and $\min J_i$, respectively.
 \end{defn}


Note that the test ordering is just
 one particular choice of traversing the spanning tree.
There might as well exist many other ways
to traverse the solution space.
For example, the feasibility-first test ordering
 is a total ordering on $J_i$
 determined first by the cardinality of the slices
 along the $i$th dimension
 and then by breaking any tie with the distance to $\mathbf{q}$,
 i.e., 
 \begin{align} 
    \nonumber
   & j < k\ \Leftrightarrow\
   \\ \label{eq:testOrderingFeasibility}
    &\bigl(\#L_{i,j}(K) < \#L_{i,k}(K)\bigr)
   \ \vee\ 
   \bigl(\#L_{i,j}(K) = \#L_{i,k}(K)\ \wedge\ 
   |j-q_{i}|> |k-q_i| \bigr).
 \end{align}
 \revise{$j$ and $k$ is picked randomly 
 if the above RHS logical statement is false.}

The set of \emph{all} solutions of a given TLG problem
 does not depend on the order of traverse;
 after all,
 one never knows whether the current set of solutions
 is the set of all solutions until she has exhausted
 all possibilities.
 
It suffices to
 find just one solution of the TLG problem; 
 so we save some time by stopping the traverse
 with the first solution.

\begin{defn}[A backtracking algorithm for TLG]
\label{def:algorithm-plg}
The following recursive algorithm
 finds all solutions to the TLG problem
 in Definition \ref{def:poisedLatticeGenProblem}.

\IncMargin{1em}
\begin{algorithm2e}[H]
  \DontPrintSemicolon
  \SetKwInOut{Precond}{Preconditions}
  \SetKwInOut{Postcond}{Postconditions}
  \SetKw{BT}{BackTrack}
  \KwIn{the degree $n\in \mathbb{N}^+$,
    the dimensionality $\Dim\in \mathbb{N}^+$,
    the search domain $K\subseteq \mathbb{Z}_{n}^{\Dim}$,
    and the starting point $\mathbf{q}\in K$}
  \KwOut{a set ${\mathcal U}$ of $\Dim$-permutations}
\Postcond{$\{A{\mathcal P}_n^{\Dim}: A\in {\mathcal U}\}$
  is the set of all solutions
  of the TLG problem} 
\BlankLine
${\mathcal U} \leftarrow$ an empty set of $\Dim$-permutations \;
$A^{(0)}\leftarrow$ \texttt{root}($n,\Dim$)\;
\BT $\left ((K,\mathbf{q}), A^{(0)}, {\mathcal U}\right)$\;
\end{algorithm2e}
\DecMargin{1em}

where $A^{(t)}$ is the matrix in (\ref{eq:partialPermutScalar})
and the procedures in Definition \ref{def:backtracking}
are described as follows.
\begin{itemize}\itemsep0em
\item \texttt{root}$(n,\Dim)$:
  set $A^{(0)}$ to a $\Dim$-by-$(n+1)$ matrix
  of constant $-1$ following Notation \ref{ntn:DPermutationMatrix};
  return $A^{(0)}$.
\item \texttt{accept}$((K,\mathbf{q}),A^{(t)})$:
  return false if $t<\Dim(n+1)$;
  otherwise
  return true if and only if
  \begin{displaymath}
    A^{(t)} W_{(\Dim,n)} \subset K \text{ and }
    \mathbf{q}\in A^{(t)} {\mathcal P}_n^{\Dim}.
  \end{displaymath}
\item \texttt{stopAfterAccept}$((K,\mathbf{q}),A^{(t)})$:
  return true.
\item \texttt{reject}$((K,\mathbf{q}),A^{(t)})$:
  return the negation of \texttt{accept}$((K,\mathbf{q}),A^{(t)})$ 
  if \mbox{$t=\Dim(n+1)$}; 
  otherwise return the value of the logical statement
  \begin{displaymath}
    \exists \mathbf{p}\in W_{(\ell,m)} \text{ s.t. }
    A^{(t)} \mathbf{p} \not\in K,
  \end{displaymath}
  where $(\ell,m)=s^{-1}(t)$.
\item \texttt{first}$((K,\mathbf{q}),A^{(t)})$:
  let $(\ell,m)=s^{-1}(t+1)$ and initialize
  \begin{displaymath}
    B^{(t+1)}\leftarrow A^{(t)}, 
  \end{displaymath}
  \begin{equation}
    \label{eq:freeSet}
    J_{\ell} := \mathbb{Z}_n\setminus \left\{B^{(t+1)}(\ell,j):
        j=0,1,\ldots,m-1\right\}.
  \end{equation}
  Set $B^{(t+1)}(\ell,m) \leftarrow \max J_{\ell}$
  as in Definition \ref{def:testOrderingOneNorm} and 
  return $B^{(t+1)}$.
\item \texttt{next}$((K,\mathbf{q}),A^{(t)},B^{(t+1)})$:
  let $(\ell,m)=s^{-1}(t+1)$ and compute $J_{\ell}$ by (\ref{eq:freeSet}).
  Return null if $B^{(t+1)}(\ell,m)$ equals $\min J_{\ell}$
   as in Definition \ref{def:testOrderingOneNorm}.
  Otherwise initialize $C^{(t+1)}\leftarrow B^{(t+1)}$,
  set $C^{(t+1)}(\ell,m)$ to be the element in $J_{\ell}$
  that is \emph{immediately} smaller than $B^{(t+1)}(\ell,m)$,
  and return $C^{(t+1)}$.
\end{itemize}
\end{defn}

The above specifications of subroutines 
 entail some explanation.
First, a partial $\Dim$-permutation
 is formally not a $\Dim$-permutation until the last stage $t=\Dim(n+1)$; 
 this is the reason for the conditions
 $t<\Dim(n+1)$ and $t=\Dim(n+1)$
 in the subroutines \texttt{accept}
 and \texttt{reject}, respectively.
However, in the case of test sets of type II, 
 the partial $\Dim$-permutation
 has already been fully determined at the stage $t=\Dim n$;
 see Example \ref{exm:testSet2PartitionD2n2}. 
\revise{
Second, the subroutine \texttt{stopAfterAccept}
 always returns true because
 a $\Dim$-permutation, once determined,
 can never be augmented to another one.
}
Third, the subroutine \texttt{first}
 returns the first child node of the parent node $A^{(t)}$, 
 hence in (\ref{eq:freeSet}) we should remove from $\mathbb{Z}_n$
 the numbers that are already in the same row of $A^{(t)}$.
Lastly, the subroutine \texttt{next}
 returns the child node of the parent node $A^{(t)}$
 that is immediately after the child node $B^{(t+1)}$.
Hence $B^{(t+1)}=\min {\mathcal J}_{\ell}$ means that
 all child nodes of $A^{(t)}$ have already been checked.
 


\section{The PLG-FD method}
\label{sec:PLG-FD}

Based on the TLG algorithm in Section \ref{sec:algorithms}, 
 we propose a new FD method in 

\begin{defn}[The PLG-FD method]
  \label{def:PLG-FD}
  For an elliptic equation 
  \begin{equation}
    \label{eq:PDE2solve}
    {\mathcal L}u(\mathbf{x}) = 0\quad  \text{ in } \Omega
  \end{equation}
  of a scalar function $u: \overline{\Omega} \rightarrow \mathbb{R}$
  with appropriate boundary conditions on $\partial \Omega$,
  the \emph{PLG-FD method} numerically solves
  (\ref{eq:PDE2solve})
  by four steps as follows.
  \begin{enumerate}[(a)]
  \item \emph{Embed $\Omega\subset \mathbb{R}^{\Dim}$ in a rectangular Cartesian grid}.
    For ease of exposition, the grid is assumed to have a uniform grid
    size $h$ for all dimensions.
    Then each cell can be indexed by a multiindex $\bmi \in \mathbb{Z}^{\Dim}$
    with its center at
      $\mathbf{x}_{\bmi} = h \bmi + \frac{1}{2}h \unitV$,
    where $\unitV$ is the multiindex whose components are all 1's.
  \item \emph{Classify the cell centers}.
    \begin{itemize}
    \item A cell center $\mathbf{x}_{\mathbf{i}}$ is called an \emph{exterior node}
      if $\mathbf{x}_{\bmi} \not\in \Omega$ and its distance to $\partial \Omega$ is greater than $\eta h$,
      where $\eta<\frac{1}{2}$ is a user-specified positive number.
    \item A cell center $\mathbf{x}_{\mathbf{j}}$ is called a \emph{boundary node}
      if $\mathbf{x}_{\mathbf{j}}$ is not an exterior node
      but a face neighbor of an exterior node $\mathbf{x}_{\mathbf{i}}$,
      i.e., $\|\mathbf{j}-\mathbf{i}\|_1=1$.
      Each boundary node is also associated with a \emph{boundary point},
      i.e., the point on $\partial \Omega$ that is closest to the boundary node.
    \item All other cell centers are called \emph{interior node}s.
    \end{itemize}
    Boundary nodes and interior nodes are collectively called
    \emph{FD node}s, where discrete solutions of the PDE are sought.
    See Figure \ref{fig:geometricFD} for an illustration.
  \item \emph{Discretize the PDE for each FD node}.
    \begin{itemize}
    \item If the FD node $\mathbf{x}_{\mathbf{i}}$
      is a \emph{regular FD node},
      i.e., an interior node away from the irregular boundary such that
      all points in its standard one-dimensional FD stencils
      are FD nodes, 
      we use these standard stencils and their tensor products
      to discretize ${\mathcal L}$ at $\mathbf{x}_{\mathbf{i}}$.
      For example, for fourth-order accuracy in two dimensions, 
      the Laplacian is discretized with the standard 9-point stencil
      and the cross derivative term
      $\frac{\partial^2 u}{\partial x \partial y}$
      with a 5-by-5 square box.
    \item Otherwise $\mathbf{x}_{\mathbf{i}}$ is called 
      \emph{an irregular FD node}.
      We generate a poised lattice $\mathbf{S}(\mathbf{i})$
      for $\mathbf{x}_{\mathbf{i}}$ 
      from nearby FD nodes
      (as a solution to the problem in Definition \ref{def:PLG}). 
      \revise{Some extra FD nodes may also be added 
      into $\mathbf{S}(\mathbf{i})$}.
      \begin{itemize}
      \item \revise{If $\mathbf{x}_{\bmi}$ is an interior node, 
          we approximate $u$ at each FD node
          $\mathbf{j}\in\mathbf{S}(\mathbf{i})$
          by a multivariate polynomial
          $P_\mathbf{i} = \sum_{k=1}^N
          c_{\mathbf{i},k}\phi_k(\mathbf{x})$
          with $\phi_i$'s as the monomial basis of $\Pi^{\Dim}_n$;
          this yields
          \begin{equation}
          \label{eq:fittingPoly}
          \begin{array}{l}
            u(\mathbf{x}_\mathbf{j}) \approx
            P_\mathbf{i}(\mathbf{x}_{\mathbf{j}})
            = \sum_{k=1}^N b_{\mathbf{j},k}c_{\mathbf{i},k},
          \end{array}
        \end{equation}
        where $b_{\mathbf{j},k}:= \phi_k(\mathbf{x}_{\mathbf{j}})$.
        Equations in (\ref{eq:fittingPoly}) 
        yield a linear system, 
        \begin{equation}
          \label{eq:fittingLinearSystem}
          \mathbf{B}_\mathbf{i}\mathbf{c}_\mathbf{i}
          = \mathbf{U}_{\mathbf{S}(\mathbf{i})},
        \end{equation}
        where the $\mathbf{j}$th row
        of the sample matrix $\mathbf{B}_\mathbf{i}$ 
        is the vector $\mathbf{b}_\mathbf{j}$
        that consists of $b_{\mathbf{j},k}$'s,
        the unknown vector $\mathbf{c}_\mathbf{i}$
        consists of $c_{\mathbf{i},k}$'s,
        and the vector $\mathbf{U}_{\mathbf{S}(\mathbf{i})}$
        contains the \emph{symbol}s of the values of $u$ at the nodes
        in $\mathbf{S}(\mathbf{i})$.
        Solve (\ref{eq:fittingLinearSystem})
        (in the least square sense
        if $\mathbf{S}(\mathbf{i})> \dim \Pi^{\Dim}_n$)
        to express $\mathbf{c}_\mathbf{i}$
        as a linear combination of the symbols
        in $\mathbf{U}_{\mathbf{S}(\mathbf{i})}$,
        substitute $\mathbf{c}_\mathbf{i}$ into
        (\ref{eq:fittingPoly}) to obtain a multivariate polynomial
        in $\mathbf{x}$, 
        differentiate the polynomial according to the operator ${\cal L}$
        in the PDE (\ref{eq:PDE2solve}), 
        and we have a discrete equation 
        whose unknowns are the symbols
        in $\mathbf{U}_{\mathbf{S}(\mathbf{i})}$.}
      \item Otherwise $\mathbf{x}_{\bmi}$ is a boundary node.
        The procedures to obtain the discretized PDE are almost the same
        as those of the previous case except that
        we incorporate the boundary condition
        at the boundary point $b_{\bmi}$
        \revise{by adding an extra equation
        into the linear system (\ref{eq:fittingLinearSystem})
        and solve the resulting least squares problem.
        For example,
        for the Neumann boundary condition
        $\frac{\partial u}{\partial \mathbf{n}} = v_n$,
        we append
        to the symbolic vector $\mathbf{U}_{\mathbf{S}(\mathbf{i})}$
        the numerical value of $v_n$
        and augment the matrix $\mathbf{B}_\mathbf{i}$ 
        with the extra row
        $\left[
            \frac{\partial \phi_1}{\partial \bmn}(b_{\bmi}), \ \cdots, \ \frac{\partial \phi_N}{\partial \bmn}(b_{\bmi}) 
          \right]$ 
        where $\bmn$ is the unit outward normal of $\partial \Omega$}.
      \end{itemize}
    \end{itemize}
  \item \emph{Solve the system of discrete equations
      that result from (c)
      to produce the numerical solution of (\ref{eq:PDE2solve})}.
  \end{enumerate}
\end{defn}

\begin{figure}
  \centering
  \includegraphics[width=0.5\textwidth]{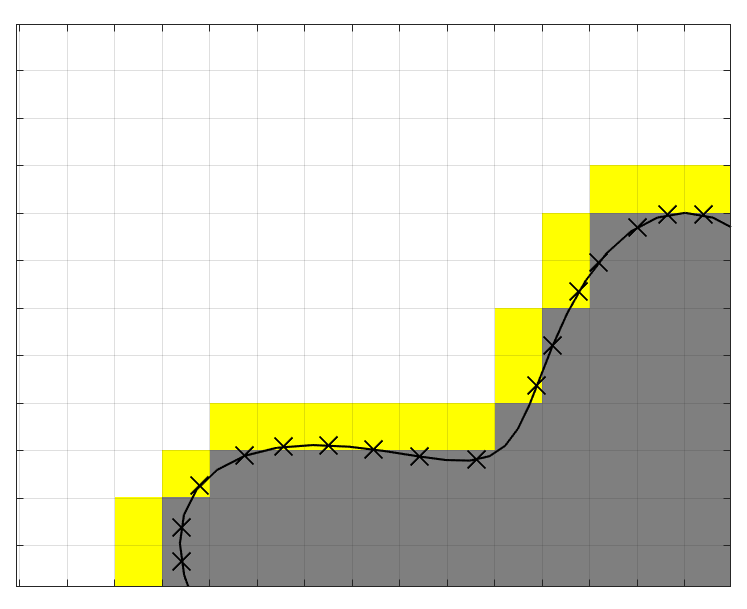}
  \caption{Classify cell centers in the PLG-FD method. 
    The curve represents part of the domain boundary;
    a gray square represents an exterior node,
    a white square an interior node,
    a yellow square a boundary node,
    and a cross mark a boundary point.
  }
  \label{fig:geometricFD}
\end{figure}


\begin{figure}
  \centering
  \subfigure[Discretizing $\Delta u$. 
  The first pivot set
  only contains the starting point $\mathbf{q}$
  and the resulting set $K$ contains all yellow cells.
  ]{
    \includegraphics[width=0.40\linewidth]{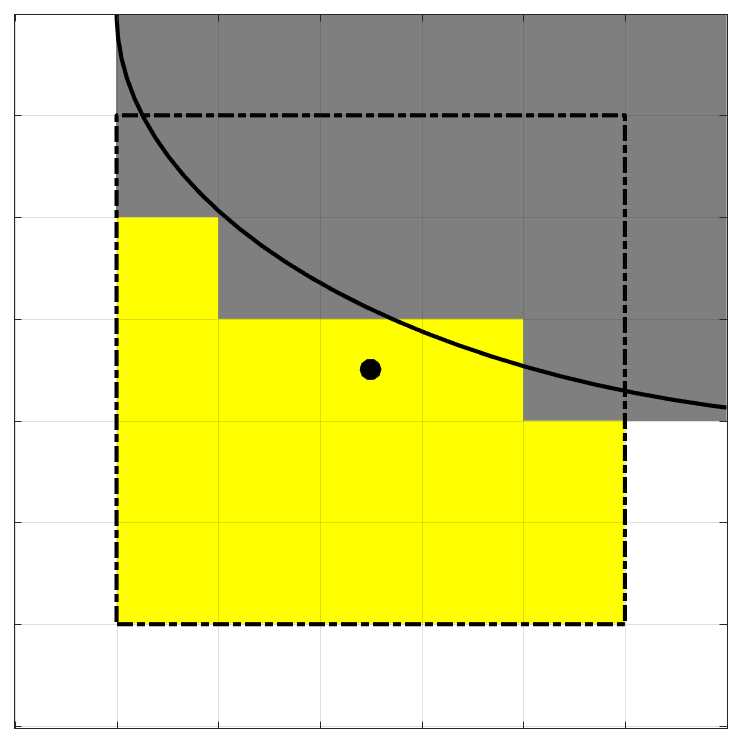}
  }
  \hfill
  \subfigure[Discretizing $\mathbf{v}\cdot \nabla {u}$. 
  The first and the last pivot sets are $\{\mathbf{p}_1\}$
   and $\{\mathbf{p}_2\}$,
   where $\mathbf{p}_1$ and $\mathbf{p}_2$ are defined in (\ref{eq:advectionP1P2})
   and represented by the hollow square and the hollow dot, respectively.
  The set $K$ resulting from $\mathbf{p}_1$
   contains the yellow cells and white cells inside the dot-dashed box.
  ]{
    \includegraphics[width=0.41\textwidth]{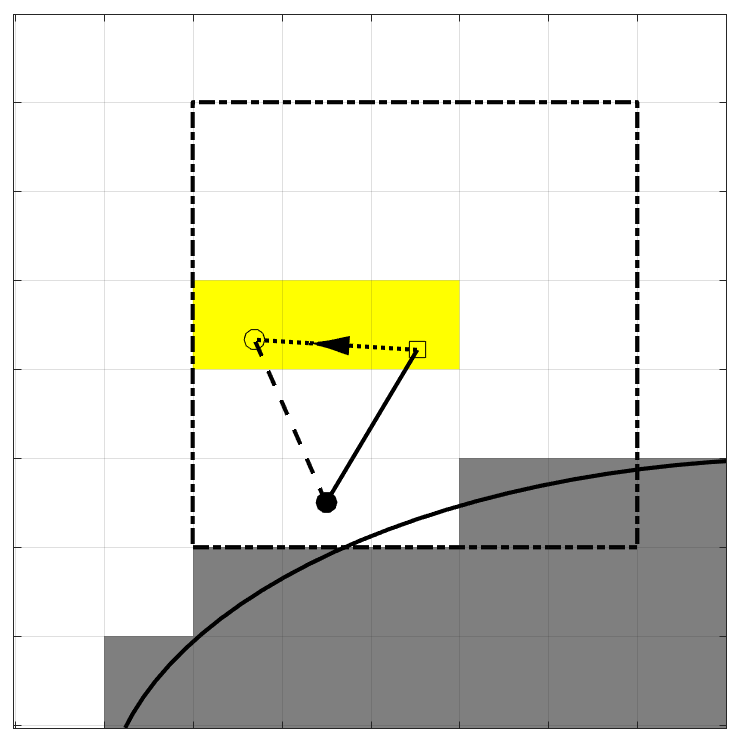}
  }
  \caption{Our strategy (KQN) to determine the set $K$ of feasible nodes
    from an even $n$ and an irregular FD node $\mathbf{q}$
    for PLG discretization 
    in the PLG-FD method. 
    A solid dot represents the irregular FD node,
     to which the starting point $\mathbf{q}$ is always assigned.
    Gray squares represent exterior nodes
     and the curve the domain boundary $\partial \Omega$.
    A box of dot-dashed line represents a shifted cube 
     centered at the first pivot.
    %
    Yellow squares represent pivots in the sequence of pivot sets.
  }
  \label{fig:select_box}
\end{figure}


 
The above PLG-FD method is appealing in a number of aspects.

First, the steps in Definition \ref{def:PLG-FD}
 directly apply to a wide array of PDEs.
For elliptic equations,
 adding a cross derivative term
 such as $\frac{\partial^2 u}{\partial x \partial y}$
 incurs no changes on algorithmic steps.
Coupled to a time integrator,
 the PLG-FD method is also an efficient solver
 for time-dependent problems
 such as the advection-diffusion equation. 
\revise{
The resulting method has two simple steps: 
\begin{enumerate}[(PFT-1)]
\item discretize the time derivative according
  to the time integrator to obtain
  a sequence of elliptic equations in the form of
  (\ref{eq:PDE2solve});
\item solve the elliptic equations by the steps
  in Definition \ref{def:PLG-FD}. 
\end{enumerate}
}

\revise{Second,
  the algorithmic steps of PLG-FD remain the same
  for different boundary conditions
  and different forms of the PDE, 
 due to two reasons.
 First, the unknown $u$ is locally represented
 by a multivariate polynomial in $\mathbf{x}$
 whose coefficients are functions of
 the symbols in $\mathbf{U}_{\mathbf{S}(\mathbf{i})}$
 that represent values of $u$ at fixed locations;
 see step (c) in Definition \ref{def:PLG-FD}.
 Second, it is extremely easy
 to take derivatives of this multivariate polynomial
 according to the operators in the PDE (\ref{eq:PDE2solve})
 and its boundary conditions. 
 This generality/flexibility
 is an advantage of PLG-FD over other FD methods
 where the discretization of spatial operators
 is coupled to the PDE and boundary/interface conditions.}


Third, the PLG-FD method achieves high-order 
 accuracy while retaining the simplicity of traditional FD methods.
For second- and lower-order methods (such as that in
\cite{johansen98:_cartes_grid_embed_bound_method}),
the problem of choosing interpolation sites in fitting a polynomial is
trivial for $n=1,2$,
but not so for higher values of $n$.
Some authors resort to least squares
 by adding a large number of interpolation sites in polynomial fitting.
However, least squares do not guarantee
 the nonsingularity of the linear system
 and a large number of redundant interpolation sites
 might deteriorate the efficiency.
In contrast,
 least squares in the PLG-FD method
 start with a poised lattice
 and the balance between conditioning and efficiency is completely under control.
See the last paragraph in Section \ref{sec:testTLG} for more discussions.

Lastly, the TLG algorithm is dimensionality-agnostic
 and is flexible, through the choice of feasible sets, 
 in catering for the physics of the operators to be discretized.

\subsection{How to choose the feasible set in PLG-FD?}
\label{sec:how-choose-feasible}

In this subsection we suggest a choice of the feasible set
 according to the physics of the problem at hand.
For example, to discretize the Laplacian $\Delta u$,
 the generated lattice should be centered 
 at the starting point as much as possible.
In comparison, for the advection operator $\mathbf{v}\cdot \nabla {u}$,
 we typically have a lopsided lattice
 to capture the physics that most information
 comes from the upwind direction $-\mathbf{v}$. 

\begin{enumerate}[({CFS-}1)]
\item From $n$ and $\mathbf{q}$, we determine 
  a \emph{sequence of pivot sets} $O_j\subset \mathbb{R}^{\Dim}$.
  Denote the union of cell corners and cell centers by
  \begin{equation}
    \label{eq:UnionofCornersAndCenters}
    V:= \left( h\mathbb{Z} \right)^{\Dim} \cup \left( h\mathbb{Z} + \frac{h}{2} \unitV \right)^{\Dim}.
  \end{equation}
  For the Laplacian $\Delta u$, we choose
  \begin{equation}
    \label{eq:sequencePivotSets4Laplacian}
    O_j^{\text{Lap}} = \left\{\mathbf{x}\in V:
      \frac{1}{h}\|\mathbf{x}-\mathbf{q}\|_{\infty} = j-1
      + \frac{1}{2}(n \textrm{ mod } 2)\right\}
  \end{equation}
   with $j=1,2,\ldots, \alpha_n$ and $\alpha_n :=\lfloor \frac{n}{2}\rfloor$.
  For the advection $(\mathbf{v}\cdot \nabla) u$,
   we use (\ref{eq:sequencePivotSets4Laplacian})
   if $\|\mathbf{v}\|_2$ is sufficiently small;
   otherwise we define
   \begin{equation}
     \label{eq:advectionP1P2}
     \mathbf{p}_1:= \mathbf{q} - \alpha_nh 
     \frac{\mathbf{v}}{\|\mathbf{v}\|_2},\qquad
     \mathbf{p}_2:= \mathbf{q} - \alpha_nh \mathbf{n}_b,
   \end{equation}
   where $\mathbf{n}_b$ is the unit outward normal of
   $\mathbf{q}_b\in\partial \Omega$ and $\mathbf{q}_b$
   is the closest point to $\mathbf{q}$ on $\partial \Omega$.
  As a particle travels from $\mathbf{p}_1$ to $\mathbf{p}_2$,
   it intersects a sequence of cells 
   whose centers are FD nodes $\mathbf{x}^{\text{adv}}_j$. 
  Then we choose  
   \begin{equation}
    \label{eq:sequencePivotSets4Advection}
    O_j^{\text{Adv}} = \left\{\mathbf{x}\in V:
      \frac{1}{h}\|\mathbf{x}-\mathbf{x}^{\text{adv}}_j\|_{\infty} = \frac{1}{2}(n \textrm{ mod } 2)\right\}.
  \end{equation}
\item For each pivot $\mathbf{x}\in O_j$, we shift $\mathbb{Z}_n^{\Dim}$
     to center it at $\mathbf{x}$,
     identify $K$ with the set of FD nodes in the shifted cube, 
     and solve the resulting TLG problem to obtain at most one poised lattice.
\item If the previous step yields multiple poised lattices,
  we select the lattice $\{\mathbf{x}_{\mathbf{j}}\}$ of which
  $\sum_{\mathbf{j}}\|\mathbf{x}_{\mathbf{j}}-\mathbf{q}\|_2^2$ is minimized.

\item Otherwise no solution exists for any pivot in $O_j$.
     We increase $j$ by 1 and repeat (CFS-1,2,3)
     until we have a solution 
     or until, for the given $n$, 
     no solution exists for any pivot set in the sequence.
\end{enumerate}

In Figure \ref{fig:select_box} we illustrate the above strategy
in the case of $n$ being an even number.
Numerical experiments show that $j=1$ in (CFS-1) yields a poised lattice
in most cases.
%



\section{Tests}
\label{sec:tests}

In Section \ref{sec:testTLG},
 we test the TLG algorithm by generating poised lattices
 near irregular boundaries.
The variation of condition number of the sample matrix
 with respect to the number of points in least square stencils
 is also studied for representative scenarios. 
In Section \ref{sec:testDiscretization},
 we test the TLG discretization of a spatial operator
 to show fourth- and sixth-order convergence of truncation errors.
In Section \ref{sec:ellipticEqWithCrossDerivative}, 
 we demonstrate the effectiveness of PLG-FD
 by solving an elliptic equation with a cross-derivative term.
In Section \ref{sec:PoissonEq},
 we show that the fourth-order PLG-FD method 
 is more accurate than a previous second-order EB method
 in solving Poisson's equation. 
\revise{
In Section \ref{sec:PoissonEq3D},
 we extend the solution of Poisson's equation to three dimensions.
In Sections \ref{sec:heatEq} and \ref{sec:adv-diffu-eq},
 we solve the heat equation
 and the advection-diffusion equation by PLG-FD,
 demonstrating its fourth-order accuracy 
 for time-dependent problems.
}
 
\subsection{TLG near irregular boundaries}
\label{sec:testTLG}

By Definition \ref{def:poisedLatticeGenProblem}, 
 a TLG problem is uniquely specified by
 the feasible set $K$ and the starting node $\mathbf{q}$.
Our design of test cases has been centered around
 representative scenarios in the context of FD methods.

First,
 although the dimensionality $\Dim$ can be an arbitrary positive integer, 
 we only test the cases of $\Dim=2$ and $\Dim=3$.

Second,
 the starting point 
 is an irregular FD node $\mathbf{x}_{\mathbf{i}}$
 that is either a boundary node or an interior node close to
 the boundary, c.f. Definition \ref{def:PLG-FD} (c).

Third, 
 in the asymptotic case of local geometry being well resolved
 by the rectangular grid, 
 the set $K$ 
 is well described either by one hyperplane or by multiple hyperplanes.
The former corresponds to a smooth boundary surface
 while the latter a boundary surface with discontinuous normal vectors.
In both cases, however,
 $K$ must lie entirely at one side of the boundary surface.
In light of these observations, 
 for $\Dim=2$ we select the irregular boundary to be either a rotated box
 or an ellipse
\begin{equation}
  \label{eq:ellipse}
  \frac{(x-x_0)^2}{a^2} + \frac{(y-y_0)^2}{b^2} = 1.
\end{equation}

For $\Dim=3$, 
the irregular boundary is either the surface of a rotated cube
 or an ellipsoid
\begin{equation}
  \label{eq:ellipsoid}
  \frac{(x-x_0)^2}{a^2} + \frac{(y-y_0)^2}{b^2} + \frac{(z-z_0)^2}{c^2} = 1.
\end{equation}

The above characterizations of test cases are orthogonal,
 which helps the enumeration
 of representative TLG problems.
We invoke the TLG algorithm in Definition \ref{def:algorithm-plg}
 on all these TLG problems,
 stop on finding the first poised lattice,
 and list results of some of them 
 in Figures \ref{fig:TLG_2D} and \ref{fig:TLG_3D}.
 
\begin{figure}
\centering
\subfigure[$n=3$, a concave boundary]{
\includegraphics[width=.31\linewidth]{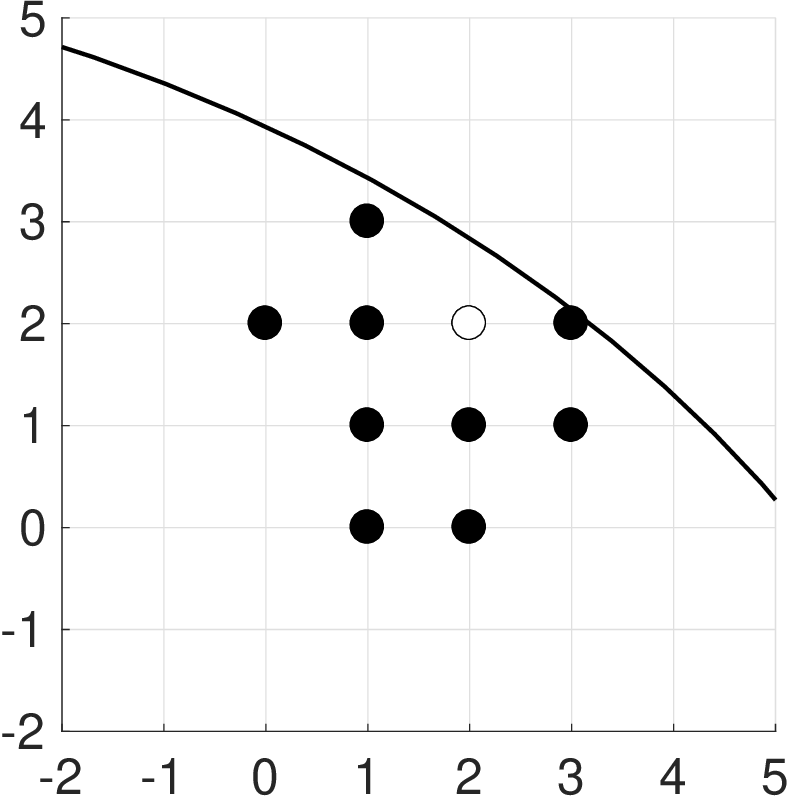}
}
\hfill
\subfigure[$n=3$, a convex boundary]{
\includegraphics[width=.31\linewidth]{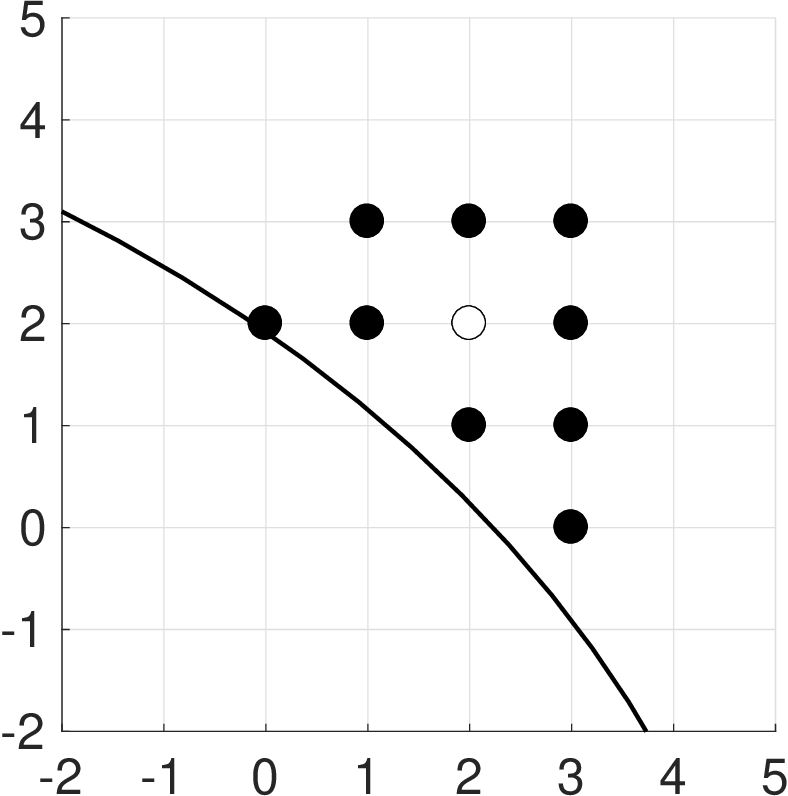}
}
\hfill
\subfigure[$n=3$, the corner of a box] {
\includegraphics[width=.31\linewidth]{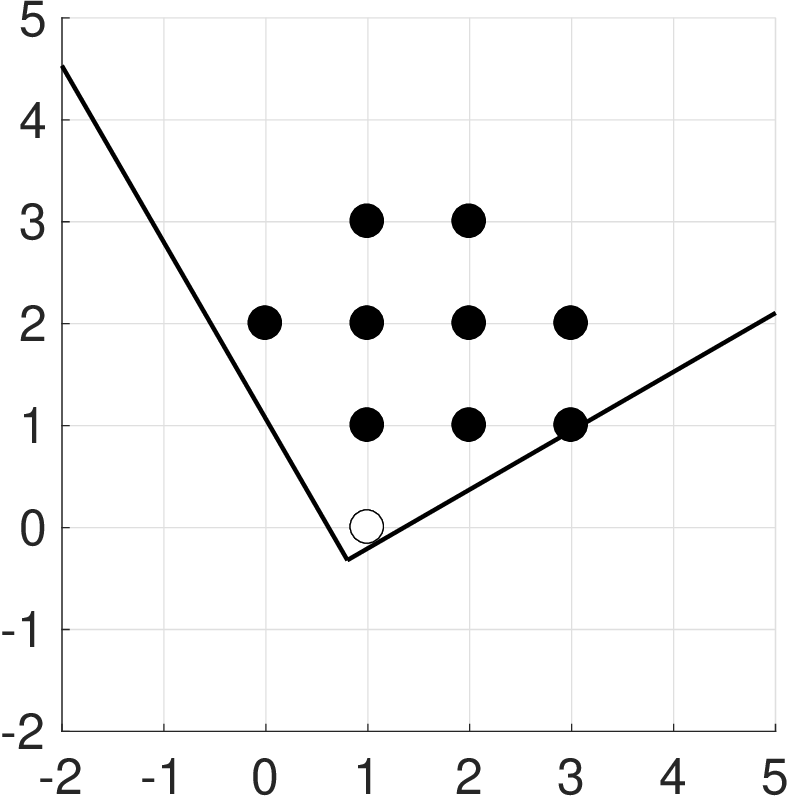}
}

\subfigure[$n=4$, a concave boundary]{
\includegraphics[width=.31\linewidth]{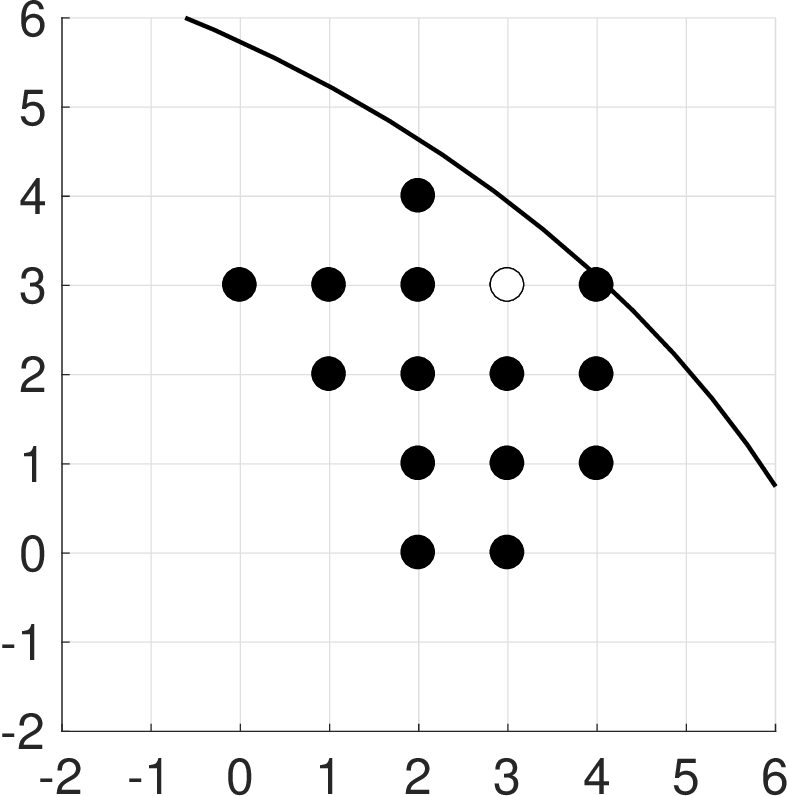}
}
\hfill
\subfigure[$n=4$, a convex boundary]{
\includegraphics[width=.31\linewidth]{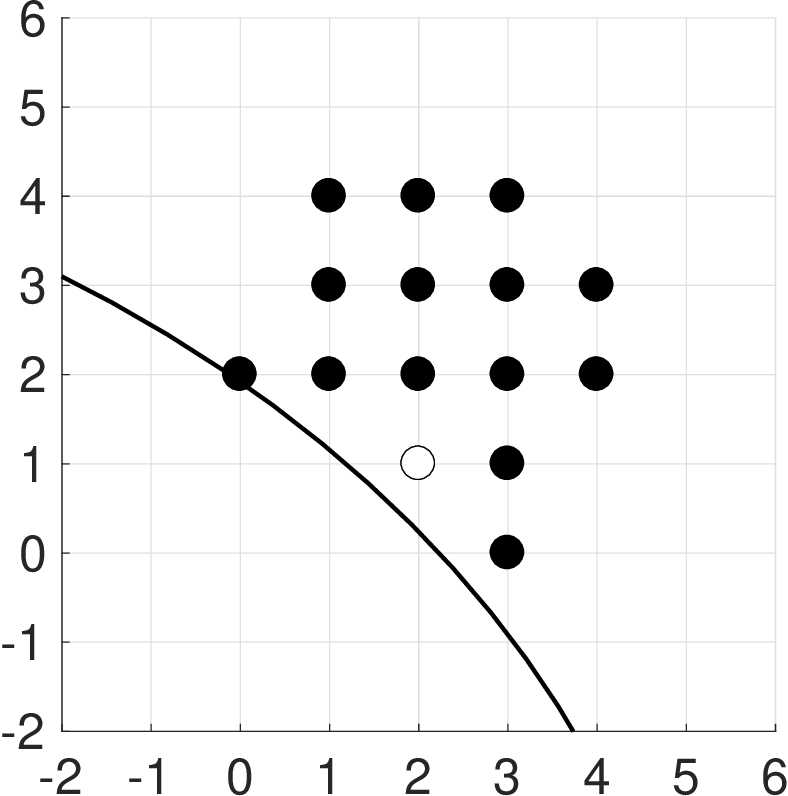}
}
\hfill
\subfigure[$n=4$, the corner of a box] {
\includegraphics[width=.31\linewidth]{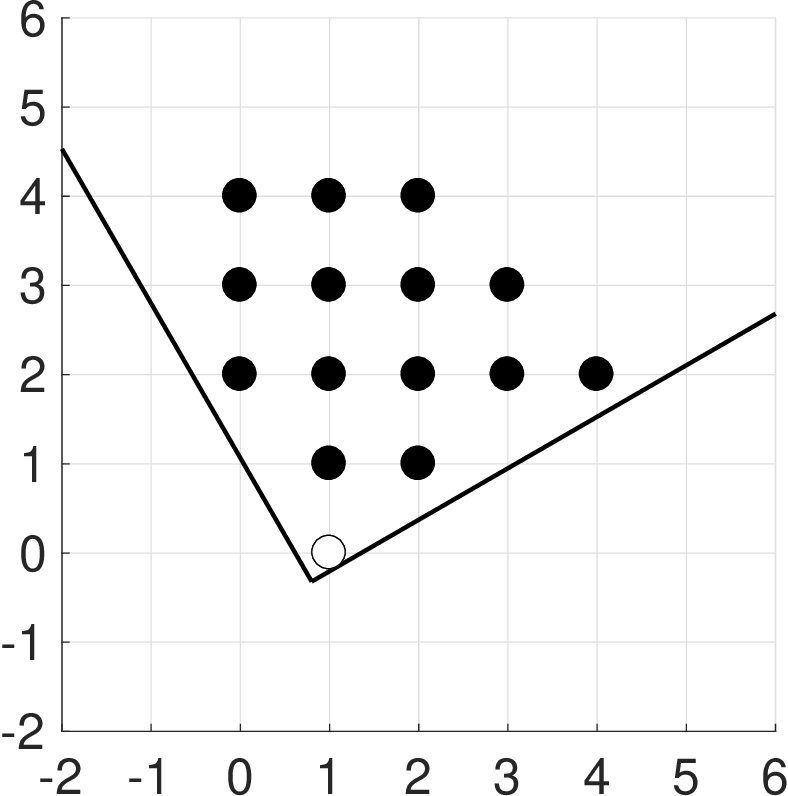}
}

\subfigure[$n=5$, a concave boundary]{
\includegraphics[width=.31\linewidth]{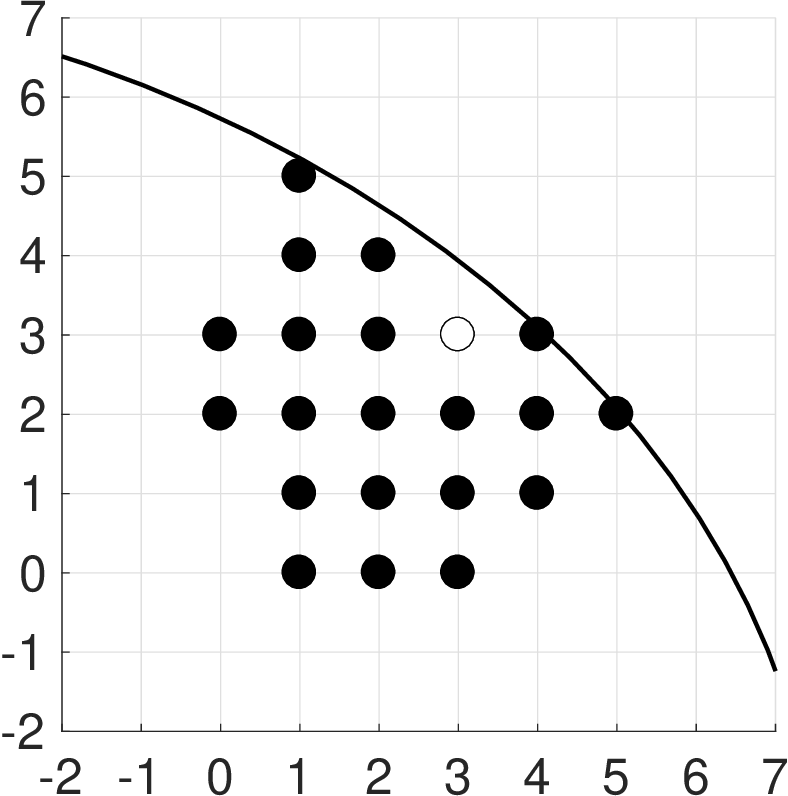}
}
\hfill
\subfigure[$n=5$, a convex  boundary]{
\includegraphics[width=.31\linewidth]{eps/S-2-5-cv}
}
\hfill
\subfigure[$n=5$, the corner of a box] {
\includegraphics[width=.31\linewidth]{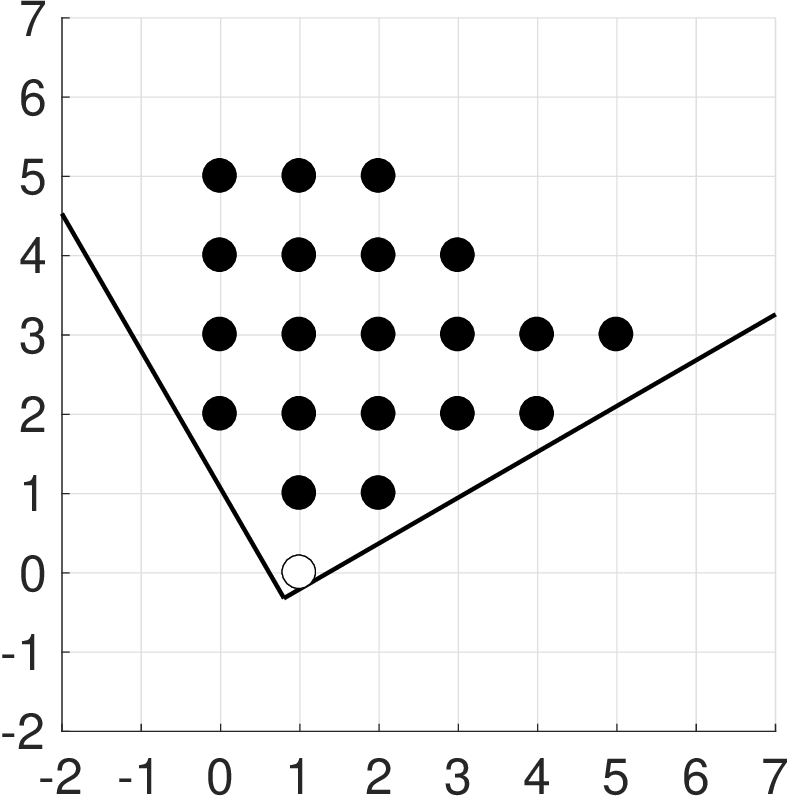}
}
\caption{Poised lattices generated
  by the TLG algorithm in Definition \ref{def:algorithm-plg}
  with the compactness-first test ordering
  (\ref{eq:testOrderingOneNorm}).
  In all cases,
   the starting point $\mathbf{q}$ is represented by
   the hollow dot
   and $K$ is the set of FD nodes in the shifted cube $\mathbb{Z}^2_n$
   that contains all hollow and solid dots.
  All smooth boundaries are arcs of ellipses in (\ref{eq:ellipse})
  with $(a,b)=(16, 12)$.
  We have \mbox{$(x_0, y_0)=(-8.40, -6.28)$} in subplot (a),
  \mbox{$(x_0, y_0)=(-8.40, -4.48)$} in subplots (d,g), 
  and \mbox{$(x_0, y_0)=(-10.88, -6.88)$} in subplots (b,e,h). 
  In subplots (c,f,i), 
  the lower corner of each box is at $(0.80, -0.32)$
  with the right side angled at $\frac{\pi}{6}$.
  \revise{
    Each dot, hollow or solid, is a cell center
    and thus by Definition \ref{def:PLG-FD}(a),
    the grids do not represent cell boundaries.
    Consequently, the hollow dots in subplots (b,h)
    represent irregular interior nodes,
    c.f. Definition \ref{def:PLG-FD},
    while those in all other subplots
    denote boundary nodes.
  }
 }
\label{fig:TLG_2D}
\end{figure}

\begin{figure}
\centering
\subfigure[$n=3$, 
$x_0=y_0=z_0 = 0$, \mbox{$a=b=c=3.6$}]{
\includegraphics[width=.47\linewidth]{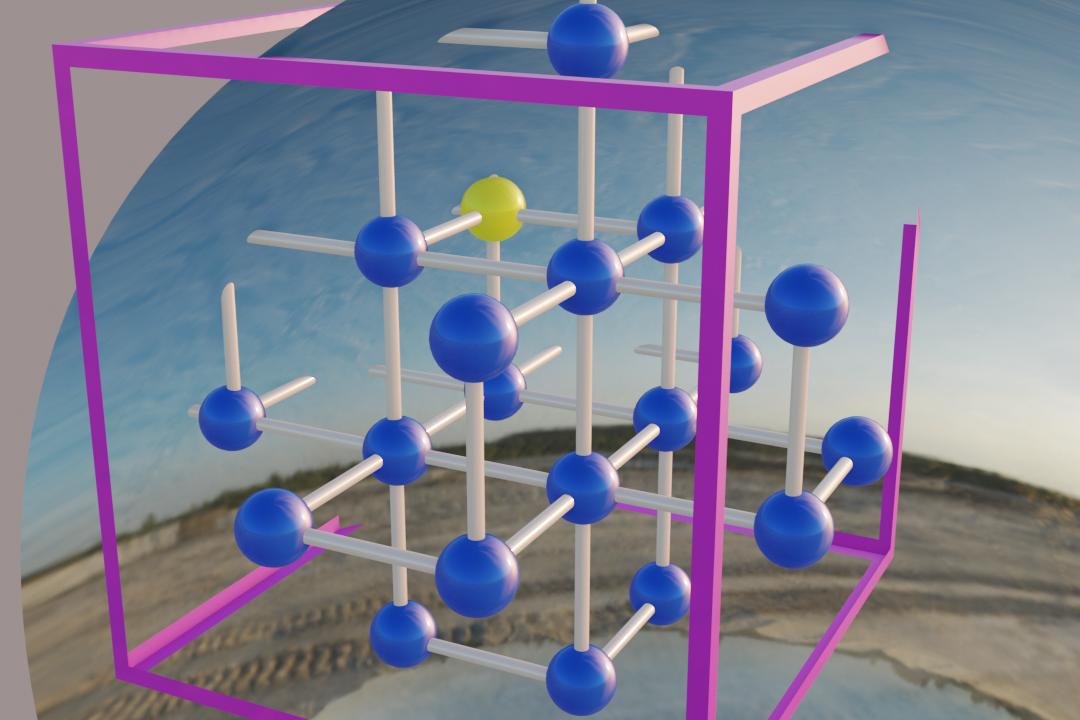}
}
\hfill
\subfigure[$n=4$, 
$x_0=y_0=z_0 = 0$, \mbox{$a=b=c=4.8$}]{
\includegraphics[width=.47\linewidth]{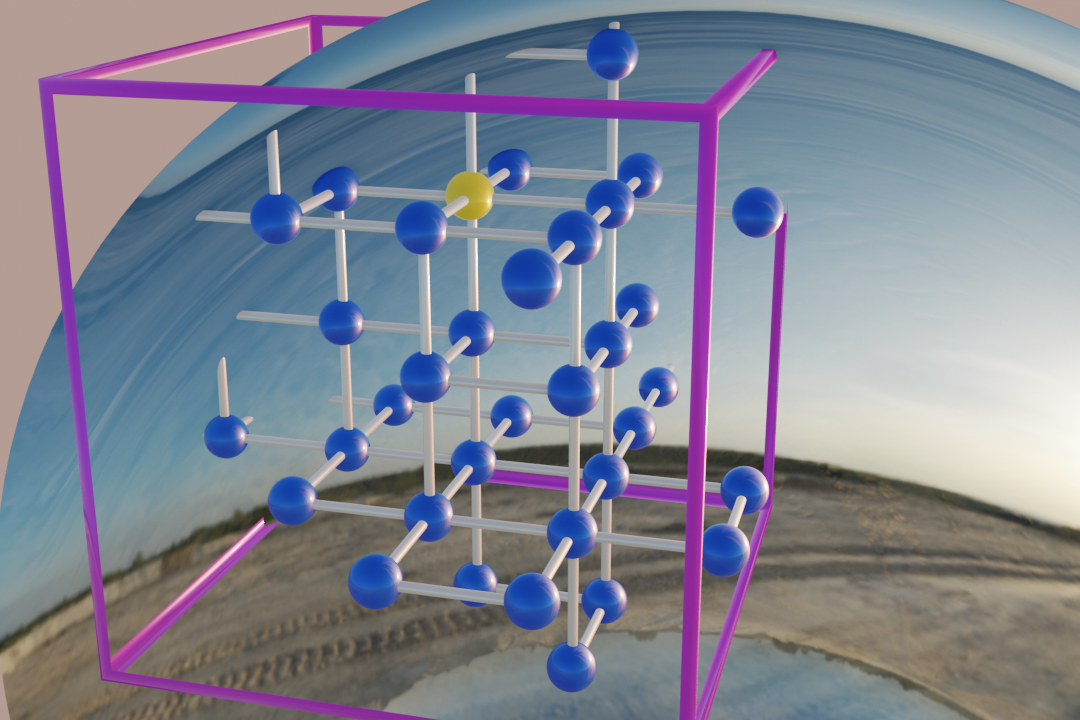}
}

\subfigure[$n=3$, 
$(x_0, y_0, z_0) = (12, 9, -3)$, {$(a,b,c) = (15, 12, 9)$}]{
\includegraphics[width=.47\linewidth]{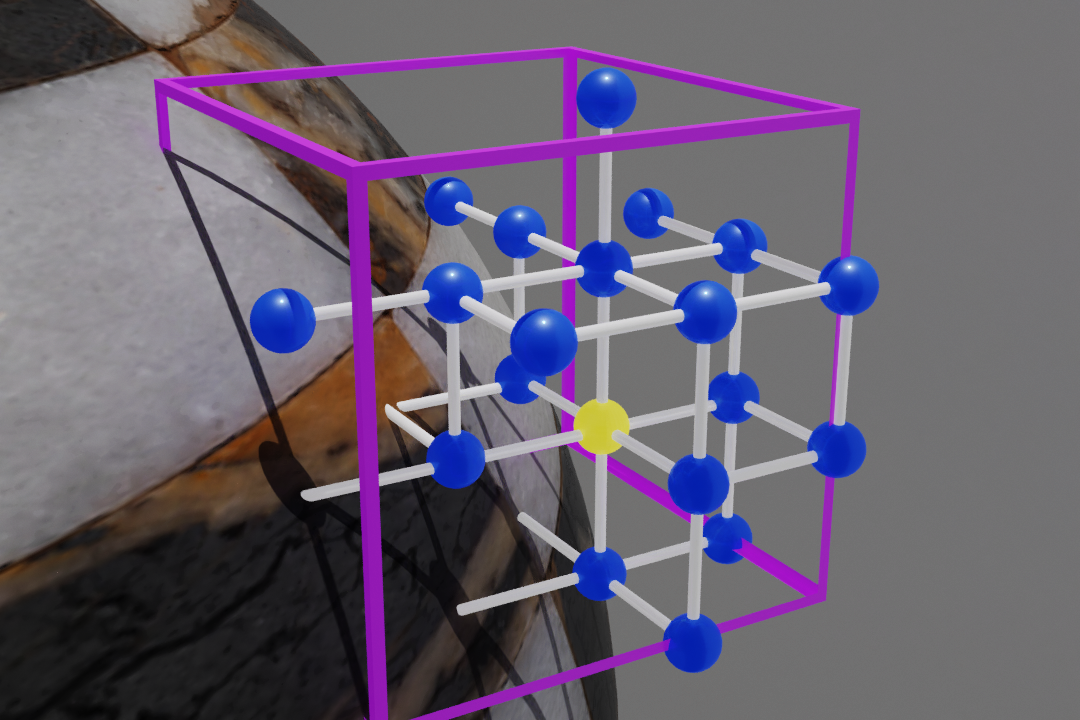}
}
\hfill
\subfigure[$n=4$, 
$(x_0, y_0, z_0) = (16, 12, -4)$, $(a,b,c) = (20, 16, 12)$]{
\includegraphics[width=.47\linewidth]{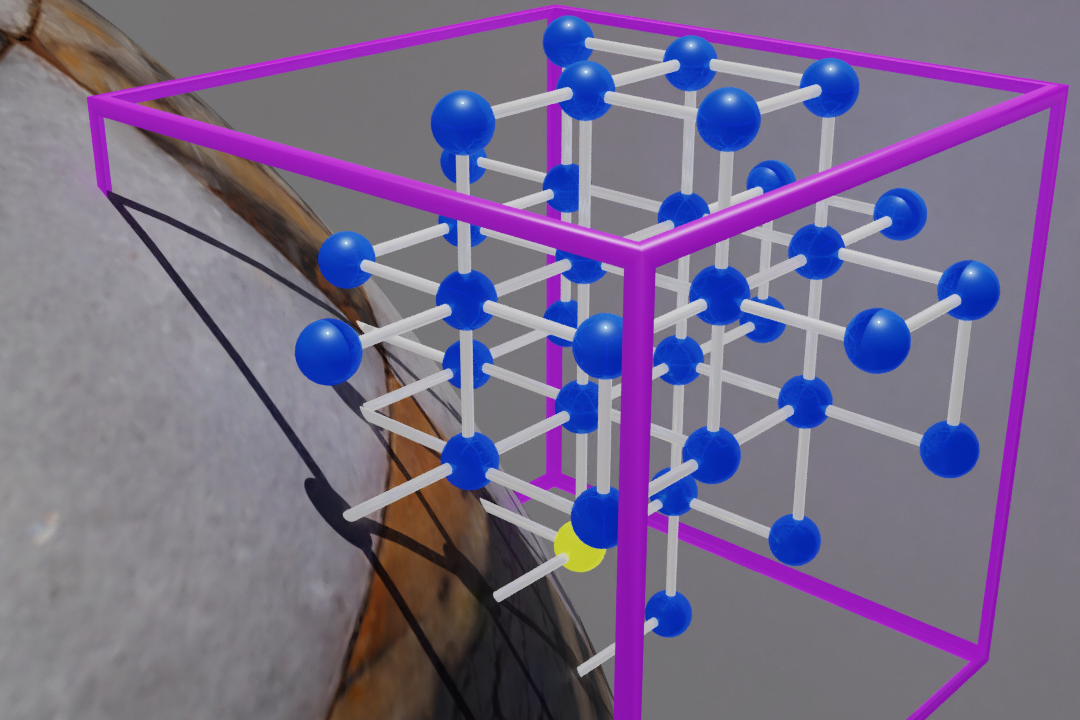}
}

\subfigure[$n=3$, the box corner is at $(1.0, 1.0, -0.5)$]{
\includegraphics[width=.47\linewidth]{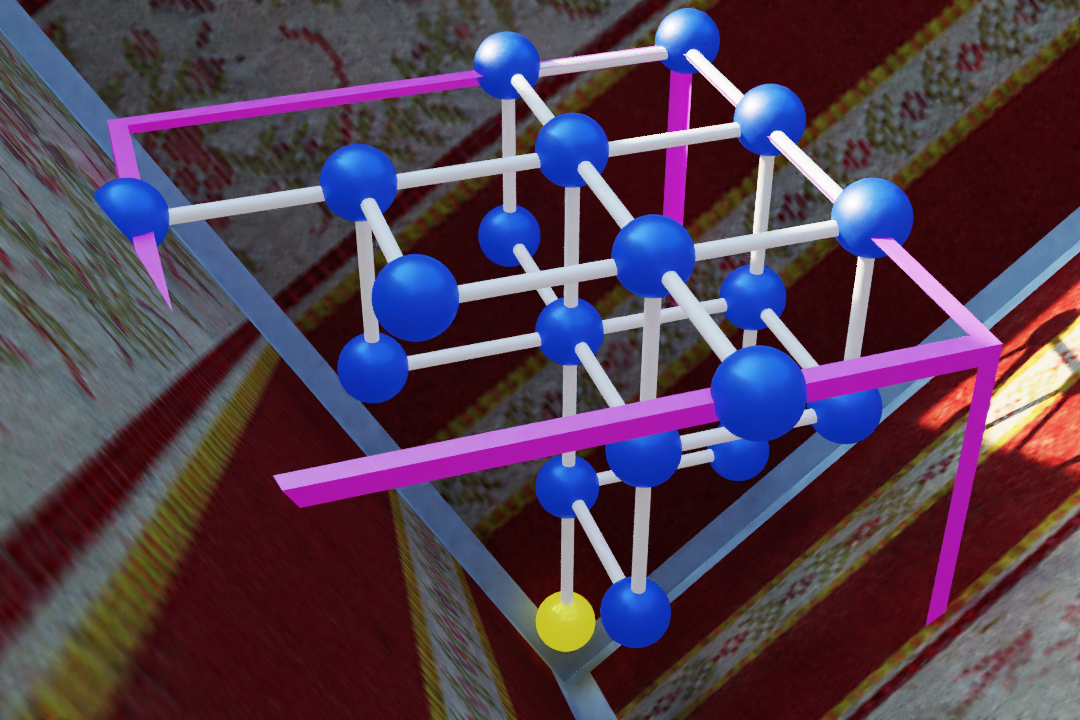}
}
\hfill
\subfigure[$n=4$, the box corner is at $(2.0, 2.0, -0.5)$]{
\includegraphics[width=.47\linewidth]{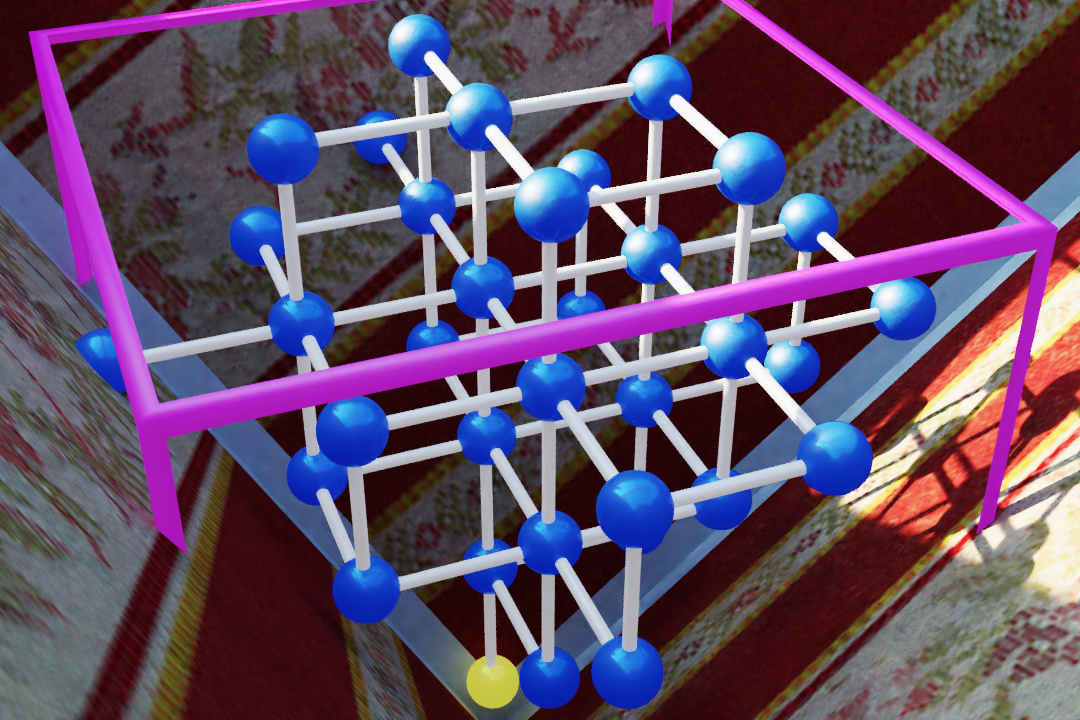}
}
\caption{Poised lattices generated
  by the TLG algorithm in Definition \ref{def:algorithm-plg}
  with the compactness-first test ordering (\ref{eq:testOrderingOneNorm}).
 In all cases,
  the starting point $\mathbf{q}$ is represented by the yellow ball
  and the cube $\mathbb{Z}^3_n$ is shifted to coincide with the blue frame.
  All smooth boundaries are patches of the ellipsoid in (\ref{eq:ellipsoid}).
  In subplots (e,f), 
   the box is obtained by rotating the unit cube with Euler angles
   $\frac{\pi}{4}$.
Note that any slice of the lattice along any dimension
 is a two-dimensional triangular lattice.
For $\Dim=3$ and $n=3$, the result of the TLG algorithm with
 the feasibility-first test ordering
 (\ref{eq:testOrderingFeasibility})
 is shown in Figure \ref{fig:geometricFD}(c).
 \revise{
 The yellow ball in subplot (c) represents 
 an irregular interior node, c.f. Definition \ref{def:PLG-FD}.}
 }
\label{fig:TLG_3D}
\end{figure}

By Definition \ref{def:triangularLatticeDimD}, 
 all lattices shown in Figures \ref{fig:TLG_2D} and \ref{fig:TLG_3D}
 are indeed triangular lattices
because each lattice contains ${n+\Dim \choose n}$ nodes
and projecting these nodes to each axis
yields exactly $n+1$ distinct coordinates.

We further test how the number of redundant points
 in a least-square stencil affects the condition number
 of the sample matrix $M$ in (\ref{eq:sampleMatrix})
 for the least-square problem.
Starting with a poised lattice in Figure \ref{fig:TLG_2D} or \ref{fig:TLG_3D},
 we add redundant points one at a time into the stencil, 
 calculate the corresponding condition number of $M$, 
 and present the results in Figure \ref{fig:TLG_conditioning},
 in which we observe several prominent features.
First,
 the condition number is reduced
 as we increase the number of redundant points; 
 there exists a (typically flat) minimum of condition numbers.
Second,
 further adding redundant points into the stencil
 might lead to an increase of the condition number, 
 thus a large number of redundant points in a least square stencil 
 might be detrimental to both efficiency and conditioning. 
Third,
 since a poised lattice has no redundant points,
 it admits a fine control of the least square stencil 
 to strike a good balance between efficiency and conditioning. 

 \begin{figure}
\centering
\subfigure[for Figure \ref{fig:TLG_2D}(c) with $\Dim=2$ and $n=3$]{
\includegraphics[width=.47\linewidth]{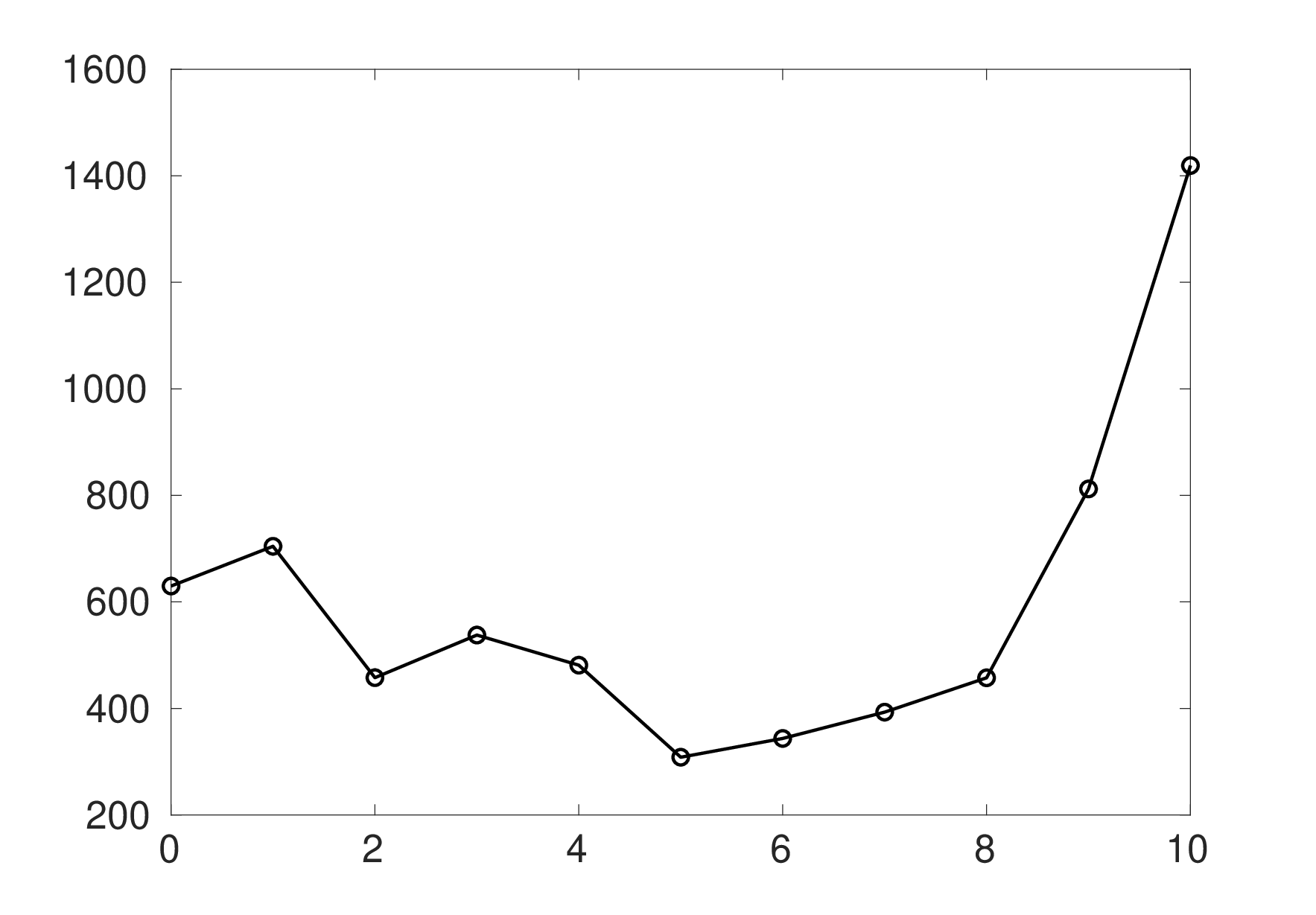}
}
\hfill
\subfigure[for Figure \ref{fig:TLG_2D}(e) with $\Dim=2$ and $n=4$]{
\includegraphics[width=.47\linewidth]{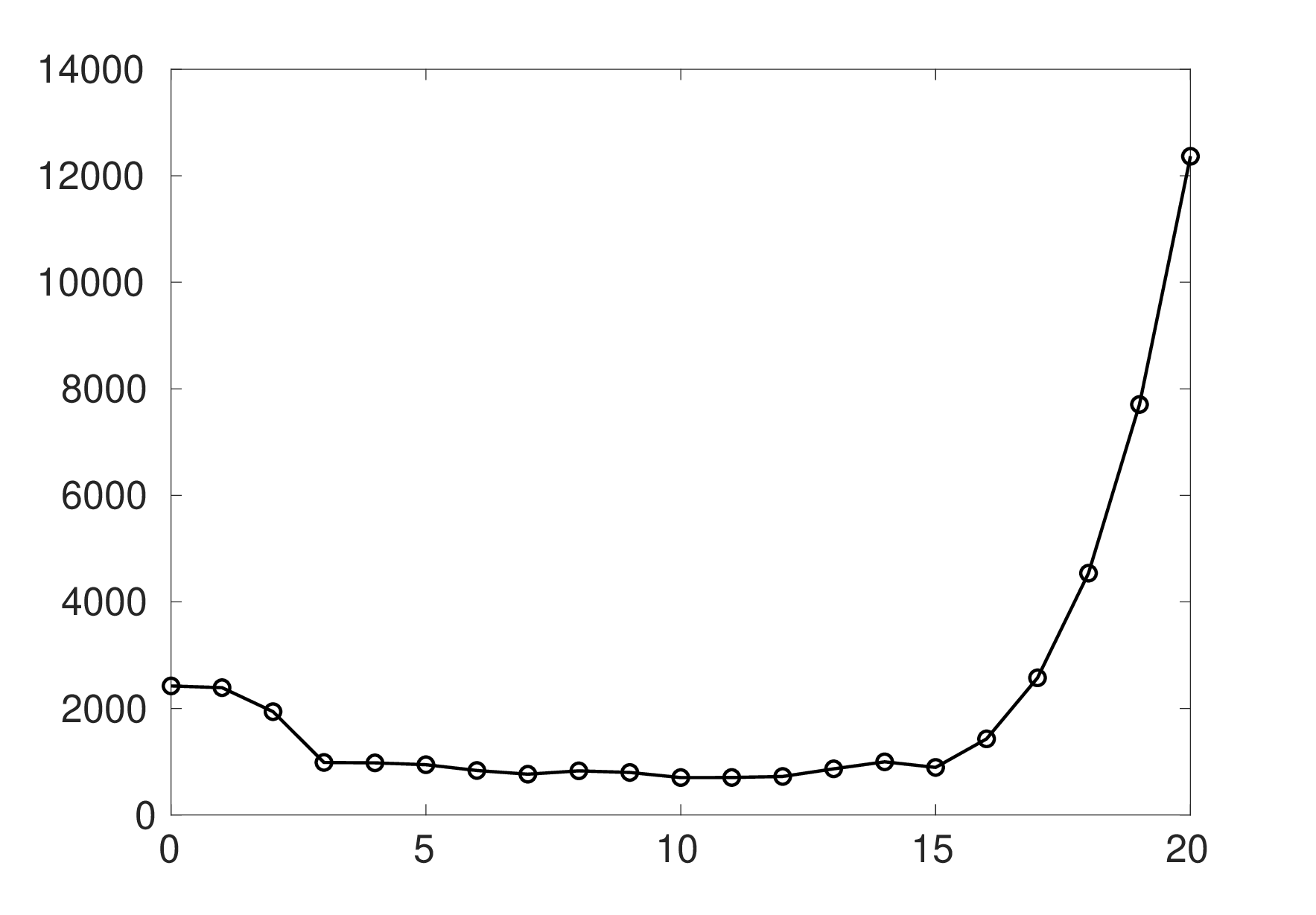}
}

\subfigure[for Figure \ref{fig:TLG_3D}(b) with $\Dim=3$ and $n=4$]{
\includegraphics[width=.47\linewidth]{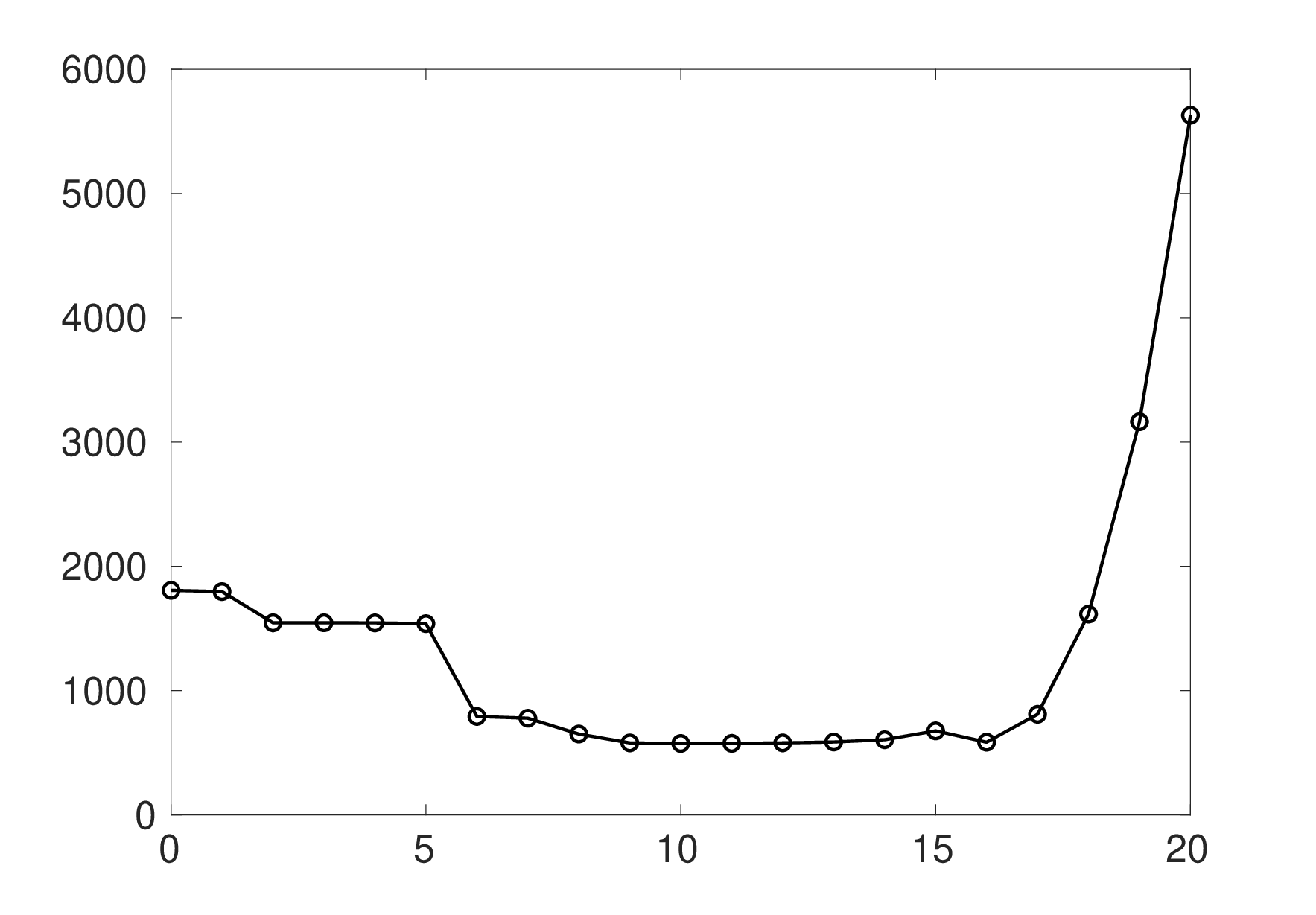}
}
\hfill
\subfigure[for Figure \ref{fig:TLG_3D}(e) with $\Dim=3$ and $n=3$]{
\includegraphics[width=.47\linewidth]{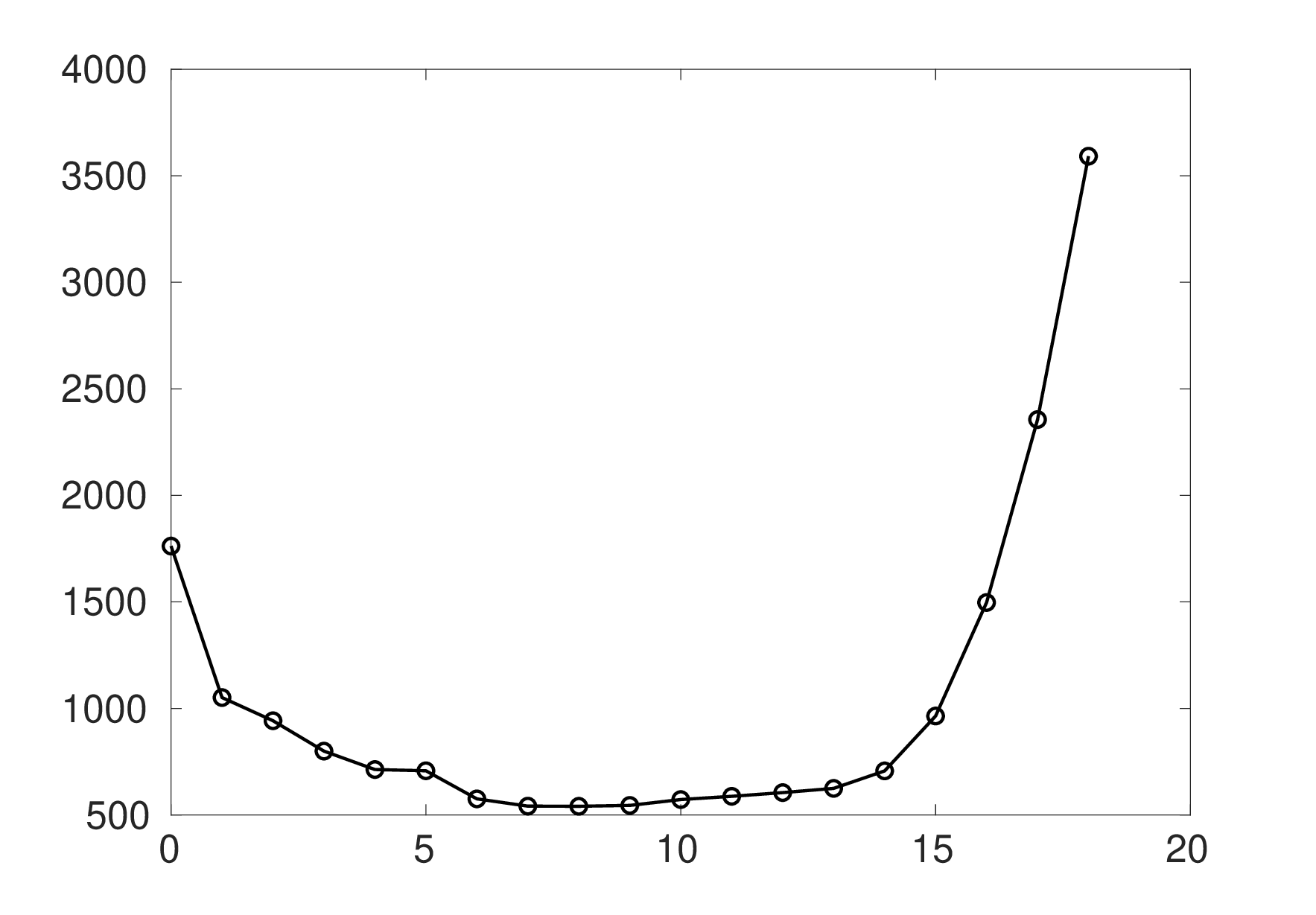}
}
\caption{Condition numbers of the sample matrix (\ref{eq:sampleMatrix})
  for least square stencils based on poised lattices. 
  Each subplot is generated by adding redundant points, one at a time, 
  into a poised lattice
  in Figure \ref{fig:TLG_2D} or Figure \ref{fig:TLG_3D}
  and computing the resulting condition number.
  The abscissa is the number of redundant points
  and the ordinate is the condition number of the sample matrix.
  Redundant points with smaller Manhattan distance to the starting
  point are added before those with larger Manhattan distances.
 }
\label{fig:TLG_conditioning}
\end{figure}

\begin{table}
  \caption{CPU time of TLG discretization of $\nabla \cdot (\mathbf{u} \mathbf{u})$
    on an Intel Xeon CPU E5-2698 v3 at 2.30GHz.
}
\label{tab:timeConsumption}
\centering
\renewcommand{\arraystretch}{1.28}
\begin{tabular}{c|c|c|c|c|c|c}
  \hline
  & & & number of & \multicolumn{2}{c}{CPU time in seconds} &
  \\ \cline{5-7}
  $\Dim$ & $n$ & $h$ & TLG calls & test sets of type I & test sets of type
                                                       II & ratio \\
  \hline
$2$ & 4 & $\frac{1}{64}$ & 192 & 5.54e$-$3 & 6.49e$-$4 & 8.54e+0 \\
 \hline
$2$ & 6 & $\frac{1}{64}$ & 288 & 6.83e+0 & 1.66e$-$2 & 4.11e+2 \\
\hline
$2$ & 8 & $\frac{1}{64}$ & 384 & 2.99e+4 & 1.61e+1 & 1.86e+3 \\
\hline
$2$ & 8 & $\frac{1}{128}$ & 768 & 4.62e+4 & 3.20e+1 & 1.44e+3 \\
\hline
$3$ & 4 & $\frac{1}{32}$ & 1696 & 2.03e+0 & 1.33e$-$2 & 1.53e+2 \\
 \hline
$3$ & 6 & $\frac{1}{32}$ & 2549 & 4.19e+4 & 1.81e+1 & 2.31e+3 \\
\hline
\end{tabular}
\end{table}%

\subsection{The TLG discretization of continuous operators}
\label{sec:testDiscretization}

We start with a reiteration of the TLG discretization,
 i.e., step (c) of PLG-FD in Definition \ref{def:PLG-FD}. 
For each irregular FD node $\mathbf{x}_{\mathbf{j}}$,
 we set $\mathbf{q}=\mathbf{x}_{\mathbf{j}}$,
 determine $K$ from $\mathbf{q}$ and $n$
 using the (CFS) strategy in Section \ref{sec:how-choose-feasible}, 
 invoke the TLG algorithm to solve the TLG problem, 
 \revise{adding 0--6 extra nodes into
 the generated poised lattice
 ${\mathcal T}_{\mathbf{j}}=\{\mathbf{x}_{\mathbf{i}}\}$, 
 solve (\ref{eq:fittingLinearSystem}) 
 to express $\mathbf{c}_\mathbf{i}$
 as a linear combination of the symbols
 in $\mathbf{U}_{\mathbf{S}(\mathbf{i})}$,
 substitute $\mathbf{c}_\mathbf{i}$ into
 (\ref{eq:fittingPoly})
 to obtain a multivariate polynomial $p(\mathbf{x})$
 whose coefficients are \emph{symbols} 
 that represent $u(\mathbf{x}_{\mathbf{i}})$'s,
 i.e. values of the unknown function $u$ at FD nodes.}
Then we take derivatives of $p(\mathbf{x})$
 as dictated by ${\mathcal L}u$ 
 to obtain a discrete form
 $Lu(\{\mathbf{x}_{\mathbf{i}}\})\approx
 {\mathcal L}u(\mathbf{x}_{\mathbf{j}})$.

For a spatial operator ${\mathcal L}$ on a scalar function $u$, 
 the \emph{truncation error for the TLG discretization of} ${\mathcal L}u$
 at an irregular FD node $\mathbf{x}_{\mathbf{j}}$ is
 \begin{equation}
   \label{eq:truncationErrors}
   E_T({\mathcal L}u, \mathbf{x}_{\mathbf{j}}) :=
   \left|Lu(\{\mathbf{x}_{\mathbf{i}}\}) -
     {\mathcal L}u(\mathbf{x}_{\mathbf{j}})\right|,
\end{equation}
where $\{\mathbf{x}_{\mathbf{i}}\}$ is the poised lattice
of $\mathbf{x}_{\mathbf{j}}$ generated by the TLG algorithm
and $L$ is the TLG discretization of $\mathcal{L}$
as described in the previous paragraph.

In this test, the irregular boundaries in two and three dimensions
 are respectively the ellipse and the ellipsoid
 in (\ref{eq:ellipse}) and (\ref{eq:ellipsoid}).
The scalar function $u$ is set to be each
  component of a velocity field $\mathbf{u}$,
which, in two and three dimensions, is respectively given by
 \begin{align}
& \mathbf{u}(x, y) = \left(
 \begin{array}{c}
\sin^2(\pi x) \sin(2 \pi y) \\  
-\sin(2 \pi x) \sin^2(\pi y)
 \end{array}
 \right), \\
& \mathbf{u}(x, y, z) = \frac{1}{2} \left(
 \begin{array}{c}
\sin^2(\pi x) \sin(2 \pi y) \sin(2 \pi z) \\  
\sin(2 \pi x) \sin^2(\pi y) \sin(2 \pi z) \\
-2 \sin(2 \pi x) \sin(2 \pi y) \sin^2(\pi z)
 \end{array}
 \right).
 \end{align}

\begin{table}
  \caption{Truncation errors for the TLG discretization of $\nabla \cdot (\mathbf{u} \mathbf{u})$;
    norms are taken over all components of $\mathbf{u}$.
    The two-dimensional domain is the unit square with an ellipse removed,
    prescribed by (\ref{eq:ellipse}) with $(x_0,y_0)=(\frac{1}{2}, \frac{1}{2})$ and
    $(a,b)=(\frac{1}{4}, \frac{1}{8})$.
    The three-dimensional domain is the unit cube with an ellipsoid removed,
    prescribed by (\ref{eq:ellipsoid}) with
    $(x_0,y_0,z_0) = (\frac{1}{2}, \frac{1}{2}, \frac{1}{2})$ and
    $(a,b,c) = (\frac{1}{4}, \frac{1}{8}, \frac{1}{4})$.
  }
  \label{tab:cvt}
  \centering
\begin{tabular}{cccccccc}
\hline
\multicolumn{8}{c}{$\Dim=2$, $n=4$} \\
\hline       
          &  $h=\frac{1}{32}$ &  rate &  $h=\frac{1}{64}$ &  rate & $h=\frac{1}{128}$ &  rate &  $h=\frac{1}{256}$  \\
\hline         
$L^\infty$ &   4.40e-03 &       3.51 &   3.87e-04 &       3.69 &   3.01e-05 &       3.84 &   2.10e-06 \\
     $L^1$ &   5.41e-04 &       4.08 &   3.20e-05 &       4.00 &   1.99e-06 &       4.01 &   1.24e-07 \\
     $L^2$ &   5.78e-04 &       4.18 &   3.19e-05 &       4.00 &   1.99e-06 &       4.05 &   1.21e-07 \\
\hline     
\multicolumn{8}{c}{$\Dim=2$, $n=6$} \\
\hline
         &  $h=\frac{1}{32}$ &  rate &   $h=\frac{1}{64}$ &  rate &  $h=\frac{1}{128}$ &   rate &  $h=\frac{1}{256}$  \\
\hline         
$L^\infty$ &   2.99e-04 &       6.81 &   2.66e-06 &       4.15 &   1.49e-07 &       5.30 &   3.80e-09 \\
     $L^1$ &   2.08e-05 &       6.22 &   2.80e-07 &       6.02 &   4.33e-09 &       6.02 &   6.67e-11 \\
     $L^2$ &   2.78e-05 &       6.59 &   2.87e-07 &       5.88 &   4.89e-09 &       6.05 &   7.38e-11 \\
\hline     
\multicolumn{8}{c}{$\Dim=3$, $n=4$} \\ 
\hline       
         &  $h=\frac{1}{32}$ &  rate &  $h=\frac{1}{64}$&  rate & $h=\frac{1}{128}$ & & \\
\hline         
$L^\infty$ &   6.74e-03 &       3.04 &   8.21e-04 &       4.50 &   3.63e-05 & & \\ 
     $L^1$ &   2.88e-04 &       4.01 &   1.79e-05 &       4.01 &   1.11e-06 & & \\ 
     $L^2$ &   2.85e-04 &       4.09 &   1.67e-05 &       4.07 &   9.94e-07 & & \\ 
\hline     
\multicolumn{8}{c}{$\Dim=3$, $n=6$} \\ 
\hline
         &  $h=\frac{1}{32}$ &  rate &  $h=\frac{1}{64}$&  rate & $h=\frac{1}{128}$ & & \\
\hline         
$L^\infty$ &   4.50e-04 &       4.78 &   1.64e-05 &       6.38 &   1.96e-07 & & \\ 
     $L^1$ &   9.72e-06 &       6.01 &   1.50e-07 &       6.03 &   2.30e-09 & & \\ 
     $L^2$ &   1.29e-05 &       6.11 &   1.87e-07 &       6.30 &   2.38e-09 & & \\ 
\hline
\end{tabular}
\end{table}%


In Sections \ref{sec:test-sets-first} and \ref{sec:test-sets-second},
 two types of test sets are given for use with the TLG algorithm.
In this test,
 we first compare CPU time of these two different test sets
 and list the results in Table \ref{tab:timeConsumption}.
Test sets of type II are clearly much more efficient
 than those of type I.
Furthermore,
 the ratio of improvement increases as $n$ grows, 
 but it appears to be insensitive
 to the grid size, c.f. 
\revise{
 the third and fourth rows in Table \ref{tab:timeConsumption}.
}
These observations confirm our conjecture that the backtracking efficiency
 can be greatly enhanced by enforcing in the spanning trees
 the algebraic structure of orbits of equivalence classes
 of the principal lattice.
 
When each component of $\mathbf{u}$ is approximated
 by a complete multivariate polynomial of degree no more than $n$,
 the leading error term of $\mathbf{u}\mathbf{u}$
 is $O(h^{n+1})$,
 and hence the truncation error of 
 the spatial operator ${\mathcal L}\mathbf{u}= \nabla \cdot (\mathbf{u} \mathbf{u})$
 should be $O(h^{n})$.
In Table \ref{tab:cvt},
truncation errors of TLG discretization for ${\mathcal L}\mathbf{u}$
 are listed with corresponding convergence rates; 
 in both two and three dimensions,
 fourth-order and sixth-order convergence rates 
 are clearly observed for TLG discretization
 with $n=4$ and $n=6$, respectively.
This verifies the correctness and effectiveness of the TLG algorithm.

\subsection{An elliptic equation with a cross-derivative term}
\label{sec:ellipticEqWithCrossDerivative}

In this test, our PLG-FD method is employed to numerically solve the elliptic equation
\begin{equation}
  \label{eq:secondOrderMixedDer}
  \left\{
    \begin{aligned}
      a \frac{\partial^2 u}{\partial x^2}
      + b \frac{\partial^2 u}{\partial x \partial y} + c \frac{\partial^2 u}{\partial y^2} &= f \quad \text{in}\ \Omega, \\
      u &= g \quad \text{on}\ \partial \Omega, 
    \end{aligned}
  \right.
\end{equation}
 where $\Omega$ is a simply-connected domain and 
 $u : \Omega \rightarrow \mathbb{R}$ is the unknown function.
When $b$ is not zero, the cross-derivative term
 could cause difficulty for FD methods based on one-dimensional FD formulas.
In this test, the righthand side $f$ and the boundary condition $g$ are 
 both derived from the exact solution
\begin{equation}
  u(x, y) = \sin(2 \pi x) \cos(2 \pi y),
\end{equation}
which is also used for calculating truncation and solution errors.


\begin{figure}
\centering
\subfigure[Solution on the unit square.]{
\includegraphics[width=.37\linewidth]{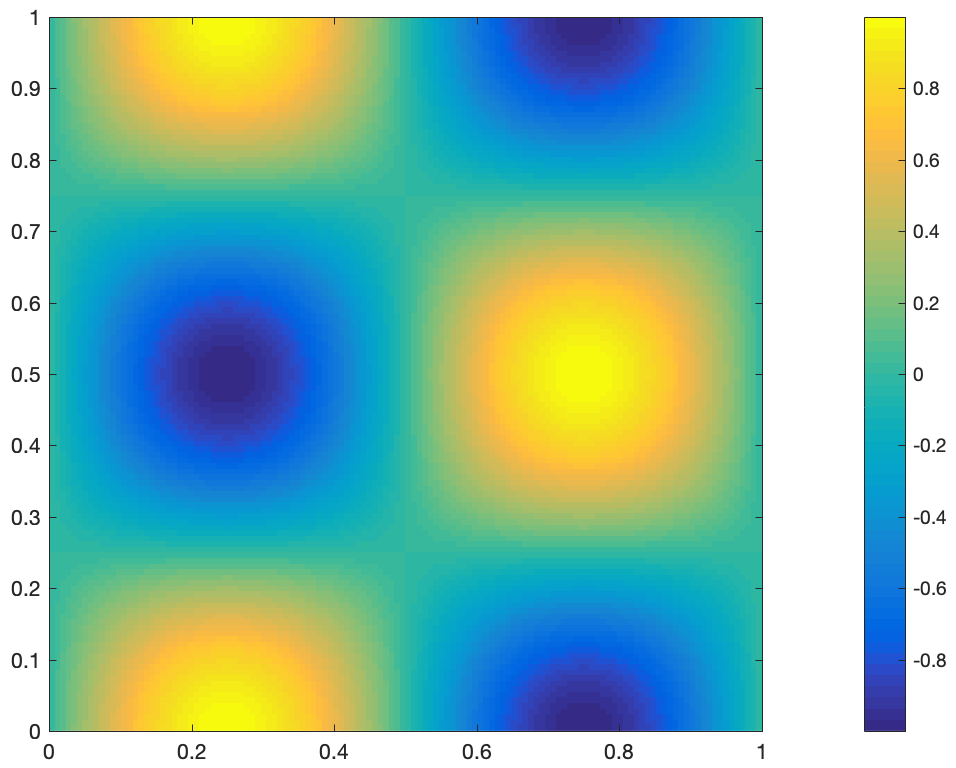}
}
\hfill
\subfigure[Solution on the unit square rotated by $\pi/6$. ]{
\includegraphics[width=.58\linewidth]{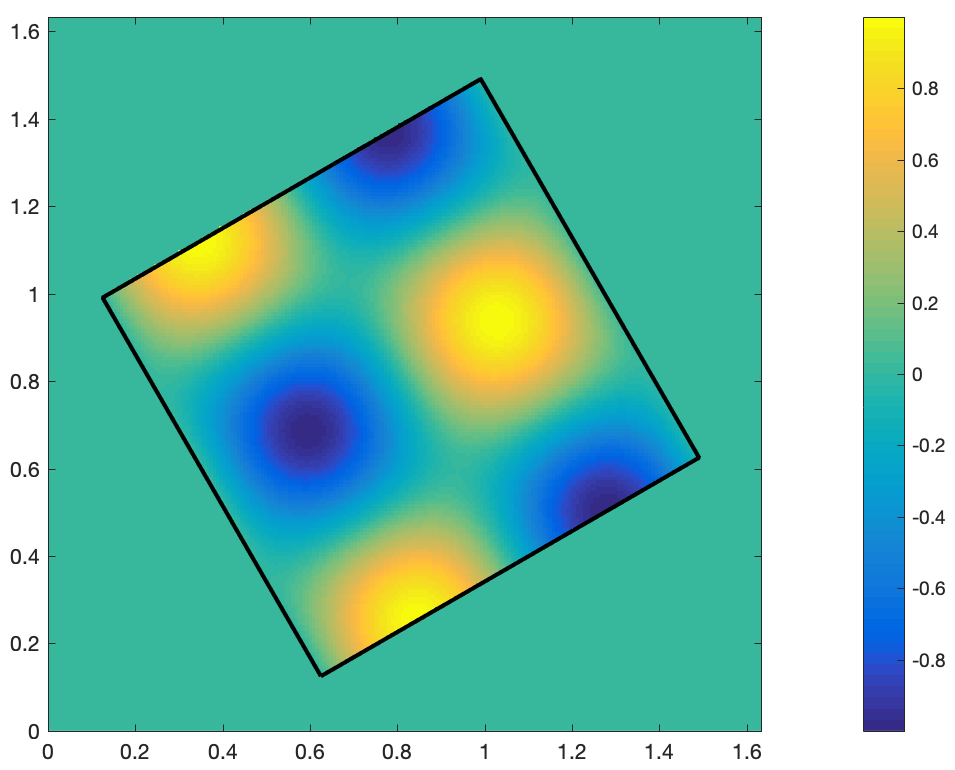}
}
\caption{Solutions of equation (\ref{eq:secondOrderMixedDer})
  by the PLG-FD method on different domains
  with the same grid size $h = \frac{1}{128}$.
  In subplot (a), the domain is the unit square
  where neither TLG nor TLG discretization is invoked;
   all discretizations are done
   via traditional one-dimensional FD formulas.
  In subplot (b), the domain is the rotated unit square
   where the TLG discretization is used for irregular FD nodes.
  See Table \ref{tab:mixedDerError} for an error comparison.
 }
\label{fig:mixedDerSolComparison}
\end{figure}

\begin{table}
  \caption{Truncation errors and solution errors of
    the PLG-FD method with $n=4$ in solving
    the elliptic equation (\ref{eq:secondOrderMixedDer}).}
\label{tab:mixedDerError}
\centering
\renewcommand{\arraystretch}{1.2}
\begin{tabular}{cccccccc}
  \hline
  \multicolumn{8}{c}{Truncation errors on the rotated square with 
  $(a, b, c) = (\frac{5}{4}, -\frac{\sqrt{3}}{2}, \frac{7}{4})$ 
  } \\
\hline
      &  $h=\frac{1}{64}$  &  rate &  $h=\frac{1}{128}$ &  rate & $h=\frac{1}{256}$ &  rate &  $h=\frac{1}{512}$ \\
\hline       
$L^\infty$ &   1.44e-01 &       2.58 &   2.42e-02 &       2.97 &   3.10e-03 &       3.36 &   3.01e-04 \\
     $L^1$ &   6.84e-04 &       4.02 &   4.22e-05 &       4.04 &   2.56e-06 &       4.02 &   1.57e-07 \\
     $L^2$ &   3.29e-03 &       3.60 &   2.71e-04 &       3.82 &   1.92e-05 &       3.85 &   1.33e-06 \\
\hline
\multicolumn{8}{c}{Solution errors on the rotated square with 
  $(a, b, c) = (\frac{5}{4}, -\frac{\sqrt{3}}{2}, \frac{7}{4})$} \\
\hline
     &  $h=\frac{1}{64}$  &  rate  &  $h=\frac{1}{128}$ &  rate &  $h=\frac{1}{256}$  &  rate  &  $h=\frac{1}{512}$  \\
\hline       
$L^\infty$ &   2.02e-05 &       6.08 &   2.99e-07 &       4.22 &   1.61e-08 &       4.01 &   1.00e-09 \\
     $L^1$ &   1.28e-06 &       4.11 &   7.42e-08 &       4.02 &   4.58e-09 &       4.00 &   2.87e-10 \\
     $L^2$ &   1.75e-06 &       4.09 &   1.03e-07 &       4.01 &   6.38e-09 &       4.00 &   3.99e-10 \\
\hline       
  \multicolumn{8}{c}{Solution errors on the unit square $[0,1]^2$
  with $(a, b, c) = (1, 0, 2)$
  } \\
\hline
     &  $h=\frac{1}{64}$  &  rate  &  $h=\frac{1}{128}$ &  rate &  $h=\frac{1}{256}$  &  rate  &  $h=\frac{1}{512}$  \\
\hline           
  $L^\infty$ &   1.24e-06 &       3.99 &   7.79e-08 &       4.00 &   4.88e-09 &       3.98 &   3.09e-10 \\
       $L^1$ &   3.30e-07 &       3.98 &   2.09e-08 &       3.99 &   1.31e-09 &       4.00 &   8.20e-11 \\
       $L^2$ &   4.81e-07 &       3.99 &   3.03e-08 &       3.99 &   1.90e-09 &       3.97 &   1.21e-10 \\
\hline       
\end{tabular}
\end{table}%

We design two test cases to examine how TLG discretization affects the
accuracy of the numerical solutions.
First, 
 we solve (\ref{eq:secondOrderMixedDer}) with $(a, b, c) = (1, 0, 2)$
 on the unit square $[0,1]^2$ and
 plot the solution in Figure \ref{fig:mixedDerSolComparison} (a). 
The regular geometry and the absence of the cross-derivative term
 admit a discretization via one-dimensional FD formulas at all FD nodes,
 and hence TLG is not invoked and no TLG discretization is necessary.
Second, we rotate the unit square by $\frac{\pi}{6}$
 and embed it in a larger square domain,
 on which the TLG discretization is applied to irregular FD nodes
 for equation (\ref{eq:secondOrderMixedDer}) with
 $(a, b, c) = (\frac{5}{4}, -\frac{\sqrt{3}}{2}, \frac{7}{4})$
 and the resulting linear system is solved by a geometric multigrid method;
 the solution is plotted in Figure \ref{fig:mixedDerSolComparison} (b). 
Since the exact solution is the same in both cases,
 the difference of the two computed solutions
 quantifies the influence of irregular geometry and
 cross-derivative terms.

As shown in Figure \ref{fig:mixedDerSolComparison},
 the two solutions are qualitatively the same.
In Table \ref{tab:mixedDerError},
we list truncation and solution errors
of our fourth-order PLG-FD method in solving (\ref{eq:secondOrderMixedDer})
on these two domains.
For truncation errors, 
 the convergence rates are above 3, 4, and 3.5
 in the max-norm, the 1-norm, and the 2-norm, respectively. 
This implies that the elliptic operator is approximated
 to third-order accurate on the irregular FD nodes of codimension one
 and to fourth-order accurate on the regular FD nodes.
In contrast,
 convergence rates of solution errors are very close to 4 in all norms, 
 due to the elliptic nature of the PDE. 
The difference of solution errors on the two domains
 seems to suggest that TLG discretization, 
 when dealing with complex geometry and cross-derivative terms,
 roughly triples the solution errors of traditional FD formulas.

\subsection{Poisson's equation with mixed Dirichlet and Neumann
  conditions}
\label{sec:PoissonEq}

In this subsection we demonstrate that 
 our PLG-FD method can be much more accurate than
 a second-order embedded boundary (EB) method \cite{johansen98:_cartes_grid_embed_bound_method}.

Consider a test \cite[Problem 3]{johansen98:_cartes_grid_embed_bound_method} of Poisson's equation
\begin{equation}
  \label{eq:Poisson}
  \frac{\partial^2 u}{\partial x^2} + \frac{\partial^2 u}{\partial y^2} = f \quad \text{in } \Omega, \\
\end{equation}
where $\Omega = \Omega_1 \cap \Omega_2$,
 $\Omega_1$ is the unit square centered at the origin and
\begin{equation*}
\Omega_2 = \left\{(r,\theta) : r \ge 0.25 + 0.05 \cos{6 \theta}
\right\}; 
\end{equation*}
 a Dirichlet boundary condition is imposed on $\partial \Omega_1$
 while a Neumann condition on $\partial \Omega_2$.
Both $f$ and the boundary conditions are derived 
 from the exact solution
 \begin{equation}
   \label{eq:PoissonSolution}
   u(r, \theta) = r^4 \cos{3 \theta},
 \end{equation}
 where $(r,\theta)$ are polar coordinates satisfying
 $(x,y)=(r\cos\theta, r\sin\theta)$.

Although no cross-derivative term is present in the equation,
 TLG discretization is still necessary
 in order to fulfill the Neumann boundary condition near $\Omega_2$; 
 see step (c) of Definition \ref{def:PLG-FD}.

As shown in Figure \ref{fig:PoissonErrorPlot},
 solution errors of our PLG-FD method
 for $h=\frac{1}{160}$
 are dominated by those near the curvilinear boundary.
This is not unexpected because, 
 while traditional FD discretizations on symmetric stencils
 have fortuitous cancellations of leading error terms,
 the TLG discretization near an irregular boundary does not.
Hence 
 truncation errors of the TLG discretization near the boundary
 are typically much higher than traditional FD discretization
 on regular FD nodes.
Despite different types of discretization across the FD nodes,
 the solution errors appear smooth on the entire domain $\Omega$.
 
To compare the accuracy of our fourth-order method to
 the second-order EB method by Johansen and Colella
 \cite{johansen98:_cartes_grid_embed_bound_method}, 
 we show truncation errors and solution errors of these two methods
 in Table \ref{tab:PoissonErrors}.
The convergence rates of solution errors of our method
 are very close to four, meeting our expectations.
Furthermore,
 our method is much more accurate than the EB method.
In particular, the max-norm of solution errors of our method
 on the grid of $h=\frac{1}{80}$ 
 is smaller by a factor of twenty
 than that of the EB method on the finest grid of $h=\frac{1}{320}$.
Also,  for $h=\frac{1}{320}$, 
 the error norm of our method
 is much less than that of the EB method
 by a factor of 5350.

\begin{figure}
\centering
\includegraphics[width=.6\linewidth]{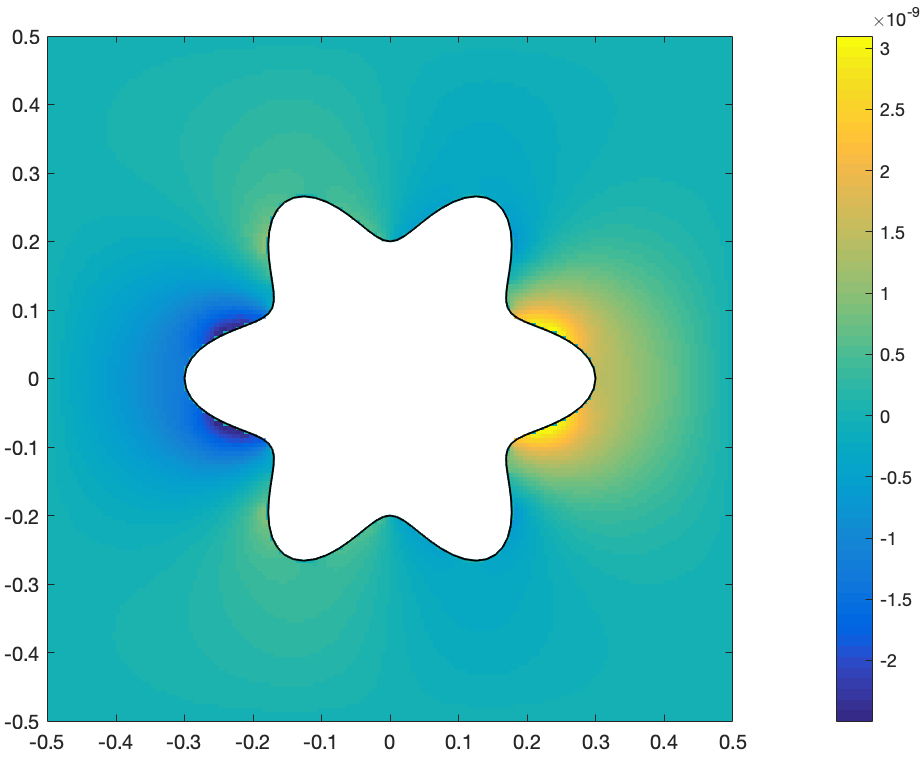}
\caption{Solution errors of the fourth-order PLG-FD method
  in solving (\ref{eq:Poisson}) with $h = \frac{1}{160}$.
  A Dirichlet condition and a Neumann condition are enforced on the square
  and the curvilinear boundary, respectively.
}
\label{fig:PoissonErrorPlot}
\end{figure}

\begin{table}
  \caption{Errors of the fourth-order PLG-FD method
    and a second-order EB method \cite{johansen98:_cartes_grid_embed_bound_method}
    in solving Poisson's equation (\ref{eq:Poisson})
    with mixed Dirichlet and Neumann conditions. }
\label{tab:PoissonErrors}
\centering
\begin{tabular}{cccccccc}
\hline
\multicolumn{8}{c}{Truncation errors of the EB method by Johansen and
  Colella \cite{johansen98:_cartes_grid_embed_bound_method}} \\
\hline
      &  $h=\frac{1}{40}$ &  rate &  $h=\frac{1}{80}$ &   rate &  $h=\frac{1}{160}$ &   rate &      $h=\frac{1}{320}$ \\
\hline
$L^\infty$ &   1.66e-03 &       2.0 &   4.15e-04 &       2.0 &   1.04e-04 &       2.0 &   2.59e-05 \\
\hline
\multicolumn{8}{c}{Truncation errors of our fourth-order PLG-FD method} \\
\hline
      &  $h=\frac{1}{40}$ &  rate &  $h=\frac{1}{80}$ &   rate &  $h=\frac{1}{160}$ &   rate &      $h=\frac{1}{320}$ \\
\hline       
$L^\infty$ &   1.64e-03 &       4.57 &   6.91e-05 &       2.03 &   1.69e-05 &       2.82 &   2.39e-06 \\
     $L^1$ &   1.34e-05 &       4.27 &   6.94e-07 &       3.65 &   5.53e-08 &       4.07 &   3.30e-09 \\
     $L^2$ &   8.98e-05 &       4.51 &   3.95e-06 &       2.98 &   5.01e-07 &       3.56 &   4.24e-08 \\
\hline       
\hline       
\multicolumn{8}{c}{Solution errors of the EB method by Johansen and
  Colella \cite{johansen98:_cartes_grid_embed_bound_method}} \\
\hline
      &  $h=\frac{1}{40}$ &  rate &  $h=\frac{1}{80}$ &   rate &  $h=\frac{1}{160}$ &   rate &      $h=\frac{1}{320}$ \\
\hline 
$L^\infty$ &   4.78e-05 &       1.85 &   1.33e-05 &       1.98 &   3.37e-06 &       1.95 &   8.72e-07 \\
\hline
\multicolumn{8}{c}{Solution errors of our fourth-order PLG-FD method} \\
\hline
      &  $h=\frac{1}{40}$ &  rate &  $h=\frac{1}{80}$ &   rate &  $h=\frac{1}{160}$ &   rate &      $h=\frac{1}{320}$ \\
\hline       
$L^\infty$ &   4.37e-06 &       6.62 &   4.43e-08 &       4.00 &   2.76e-09 &       4.08 &   1.63e-10 \\
     $L^1$ &   4.16e-07 &       7.46 &   2.36e-09 &       2.94 &   3.07e-10 &       4.22 &   1.65e-11 \\
     $L^2$ &   7.59e-07 &       7.35 &   4.64e-09 &       3.26 &   4.84e-10 &       4.18 &   2.66e-11 \\
\hline       
\end{tabular}
\end{table}%

In Table \ref{tab:PoissonTime}, 
 we report CPU time of the TLG discretization
 and that of all steps of our fourth-order PLG-FD method
 in solving (\ref{eq:Poisson}). 
It is found that CPU time spent on TLG discretization
 is no more than 10\% of that on the entire solution process; 
 furthermore, this percentage decreases as the grid size is reduced. 
Invoked only at a set of codimension one,
 TLG discretization has a complexity of $O(\frac{1}{h^{\Dim-1}})$. 
In contrast, 
 the complexity of multigrid solvers 
 is $O(\frac{1}{h^{\Dim}})$. 
Hence 
 the cost of TLG discretization is asymptotically negligible
 as far as CPU time of the entire solution process is concerned.
 
However, the discussion in the previous paragraph does not diminish the importance
 of test sets of type II: 
 combined with Table \ref{tab:timeConsumption}, Table
 \ref{tab:PoissonTime} also suggests that,
 on a fixed grid,
 CPU time of the entire PLG-FD method
 could have been dominated by that of TLG discretization
 if test sets of type I had been used.

\begin{table}
  \caption{Timing results in solving Poisson's equation
    (\ref{eq:Poisson})
    with the fourth-order PLG-FD method
    on an Intel Xeon E5-2698 v3 at 2.30GHz.
    The middle two columns are CPU time in seconds.
  }
\label{tab:PoissonTime}
\centering
\renewcommand{\arraystretch}{1.2}
\begin{tabular}{c|c|c|c}
\hline
 & TLG discretization & the entire solver & ratio \\
\hline
$h=\frac{1}{80}$  & 0.007 & 0.079 & 8.9\%\\
$h=\frac{1}{160}$ & 0.013 & 0.277 & 4.7\%\\
$h=\frac{1}{320}$ & 0.032 & 1.129 & 2.8\%\\
\hline
\end{tabular}
\end{table}%


\subsection{Poisson's equation in three dimensions}
\label{sec:PoissonEq3D}

Consider a test of Poisson's equation
\begin{equation}
    \frac{\partial^2 u}{\partial x^2} + \frac{\partial^2 u}{\partial y^2}
    + \frac{\partial^2 u}{\partial z^2} = f \quad \text{in } \Omega', \\
    \label{eq:PoissonEq3D}
\end{equation}
where $\Omega' = \Omega_3 \cap \Omega_4$,
$\Omega_3 = [0,1]^3$ is the unit box and
\begin{equation}
    \Omega_4 = \left\{ (x,y,z) \in \mathbb{R}^3 :
    \left( \frac{x-0.5}{0.25} \right)^2
    + \left( \frac{y-0.5}{0.125} \right)^2
    + \left( \frac{z-0.5}{0.25} \right)^2 \ge 1
    \right\}.
\end{equation}
Dirichlet conditions are imposed on both $\partial \Omega_3$ and $\partial \Omega_4$.
The source term $f$ and the boundary conditions are derived from the exact solution
\begin{equation}
    u(x, y, z) = \sin(2 \pi x) \cos(2 \pi y) \sin(2 \pi z).
\end{equation}

\revise{
In step (c) of PLG-FD in Definition \ref{def:PLG-FD},
 we update the poised lattice $\mathbf{S}(\mathbf{i})$
 with $\mathbf{S}(\mathbf{i})\cup\{\bmj \in \mathbb{Z}^D: 
 \Vert \bmj - \mathbf{q} \Vert_1 \le 2\}$.
Consequently, 
 for most irregular FD nodes
 the linear system \eqref{eq:fittingLinearSystem} is overdetermined
 and must be solved in the least square sense. 
It is found that this choice of $\mathbf{S}(\mathbf{i})$
 leads to substantially smaller solution errors; 
 see Table \ref{tab:PoissonErrorsIn3D}.
 Fourth-order convergence rates are also observed. 
}

\begin{table}
    \caption{Errors of the fourth-order PLG-FD method
      in solving a three-dimensional Poisson's equation (\ref{eq:PoissonEq3D})
      with Dirichlet conditions. }
    \label{tab:PoissonErrorsIn3D}
    \centering
    \begin{tabular}{cccccccc}
      \hline
      \multicolumn{8}{c}{Truncation errors} \\
      \hline
      &  $h=\frac{1}{32}$ &   rate &  $h=\frac{1}{64}$ &   rate &  $h=\frac{1}{128}$ &       rate &  $h=\frac{1}{256}$  \\
      \hline
      $L^\infty$ &   2.40e-01 &       2.50 &   4.24e-02 &       2.68 &   6.63e-03 &       3.01 &   8.23e-04 \\
      $L^1$ &   1.93e-03 &       3.85 &   1.34e-04 &       4.06 &   8.04e-06 &       4.01 &   4.98e-07 \\
      $L^2$ &   1.07e-02 &       3.18 &   1.18e-03 &       3.54 &   1.01e-04 &       3.48 &   9.06e-06 \\
      \hline
      \multicolumn{8}{c}{Solution errors} \\
      \hline
      &  $h=\frac{1}{32}$ &   rate  &  $h=\frac{1}{64}$ &  rate &  $h=\frac{1}{128}$ &   & \\
      \hline
      $L^\infty$ &   9.01e-05 &       4.09 &   5.30e-06 &       4.62 &   2.15e-07 &      &     \\
      $L^1$ &   3.80e-06 &       4.40 &   1.80e-07 &       4.25 &   9.42e-09 &     &      \\
      $L^2$ &   7.57e-06 &       4.54 &   3.26e-07 &       4.47 &   1.47e-08 &     &      \\
      \hline
    \end{tabular}
  \end{table}

\revise{
\subsection{The heat equation with a Robin boundary condition}
\label{sec:heatEq}

In this subsection, 
 we use PLG-FD to solve the heat equation
\begin{equation}
  \label{eq:heatEquation}
    \begin{array}{rll}
      \frac{\partial \varphi}{\partial t}
      &= \nu\Delta \varphi + f \quad &\text{in}\ \Omega,
      \\
      \nu\frac{\partial\varphi}{\partial \mathbf{n}}
      + a\varphi&= g\quad &\text{on}\ \partial \Omega,
    \end{array}
\end{equation}
where $\Omega$ is a simply-connected domain,
$\varphi$ the unknown function,
$f, g, a$ given functions,
$\nu$ the diffusion coefficient,
and $\mathbf{n}$ the unit outward normal of $\partial\Omega$.

\begin{figure}
\centering
\subfigure[The initial condition at $t_0=0.0$.]{
\includegraphics[width=.45\linewidth]{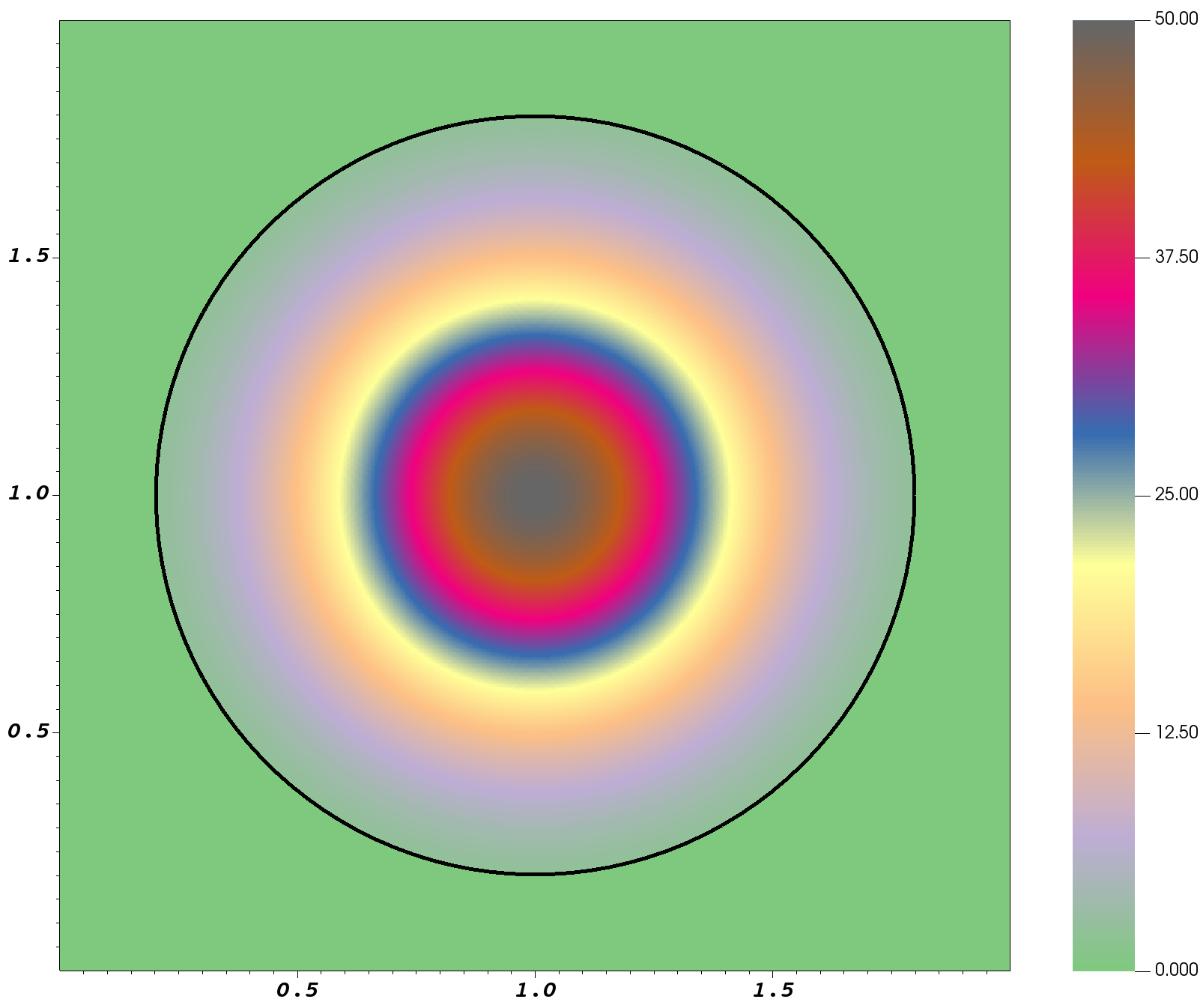}
}
\hfill
\subfigure[The final solution at $T=1.0$.]{
\includegraphics[width=.45\linewidth]{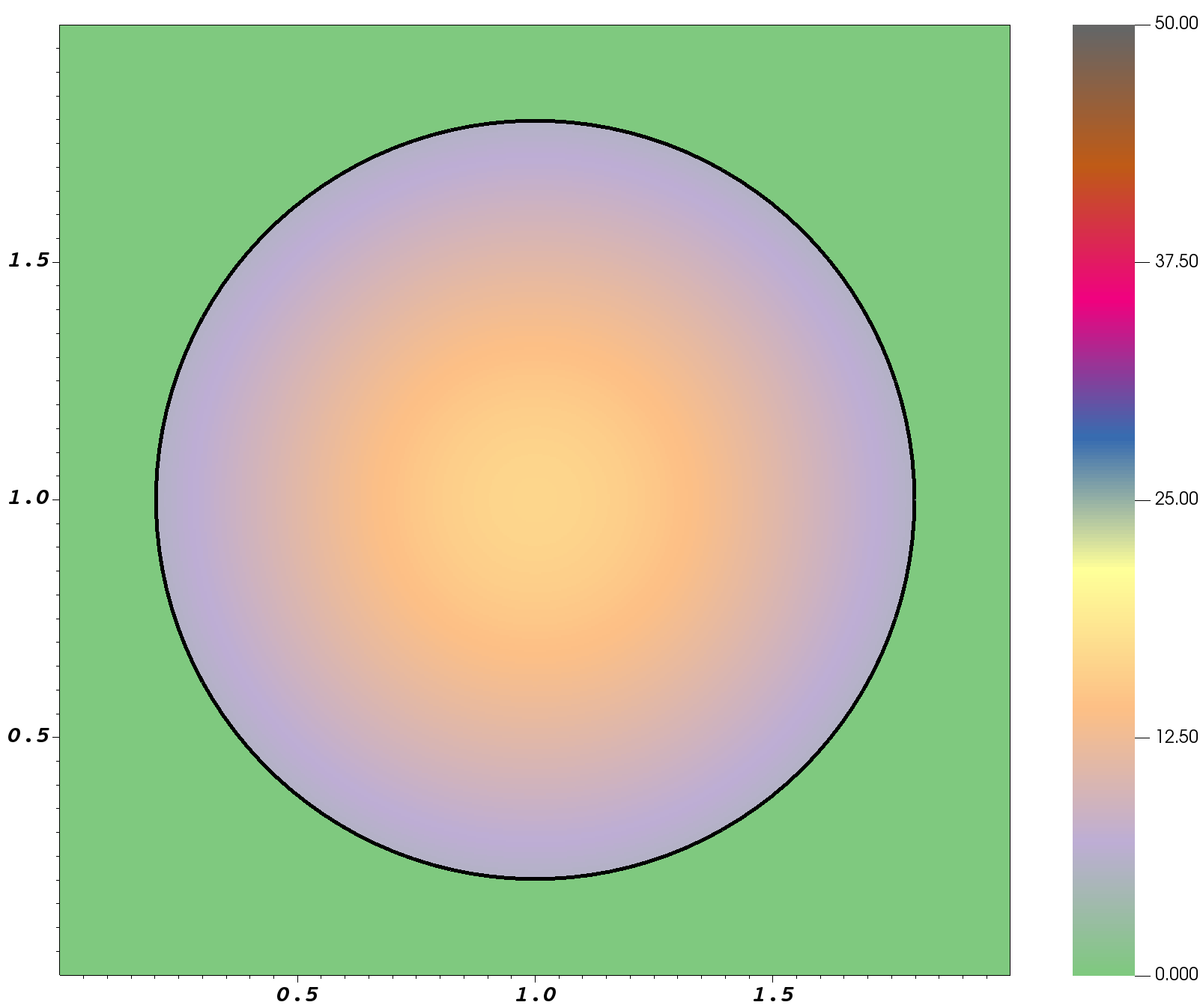}
}
\caption{Results of PLG-FD in solving the heat equation (\ref{eq:heatEquation}) with $h = \frac{1}{256}$ and $\Delta t = 2h$.
 }
\label{fig:heatEq}
\end{figure}

\begin{table}
  \caption{Errors and convergence rates
    of the fourth-order PLG-FD method
    coupled with an ESDIRK method,
    in solving the heat equation (\ref{eq:heatEquation})
    with a Robin boundary condition. }
    \label{tab:HeatEqWithRobinCondition}
    \centering
    \begin{tabular}{cccccccc}
        \hline
        &  $h=\frac{1}{32}$  &   rate &  $h=\frac{1}{64}$ &    rate &  $h=\frac{1}{128}$ &       rate &  $h=\frac{1}{256}$ \\
        \hline
        $L^\infty$ &   8.69e-05 &       4.66 &   3.43e-06 &       3.90 &   2.30e-07 &       4.12 &   1.33e-08 \\
             $L^1$ &   5.70e-05 &       4.88 &   1.95e-06 &       3.78 &   1.41e-07 &       4.17 &   7.85e-09 \\
             $L^2$ &   4.27e-05 &       4.86 &   1.48e-06 &       3.81 &   1.06e-07 &       4.16 &   5.90e-09 \\
        \hline
    \end{tabular}
\end{table}

The diffusion of a point source
 in a disk of radius $R$ is described by 
 \begin{equation}
   \label{eq:diffusionOnDisk}
   \begin{array}{l}
   \varphi(x,y,t) = \frac{10}{4\nu(t+\frac{1}{2})}
   \exp\left(-\frac{(x-x_0)^2+(y-y_0)^2}{4\nu(t+\frac{1}{2})}\right),
   \end{array}
 \end{equation}
where $(x_0, y_0)$ is the center of the disk.
We set $a = 1$, $R=0.8$, $(x_0,y_0) = (1,1)$,
 $\nu=0.1$,
 and determine the form of $f$ and $g$
 by plugging (\ref{eq:diffusionOnDisk}) into (\ref{eq:heatEquation}). 

We couple the fourth-order
explicit singly diagonal implicit Runge-Kutta (ESDIRK) scheme
\cite[p. 176]{kennedy03:_additive_runge_kutta},
 and the fourth-order PLG-FD method in Definition \ref{def:PLG-FD}
 to obtain a simple method
 as described by (PFT-1,2) in Section \ref{sec:PLG-FD}. 
Figure~\ref{fig:heatEq} shows the initial and final numerical solutions 
with the grid size $h=\frac{1}{256}$
and the time step size $\Delta t = 2h$.
Errors and convergence rates at the final time $T=1.0$
are presented in Table~\ref{tab:HeatEqWithRobinCondition}, 
where all convergence rates
in the $L^1$, $L^2$, and $L^{\infty}$ norms
are close to 4. 
}

\revise{
\subsection{The advection-diffusion equation with
 homogenuous Neumann conditions}
\label{sec:adv-diffu-eq}

In this subsection,
we test PLG-FD by solving the advection-diffusion equation
\begin{equation}
  \label{eq:adv-diff-eq}
  \begin{array}{rll}
    \frac{\partial\varphi}{\partial t}
    + \mathbf{u}\cdot\nabla\varphi
    &= \nu\Delta\varphi \quad &\text{in } \Omega;
    \\
    \frac{\partial \varphi}{\partial \mathbf{n}} &= 0
    \quad &\text{on } \partial\Omega,
  \end{array}
\end{equation}
where $\varphi$ is the unknown function,
$\mathbf{u}(x,y,t)$ the specified velocity field,
and $\nu$ the diffusion coefficient.

As shown in Figure \ref{fig:annulus-subfig1},
 the domain $\Omega$ is an annulus
 with its center at $(1.0, 1.0)$
 and with the radii of the two concentric circles
 as $R_1 = 0.25$ and $R_2 = 0.8$. 
The diffusion coefficient is $\nu = 0.01$.
The velocity $\mathbf{u}$ is a solid body rotation
 given by the stream function
\begin{equation}
  \label{eq:exactSolutionAnnulus}
    \psi(x,y) = 
    \begin{cases}
    0 & \text{if } r\in[0,R_1)\cup(R_2,\infty); \\
    0.2\pi(R_2^2-r^2)    & \text{if }r\in[R_1, R_2], \\
    \end{cases}
\end{equation}
where $r = \sqrt{(x-1)^2 + (y-1)^2}$
is the distance between $(x, y)$ and the center of the annulus.
The initial condition is 
\begin{equation}
  \label{eq:initialConditionAnnulus}
  \begin{array}{l}
  \varphi(\theta) = \frac{1}{2}
  \text{erf}\left(\frac{\pi-6\theta}{6\sqrt{4\nu}}\right)
  + \frac{1}{2}
  \text{erf}\left(\frac{\pi+6\theta}{6\sqrt{4\nu}}\right),
  \end{array}
\end{equation}
where $\theta(x, y)$ is the angle
between the vector $(x-1, y-1)$ and the $x$-axis.

\begin{figure}
  \centering
  \subfigure[$T=0.0$.]{
    \includegraphics[width=0.45\textwidth]{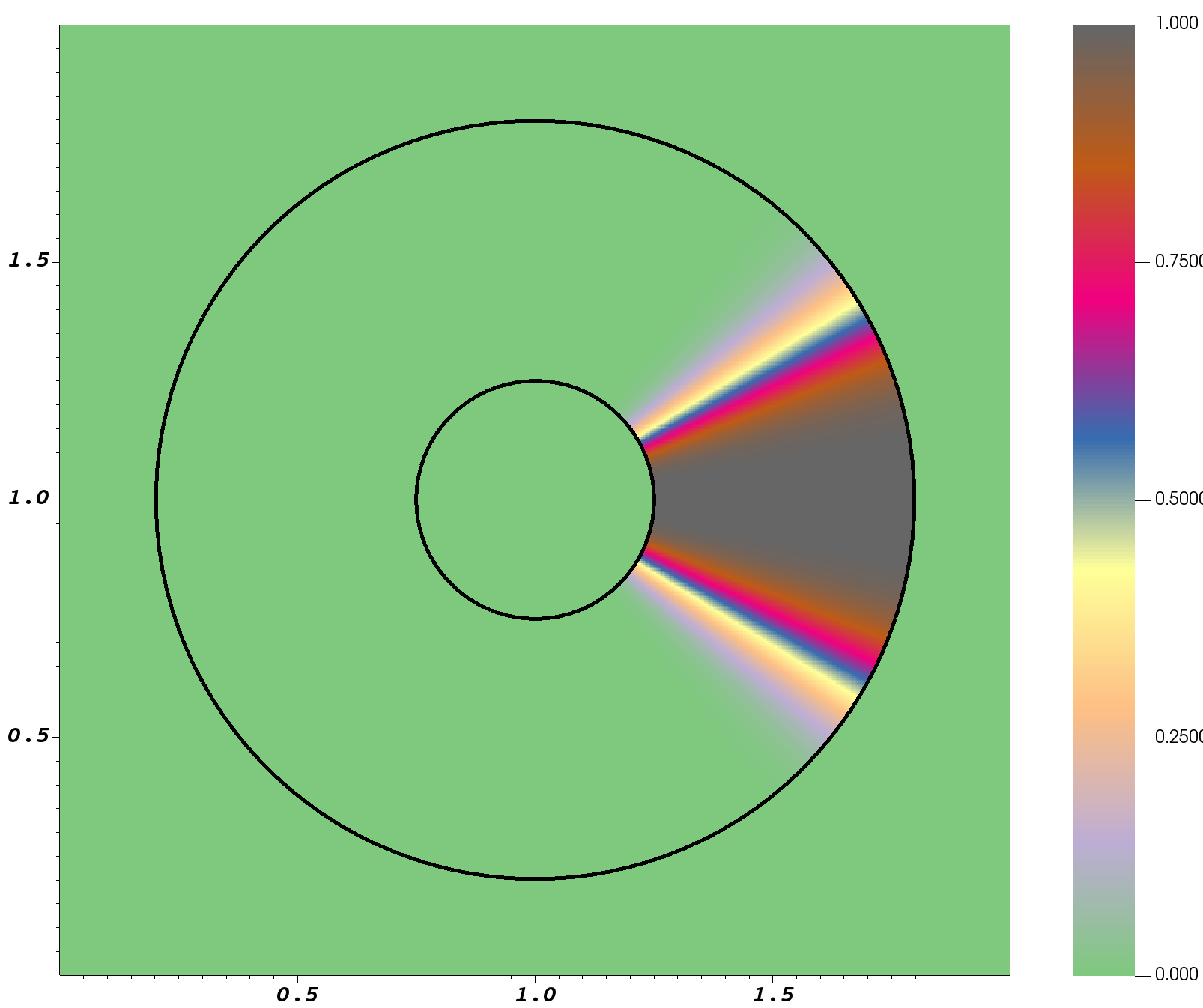}
    \label{fig:annulus-subfig1}
  }
  \hfill
  \subfigure[$T=1.0$.]{
    \includegraphics[width=0.45\textwidth]{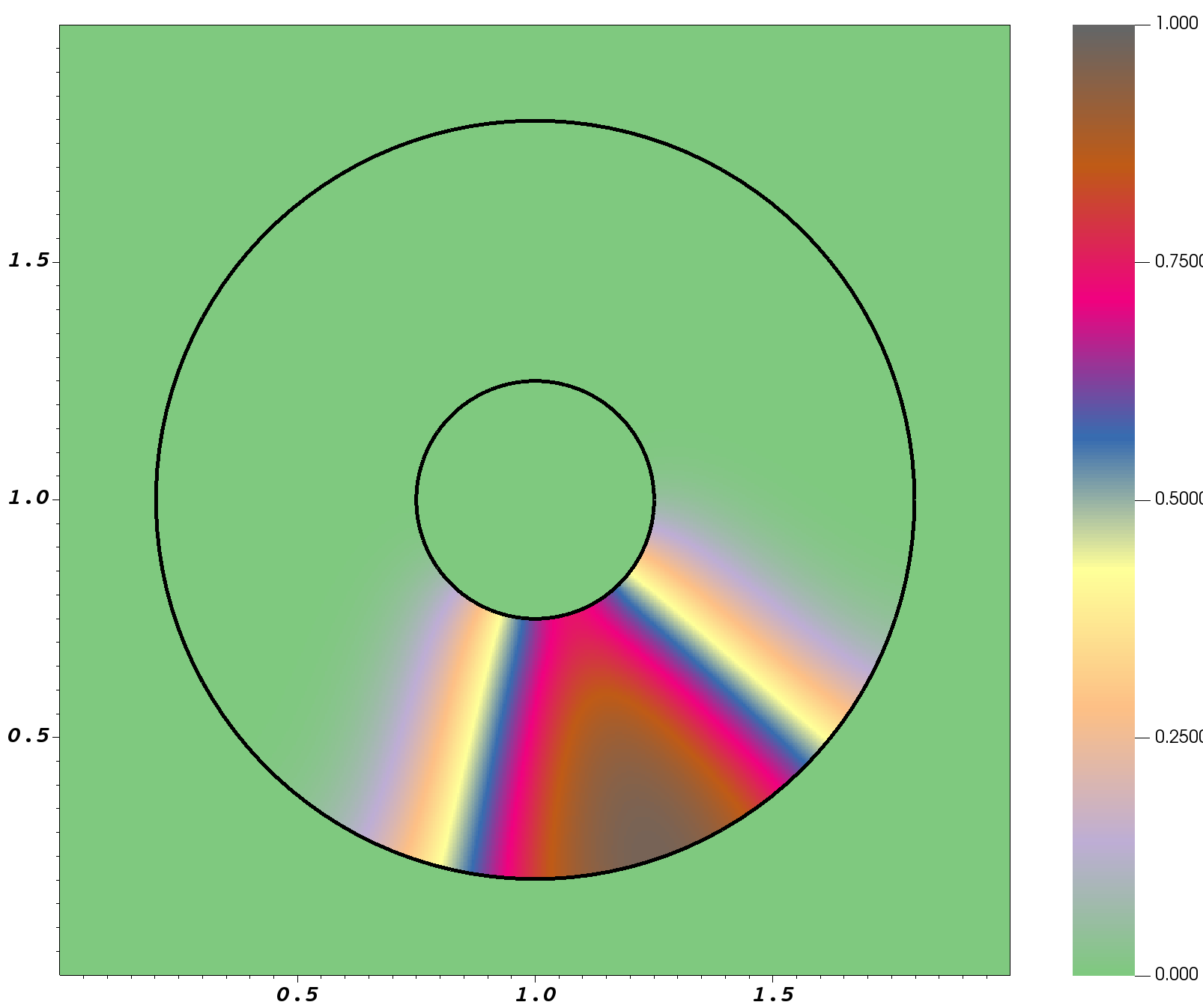}
    \label{fig:annulus-subfig2}
  }
  \hfill
    \subfigure[$T=2.0$.]{
    \includegraphics[width=0.45\textwidth]{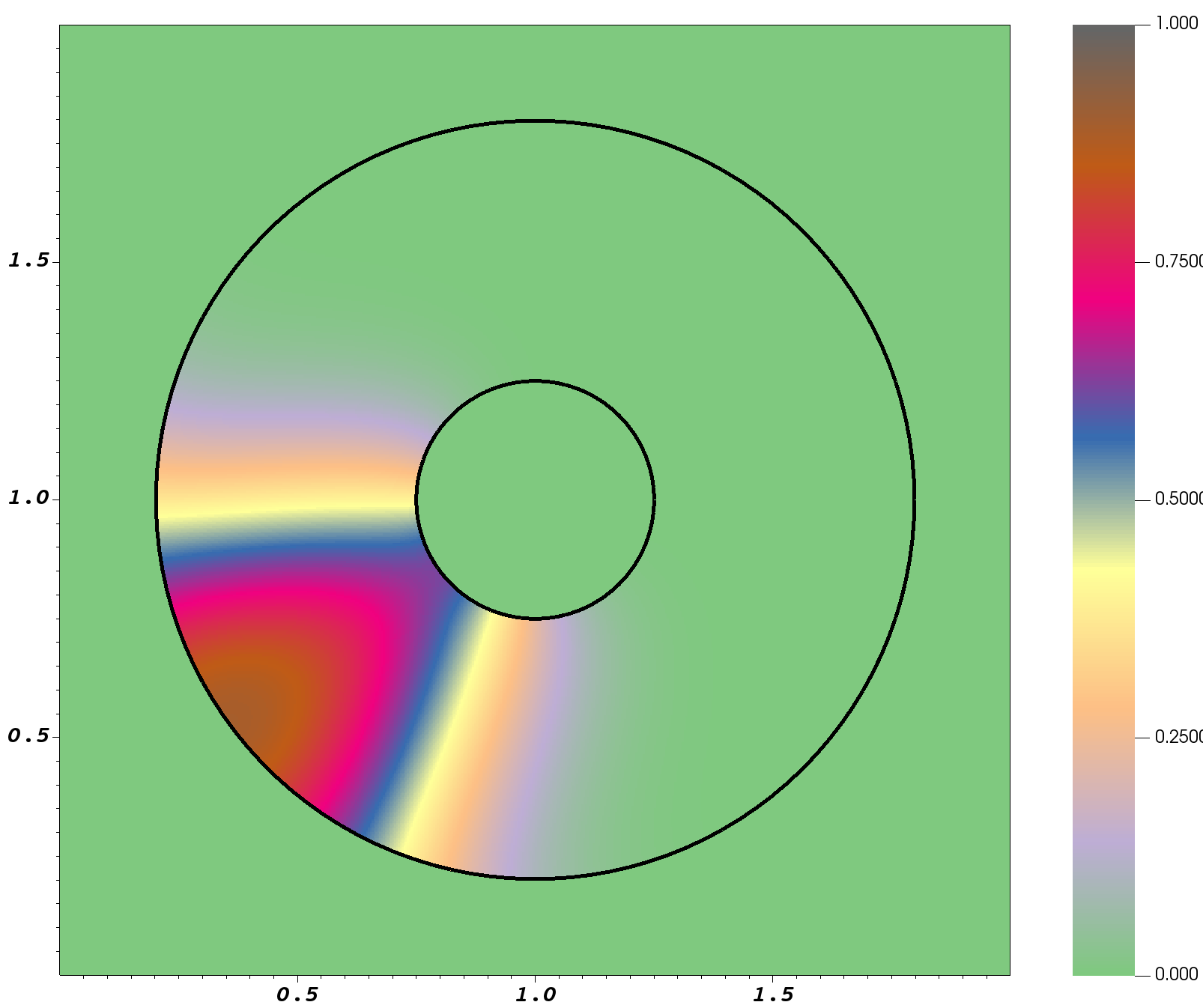}
    \label{fig:annulus-subfig3}
  }
  \hfill
    \subfigure[$T=3.0$.]{
    \includegraphics[width=0.45\textwidth]{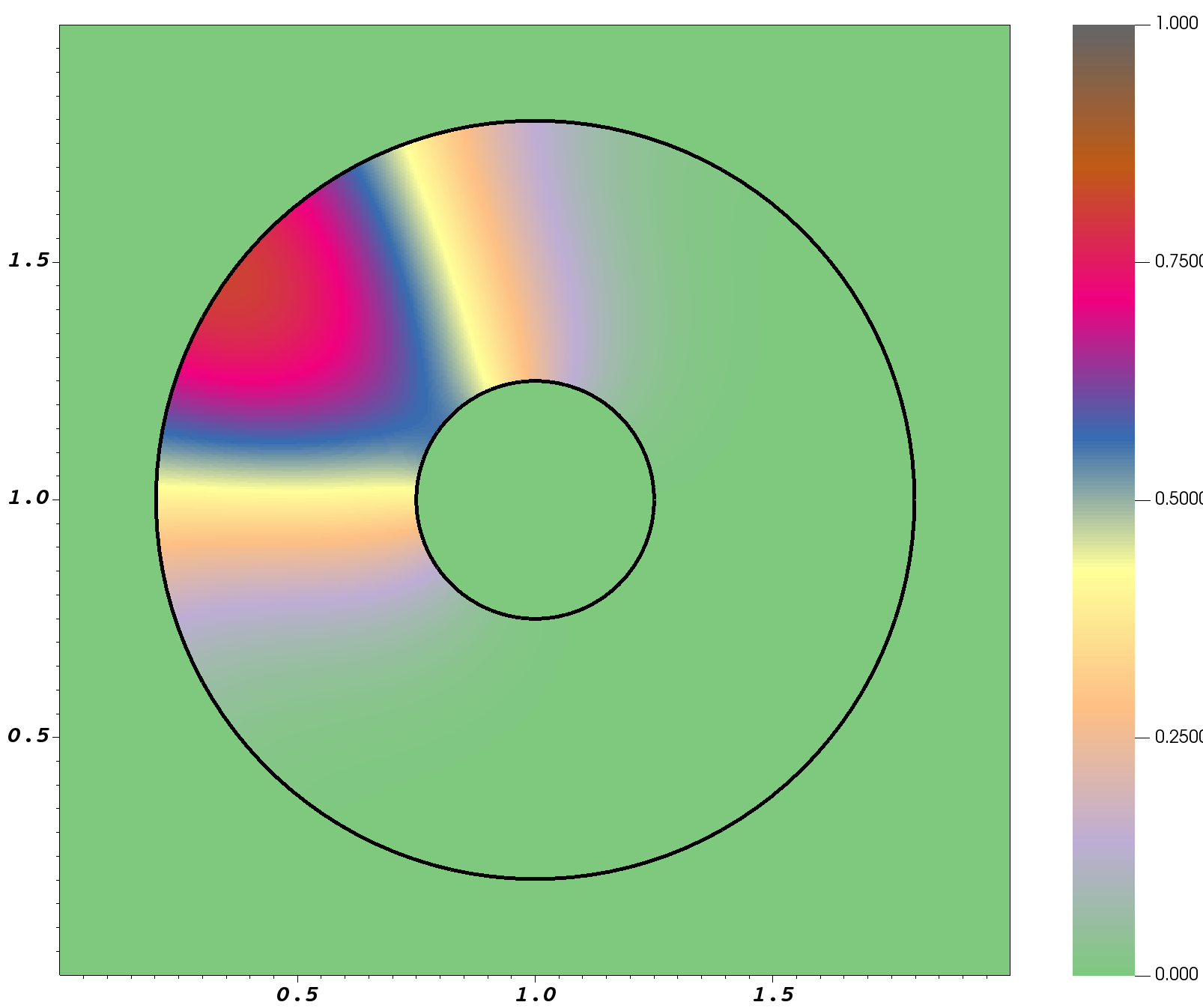}
    \label{fig:annulus-subfig4}
  }
  \hfill
  \subfigure[$T=4.0$.]{
    \includegraphics[width=0.45\textwidth]{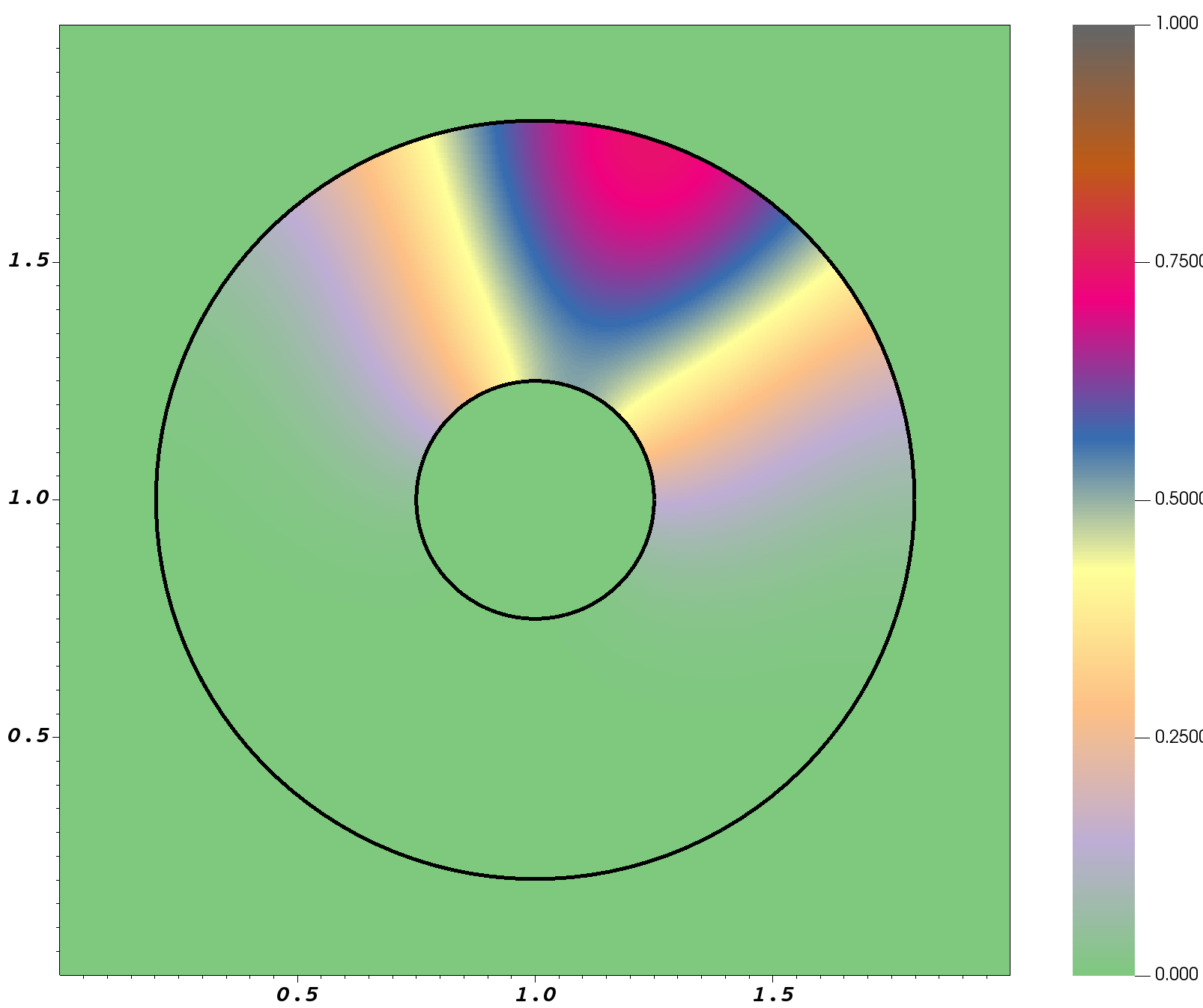}
    \label{fig:annulus-subfig5}
  }
  \hfill
    \subfigure[$T=5.0$.]{
    \includegraphics[width=0.45\textwidth]{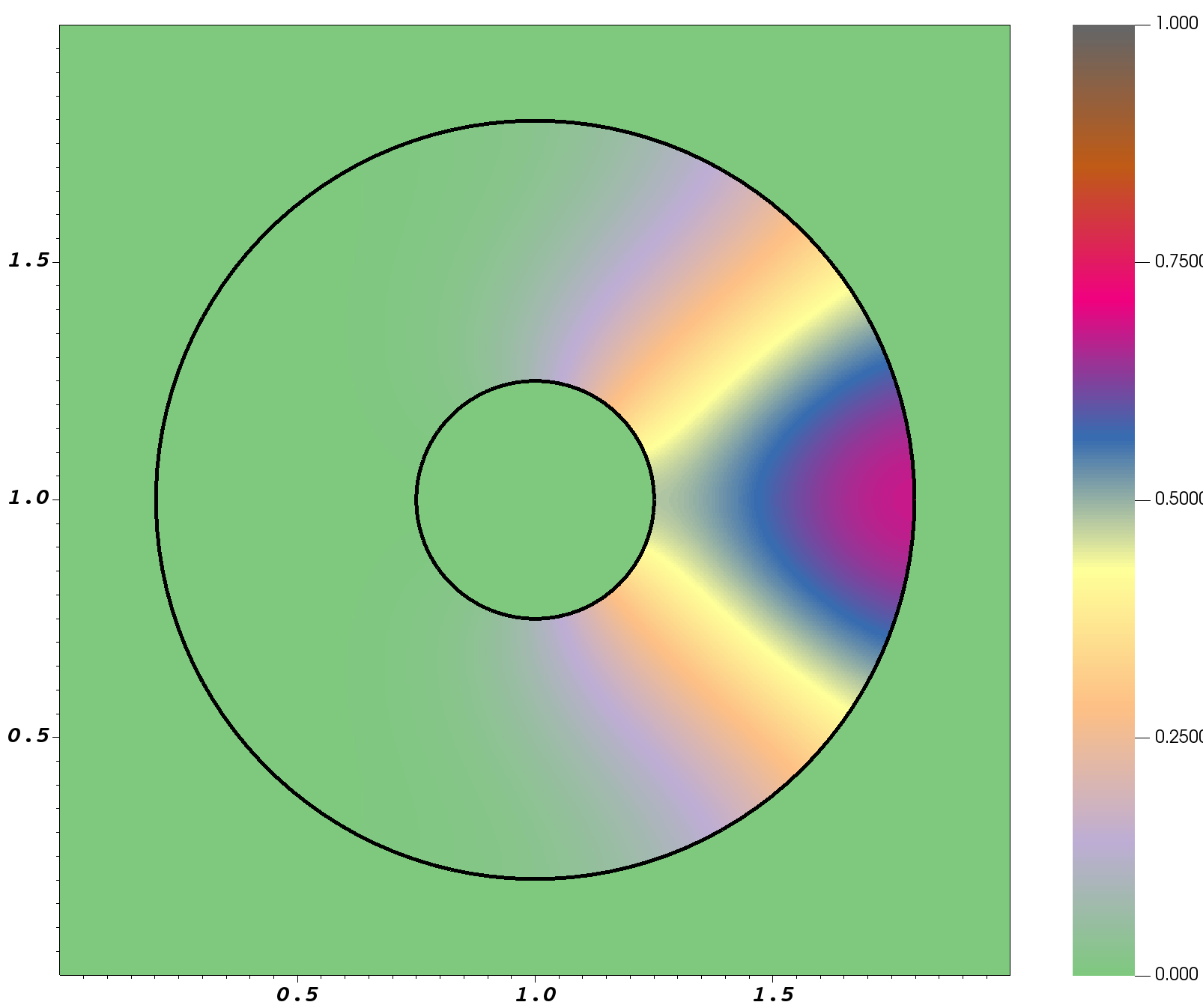}
    \label{fig:annulus-subfig6}
  }
  \caption{Results of the fourth-order PLG-FD
    in solving the advection-diffusion equation
    (\ref{eq:adv-diff-eq}) on an annulus
    with $h = \frac{1}{256}$ and $\Delta t = h$.}
  \label{fig:annulus-result}
\end{figure}

We couple the fourth-order PLG-FD
 to the fourth-order ERK-ESDIRK time integrator 
 by Kennedy and Carpenter \cite[p. 176]{kennedy03:_additive_runge_kutta}, 
 which treats the advection term explicitly
 and the diffusion term implicitly.
As mentioned in (PFT-1,2) in Section \ref{sec:PLG-FD}, 
 we solve a sequence of elliptic equations within each time step
 to advance the solution. 
Figure \ref{fig:annulus-result} presents
 results of our method
 at various time instants over one full rotation.
Since no exact solution is available,
 we use standard Richardson extrapolation
 to calculate the convergence rates at time $T=1.0$,
 which, as shown in Table~\ref{tab:adv-diff-eq},
 are close to four.

\begin{table}
  \caption{Errors and convergence rates of the fourth-order PLG-FD 
    in solving the advection-diffusion equation (\ref{eq:adv-diff-eq})
    with homogeneous Neumann boundary conditions at $T=1.0$. }
  \label{tab:adv-diff-eq}
  \centering
  \begin{tabular}{cccccc}
    \hline
    & $h=\frac{1}{64}-\frac{1}{128}$ & rate
    & $h=\frac{1}{128}-\frac{1}{256}$ & rate
    & $h=\frac{1}{256}-\frac{1}{512}$
    \\
    \hline
    $L^\infty$ & 1.94e-04 & 3.84 & 1.36e-05 & 4.14 & 7.69e-07 \\
    $L^1$ & 4.34e-06 & 3.94 & 2.84e-07 & 4.09 & 1.65e-08 \\
    $L^2$ & 9.56e-06 & 3.93 & 6.30e-07 & 4.13 & 3.57e-08 \\
    \hline
  \end{tabular}
\end{table}

As shown in Figure \ref{fig:select_box}
 and discussed in Section \ref{sec:PLG-FD},
 the feasible set of PLG-FD
 can be set by different strategies
 according to the physics of different operators to be discretized; 
 this provides certain degree of flexiblity for the user
 to cater for the problem at hand. 
In this test, the Laplacian operator is discretized
 by centering the starting point within the feasible set
 and the advection operator by a lop-sided box
 according to the direction of the velocity
 at the starting point; 
 see Figure \ref{fig:plg-example} for the generated poised lattices 
 at an irregular FD node close to the inner circular boundary.
}

\begin{figure}
  \centering
  \subfigure[Discretizing the Laplacian $\Delta \varphi$.]{
    \includegraphics[width=0.45\textwidth]{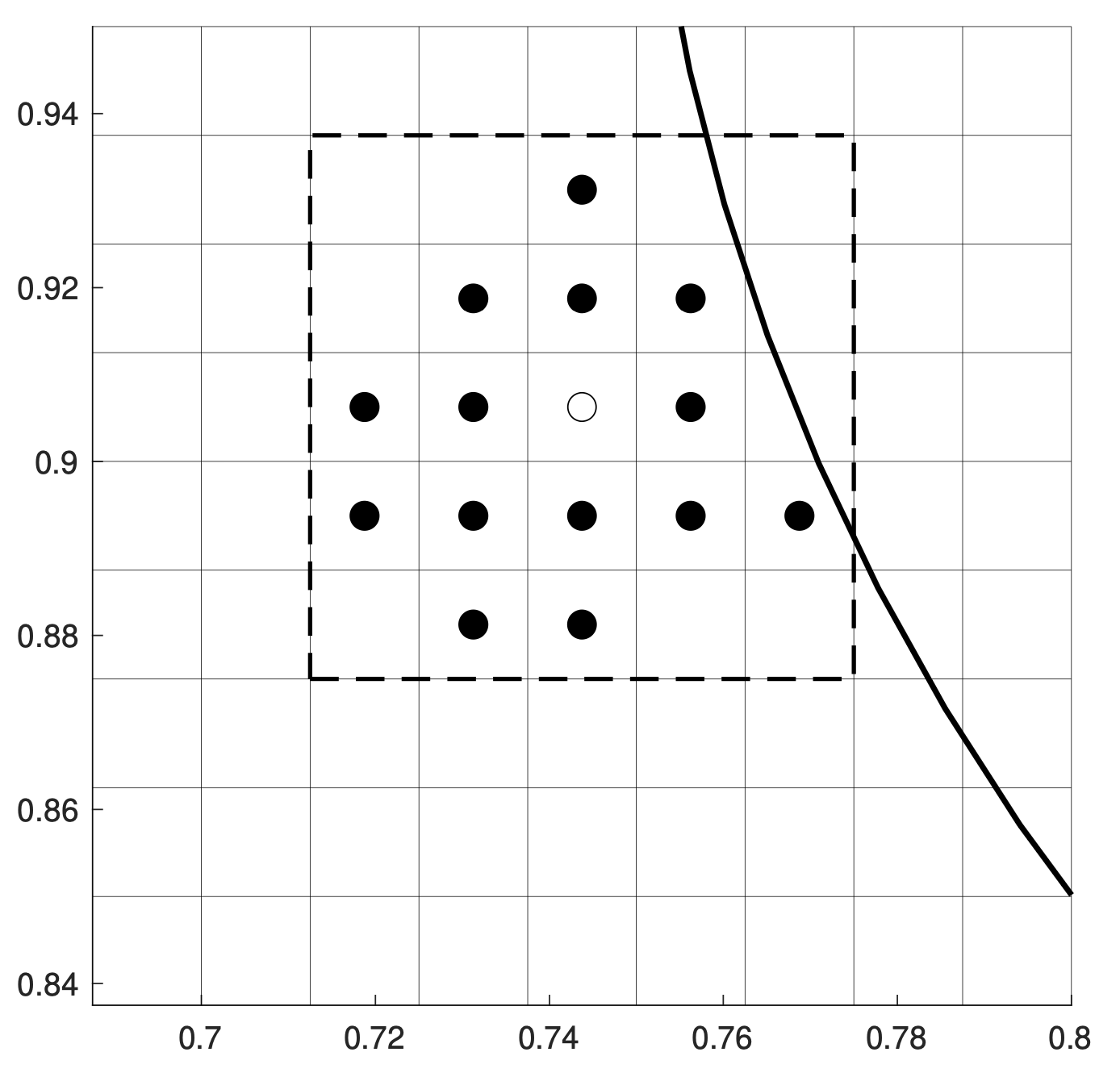}
    \label{fig:plgExample-subfig1}
  }
  \hfill
  \subfigure[Discretizing the advection $\mathbf{u} \cdot \nabla \varphi$.]{
    \includegraphics[width=0.45\textwidth]{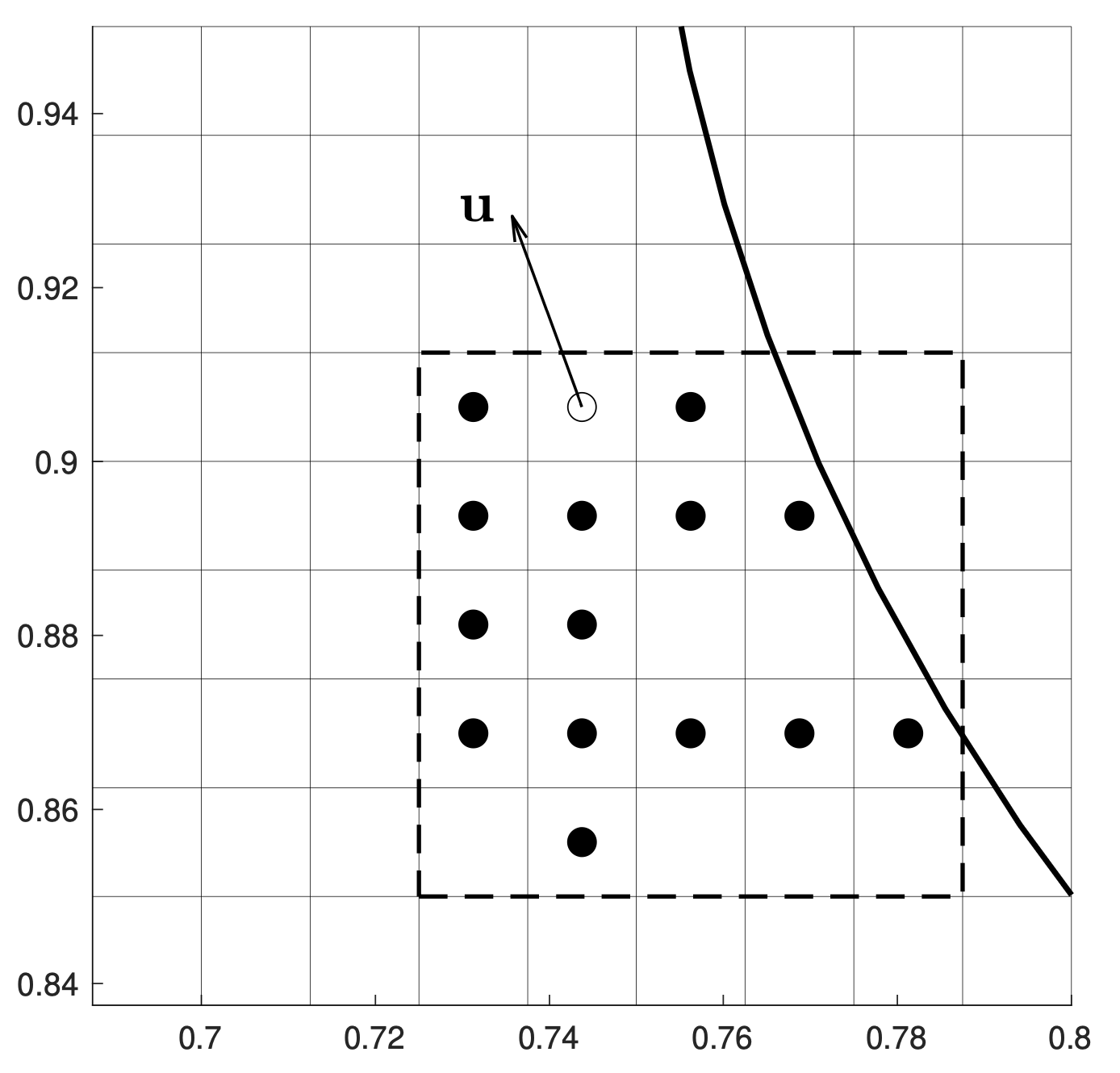}
    \label{fig:plgExample-subfig2}
  }
  \caption{Examples of discretizing different operators
    with different strategies in Section
    \ref{sec:how-choose-feasible};
    see also Figure \ref{fig:select_box}. 
    In both cases,
    the starting point $\mathbf{q}$ is represented by a hollow dot
    near the boundary of the inner circle,
    and the poised lattice for $\mathbf{q}$
    consists of $\mathbf{q}$ and all solid dots.}
  \label{fig:plg-example}
\end{figure}

\section{Conclusion}
\label{sec:conclusion}

For a positive integer $n$ and a given set $K$
 of isolated points in $\mathbb{Z}^{\Dim}$,
 we propose a generic TLG algorithm
 that outputs a subset ${\mathcal T}\subset K$ 
 such that multivariate polynomial interpolation on ${\mathcal T}$
 is unisolvent in $\Pi^{\Dim}_n$.

Based on the TLG discretization, 
 we further develop a fourth-order PLG-FD method
 for numerically solving PDEs
 on irregular domains with Cartesian grids.
Our new method has been shown to be
simple, efficient, and more accurate
than a second-order EB method.
 
Several research prospects follow.
We are currently working on augmenting the PLG-FD method
 to higher convergence rates
\revise{and to time-dependent problems such as 
 the Navier-Stokes equations on irregular domains.}
We also plan to utilize the TLG discretization
 in our finite volume methods \cite{zhang16:_GePUP}
 for simulating incompressible fluids
 with moving boundaries.



 \vspace{0.5cm}
 \noindent \textbf{Acknowledgements. }
  This work was supported by the grant 12272346
   from the National Natural Science Foundation of China.
  \revise{We are grateful to two anonymous referees
    for their insightful comments and suggestions.
  We also thank Junxiang Pan and Lei Pang
  for their comments that lead to an improvement
  of the exposition of the manuscript.
}

\section*{Declarations}

\subsection*{Data Availability}

The datasets generated during and/or analysed during the current study
are available from the corresponding author on reasonable request.

\subsection*{Conﬂict of interest}

The authors declare that they have no conﬂict of interest.

\bibliography{bib/PLG.bib}

\begin{thebibliography}{10}
\providecommand{\url}[1]{{#1}}
\providecommand{\urlprefix}{URL }
\expandafter\ifx\csname urlstyle\endcsname\relax
  \providecommand{\doi}[1]{DOI~\discretionary{}{}{}#1}\else
  \providecommand{\doi}{DOI~\discretionary{}{}{}\begingroup
  \urlstyle{rm}\Url}\fi

\bibitem{carnicer06:_inter}
Carnicer, J.M., Gasca, M., Sauer, T.: Interpolation lattices in several
  variables.
\newblock Numer. Math. \textbf{102}, 559--581 (2006)

\bibitem{carnicer09:_aitken_nevil}
Carnicer, J.M., Gasca, M., Sauer, T.: {Aitken}-{Neville} sets, principal
  lattices and divided differences.
\newblock J. Approx. Theory \textbf{156}(2), 154--172 (2009)

\bibitem{carnicer06:_geomet}
Carnicer, J.M., Godes, C.: Geometric characterization and generalized principal
  lattices.
\newblock J. Approx. Theory \textbf{143}, 2--14 (2006)

\bibitem{chung77:_lagran}
Chung, K.C., Yao, T.H.: On lattices admitting unique {Lagrange} interpolation.
\newblock SIAM J. Numer. Anal. \textbf{14}(4), 735--743 (1977)

\bibitem{boor09:_multiv}
{de Boor}, C.: Multivariate polynomial interpolation: {Aitken-Neville} sets and
  generalized principal lattices.
\newblock J. Approx. Theory \textbf{161}, 411--420 (2009)

\bibitem{devendran17:_cartes_poiss}
Devendran, D., Graves, D.T., Johansen, H., Ligocki, T.: A fourth-order
  {Cartesian} grid embedded boundary method for {Poisson's} equation.
\newblock Commun. Appl. Math. Comput. Sci. \textbf{12}, 51--79 (2017)

\bibitem{dyn14:_multiv}
Dyn, N., Floater, M.S.: Multivariate polynomial interpolation on lower sets.
\newblock J. Approx. Theory \textbf{177}, 34--42 (2014)

\bibitem{errachid20:_rmvpia}
Errachid, M., Essanhaji, A., Messaoudi, A.: {RMVPIA}: a new algorithm for
  computing the {Lagrange} multivariate polynomial interpolation.
\newblock J. Sci. Comput. \textbf{84}, 1507--1534 (2020)

\bibitem{gasca00:_polyn}
Gasca, M., Sauer, T.: Polynomial interpolation in several variables.
\newblock Adv. Comput. Math. \textbf{12}, 377--410 (2000)

\bibitem{gautschi12:_numer_analy}
Gautschi, W.: Numerical Analysis, second edn.
\newblock Birkhauser (2012).
\newblock {ISBN}: 978-0-8176-8258-3

\bibitem{ingram03:_devel_cartes}
Ingram, D.M., Causon, D.M., Mingham, C.G.: Developments in {Cartesian} cut cell
  methods.
\newblock Mathematics and Computers in Simulation \textbf{61}, 561--572 (2003)

\bibitem{jaklic10:_lattic}
Jaklic, G., Kozak, J., Krajnc, M., Vitrih, V., Zagar, E.: Lattices on
  simplicial partitions.
\newblock J. Comput. Appl. Math. \textbf{233}, 1704--1715 (2010)

\bibitem{johansen98:_cartes_grid_embed_bound_method}
Johansen, H., Colella, P.: A {Cartesian} grid embedded boundary method for
  {Poisson}'s equation on irregular domains.
\newblock J. Comput. Phys. \textbf{147}(1), 60--85 (1998)

\bibitem{kennedy03:_additive_runge_kutta}
Kennedy, C.A., Carpenter, M.H.: Additive {Runge–Kutta} schemes for
  convection–diffusion–reaction equations.
\newblock Appl. Numer. Math. \textbf{44}(1), 139--181 (2003)

\bibitem{knuth75:_alpha_beta_pruning}
Knuth, D.E., Moore, R.W.: An analysis of alpha-beta pruning.
\newblock Artif. Intell. \textbf{6}(4), 293--326 (1975)

\bibitem{lee91:_const_lagran}
Lee, S.L., Phillips, G.M.: Construction of lattices for {Lagrange}
  interpolation in projective space.
\newblock Constr. Approx. \textbf{7}, 283--297 (1991)

\bibitem{leveque07:_finit_differ_method_ordin_partial_differ_equat}
LeVeque, R.J.: Finite Difference Methods for Ordinary and Partial Differential
  Equations.
\newblock SIAM, Philadelphia PA (2007)

\bibitem{leveque94:_immer}
Leveque, R.J., Li, Z.: Immersed interface method for elliptic equations with
  discontinuous coefficients and singular sources.
\newblock SIAM J. on Numer. Anal. \textbf{31}(4), 1019--1044 (1994)

\bibitem{li06:_immer_inter_method}
Li, Z., Ito, K.: The Immersed Interface Method: Numerical Solutions of PDEs
  Involving Interfaces and Irregular Domains.
\newblock SIAM (2006)

\bibitem{li18:_numer_solut_differ_equat_introd}
Li, Z., Qiao, Z., Tang, T.: Numerical Solution of Differential Equations:
  Introduction to Finite Difference and Finite Element Methods.
\newblock Cambridge University Press, Cambridge, United Kingdom (2018)

\bibitem{liu03:_ghost}
Liu, T.G., Khoo, B.C., Yeo, K.S.: Ghost fluid method for strong shock impacting
  on material interface.
\newblock J. Comput. Phys. \textbf{190}(2), 651--681 (2003)

\bibitem{mayo84:_fast_solut_of_poiss_and}
Mayo, A.: The fast solution of {P}oisson's and the biharmonic equations on
  irregular regions.
\newblock SIAM J. Numer. Anal. \textbf{21}, 285--299 (1984)

\bibitem{mittal05:_immer_bound_method}
Mittal, R., Iaccarino, G.: Immersed boundary methods.
\newblock Annu. Rev. Fluid Mech. \textbf{37}, 239--261 (2005)

\bibitem{neidinger19:_multiv_newton}
Neidinger, R.D.: Multivariate polynomial interpolation in {Newton} forms.
\newblock SIAM Review \textbf{61}(2), 361--381 (2019)

\bibitem{overton-katz23:_stokes}
Overton-Katz, N., Gao, X., Guzik, S., Antepara, O., Graves, D.T., Johansen, H.:
  A fourth-order embedded boundary finite volume method for the unsteady
  {Stokes} equations with complex geometries.
\newblock SIAM J. Sci. Comput. \textbf{45}(5), A2409--A2430 (2023)

\bibitem{peskin77:_numer}
Peskin, C.S.: Numerical analysis of blood flow in the heart.
\newblock J. Comput. Phys. \textbf{25}, 220--252 (1977)

\bibitem{phillips03:_inter_approx_polyn}
Phillips, G.M.: Interpolation and Approximation by Polynomials.
\newblock Springer (2003)

\bibitem{russell21:_artif_intel}
Russell, S., Norvig, P.: Artificial Intelligence: A Modern Approach, 4th edn.
\newblock Pearson (2021).
\newblock {ISBN}: 978-0134610993

\bibitem{sauer03:_lagran}
Sauer, T.: Lagrange interpolation on subgrids of tensor product grids.
\newblock Math. Comput. \textbf{73}(245), 181--190 (2003)

\bibitem{sauer95:_lagran}
Sauer, T., Xu, Y.: On multivariate {Lagrange} interpolation.
\newblock Math. Comput. \textbf{64}(211), 1147--1170 (1995)

\bibitem{sauer02:_aitken_nevil}
Sauer, T., Xu, Y.: The {Aitken}-{Neville} scheme in several variables.
\newblock In: Approximation Theory X: Abstract and Classical Analysis, pp.
  353--366. Vanderbilt University Press, Nashville (2002)

\bibitem{schonhage61:_fehler_inter}
Sch\"{o}nhage, A.: Fehlerfortpanzung bei interpolation.
\newblock Numer. Math. \textbf{3}, 62--71 (1961)

\bibitem{shannon50:_progr}
Shannon, C.E.: Programming a computer for playing chess.
\newblock Phil. Mag. \textbf{41} (1950)

\bibitem{silver16:_AlphaGo}
Silver, D., Huang, A., Maddison, C.J., Guez, A., Sifre, L., van~den Driessche,
  G., Schrittwieser, J., Antonoglou, I., Panneershelvam, V., Lanctot, M.,
  Dieleman, S., Grewe, D., Nham, J., Kalchbrenner, N., Sutskever, I.,
  Lillicrap, T., Leach, M., Kavukcuoglu, K., Graepel, T., Hassabis, D.:
  Mastering the game of {Go} with deep neural networks and tree search.
\newblock Nature \textbf{529}(7587), 484--489 (2016)

\bibitem{strikwerda89:_finit_differ_schem_partial_differ_equat}
Strikwerda, J.C.: Finite Difference Schemes and Partial Differential Equations.
\newblock Wadsworth \& Brooks, Belmont CA (1989)

\bibitem{trefethen17:_approx_theor_and_approx_pract}
Trefethen, L.N.: Approximation Theory And Approximation Practice.
\newblock SIAM (2017).
\newblock {ISBN}: 978-9386235442

\bibitem{tseng03:_ghost_cell_immer_bound_method}
Tseng, Y.H., Ferziger, J.H.: A ghost-cell immersed boundary method for flow in
  complex geometry.
\newblock J. Comput. Phys. \textbf{192}, 593--623 (2003)

\bibitem{tucker00:_cartes}
Tucker, P.G., Pan, Z.: A {Cartesian} cut cell method for incompressible viscous
  flow.
\newblock Applied Mathematical Modelling \textbf{24}(8-9), 591--606 (2000)

\bibitem{turetskii40}
Turetskii, A.H.: The bounding of polynomials prescribed at equally distributed
  points.
\newblock Proc. Pedag. Inst. Vitebsk \textbf{3}, 117--127 (1940)

\bibitem{werner80:_remar_newton}
Werner, H.: Remarks on {Newton} type multivariate interpolation for subsets of
  grids.
\newblock Comput. pp. 181--191 (1980)

\bibitem{xie20}
Xie, Y., Ying, W.: A fourth-order kernel-free boundary integral method for
  implicitly defined surfaces in three space dimensions.
\newblock J. Comput. Phys. \textbf{415}, 109526 (2020)

\bibitem{xu23:_ghost}
Xu, L., Liu, T.G.: Ghost-fluid-based sharp interface methods for multi-material
  dynamics: a review.
\newblock Commun. Comput. Phys. \textbf{34}(3), 563--612 (2023)

\bibitem{zhang16:_GePUP}
Zhang, Q.: {GePUP}: Generic projection and unconstrained {PPE} for fourth-order
  solutions of the incompressible {Navier-Stokes} equations with no-slip
  boundary conditions.
\newblock J. Sci. Comput. \textbf{67}, 1134--1180 (2016)

\bibitem{zhang20:_boolean}
Zhang, Q., Li, Z.: Boolean algebra of two-dimensional continua with arbitrarily
  complex topology.
\newblock Math. Comput. \textbf{89}, 2333--2364 (2020)

\bibitem{zhang10:_handl_solid_fluid_inter_for_viscous_flows}
Zhang, Q., Liu, P.L.F.: Handling solid-fluid interfaces for viscous flows:
  explicit jump approximation vs. ghost cell approaches.
\newblock J. Comput. Phys. \textbf{229}, 4225--46 (2010)

\bibitem{zhou06:_high_order_match_inter_and}
Zhou, Y.C., Zhao, S., Feig, M., Wei, G.W.: High order matched interface and
  boundary method for elliptic equations with discontinuous coefficients and
  singular sources.
\newblock J. Comput. Phys. \textbf{213}, 1--30 (2006)

\end{thebibliography}
\bibliographystyle{abbrv}

\end{document}